\documentclass[11pt,twoside]{amsart}

\usepackage{latexsym}
\usepackage{amssymb}
\usepackage{vmargin}
\usepackage{amscd}
\usepackage{stmaryrd}
\usepackage{mathrsfs} 
\usepackage[all]{xy}
\usepackage{xr}

\usepackage{hyperref}
\hypersetup{colorlinks,
linkcolor=black,
citecolor=black,
urlcolor=blue}

\setmargins{32mm}{20mm}{14.6cm}{22cm}{1cm}{1cm}{1cm}{1cm}

\newcommand\da{\!\downarrow\!}

\newcommand\la{\leftarrow}
\newcommand\lra{\longrightarrow}

\newcommand\id{\mathrm{id}}

\newcommand\ten{\otimes}

\newcommand\vareps{\varepsilon}
\newcommand\eps{\epsilon}

\newcommand\CCC{\mathrm{CC}}
\newcommand\GT{\mathrm{GT}}
\newcommand\Levi{\mathrm{Levi}}

\renewcommand\H{\mathrm{H}}
\newcommand\z{\mathrm{Z}}
\renewcommand\b{\mathrm{B}}

\newcommand\HH{\mathrm{HH}}

\newcommand\Z{\mathbb{Z}}
\newcommand\Q{\mathbb{Q}}

\newcommand\R{\mathbb{R}}
\newcommand\Cx{\mathbb{C}}

\newcommand\vv{\mathbb{V}}

\newcommand\bA{\mathbb{A}}
\newcommand\bB{\mathbb{B}}

\newcommand\bD{\mathbb{D}}

\newcommand\bG{\mathbb{G}}
\newcommand\bH{\mathbb{H}}

\newcommand\bL{\mathbb{L}}

\newcommand\C{\mathcal{C}}

\newcommand\cD{\mathcal{D}}

\newcommand\cF{\mathcal{F}}

\newcommand\cI{\mathcal{I}}

\newcommand\cP{\mathcal{P}}

\newcommand\cR{\mathcal{R}}
\newcommand\cS{\mathcal{S}}

\renewcommand\O{\mathscr{O}}

\newcommand\sA{\mathscr{A}}

\newcommand\sD{\mathscr{D}}
\newcommand\sE{\mathscr{E}}
\newcommand\sF{\mathscr{F}}
\newcommand\sG{\mathscr{G}}

\newcommand\sL{\mathscr{L}}
\newcommand\sM{\mathscr{M}}
\newcommand\sN{\mathscr{N}}
\newcommand\sO{\mathscr{O}}

\newcommand\fX{\mathfrak{X}}

\renewcommand\L{\Lambda}

\newcommand\g{\mathfrak{g}}

\newcommand\ff{\mathfrak{f}}

\newcommand\fr{\mathfrak{r}}

\renewcommand\hom{\mathscr{H}\!\mathit{om}}
\newcommand\cHom{\mathcal{H}\!\mathit{om}}
\newcommand\cDiff{\mathcal{D}\!\mathit{iff}}
\newcommand\cDer{\mathcal{D}\!\mathit{er}}

\newcommand\CAlg{\mathrm{CAlg}}

\newcommand\Ass{\mathrm{Ass}}
\newcommand\Br{\mathrm{Br}}

\newcommand\Mod{\mathrm{Mod}}

\newcommand\Hom{\mathrm{Hom}}

\newcommand\map{\mathrm{map}}

\newcommand\HHom{\underline{\mathrm{Hom}}}

\newcommand\DDer{\underline{\mathrm{Der}}}

\newcommand\End{\mathrm{End}}

\newcommand\Diff{\mathrm{Diff}}
\newcommand\Iso{\mathrm{Iso}}

\newcommand\cone{\mathrm{cone}}
\newcommand\cocone{\mathrm{cocone}}
\newcommand\dg{\mathrm{dg}}
\newcommand\per{\mathrm{per}}

\newcommand{\dbar}{{\,\mathchar'26\mkern-12mu d}}
\newcommand{\brh}{\llbracket \hbar \rrbracket}
\newcommand{\brhh}{\llbracket \hbar^2 \rrbracket}

\newcommand\At{\mathrm{At}}

\newcommand\Co{\mathrm{Co}}
\newcommand\CoS{\mathrm{CoS}}

\newcommand\Spec{\mathrm{Spec}\,}

\newcommand\Set{\mathrm{Set}}

\newcommand\Aff{\mathrm{Aff}}

\newcommand\Sp{\mathrm{Sp}}

\newcommand\Pol{\mathrm{Pol}}

\newcommand\Lag{\mathrm{Lag}}

\newcommand\Com{\mathrm{Com}}
\newcommand\Comp{\mathrm{Comp}}

\newcommand\nondeg{\mathrm{nondeg}}

\newcommand\ad{\mathrm{ad}}

\newcommand\<{\langle}
\renewcommand\>{\rangle}
\newcommand\Lim{\varprojlim}
\newcommand\LLim{\varinjlim}
\DeclareMathOperator*{\holim}{holim}
\newcommand\ho{\mathrm{ho}\!}

\newcommand\xra{\xrightarrow}
\newcommand\xla{\xleftarrow}

\newcommand\pr{\mathrm{pr}}

\newcommand\bt{\bullet}
\newcommand\by{\times}

\newcommand\mc{\mathrm{MC}}
\newcommand\mmc{\underline{\mathrm{MC}}}

\newcommand\Symm{\mathrm{Symm}}

\newcommand\et{\acute{\mathrm{e}}\mathrm{t}}

\newcommand\an{\mathrm{an}}

\newcommand\Tot{\mathrm{Tot}\,}

\newcommand\pd{\partial}

\newcommand\half{\frac{1}{2}}

\newcommand\gr{\mathrm{gr}}

\newcommand\Fil{\mathrm{Fil}}

\newcommand\Lie{\mathrm{Lie}}

\newcommand\DR{\mathrm{DR}}

\newcommand\op{\mathrm{opp}}

\newcommand\co{\colon\thinspace}

\newcommand\oR{\mathbf{R}}

\newcommand\oL{\mathbf{L}}

\newcommand\uleft\underleftarrow
\newcommand\uline\underline
\newcommand\uright\underrightarrow

\newtheorem{theorem}{Theorem}[section]
\newtheorem{proposition}[theorem]{Proposition}
\newtheorem{corollary}[theorem]{Corollary}
\newtheorem{conjecture}[theorem]{Conjecture}
\newtheorem{lemma}[theorem]{Lemma}
\newtheorem*{theorem*}{Theorem}
\newtheorem*{proposition*}{Proposition}
\newtheorem*{corollary*}{Corollary}
\newtheorem*{lemma*}{Lemma}
\newtheorem*{conjecture*}{Conjecture}

\theoremstyle{definition}
\newtheorem{definition}[theorem]{Definition}
\newtheorem*{definition*}{Definition}

\newtheorem*{notation*}{Notation}

\theoremstyle{remark}
\newtheorem{example}[theorem]{Example}
\newtheorem{examples}[theorem]{Examples}
\newtheorem{remark}[theorem]{Remark}
\newtheorem{remarks}[theorem]{Remarks}

\newtheorem{assumption}[theorem]{Assumption}

\newtheorem*{example*}{Example}
\newtheorem*{examples*}{Examples}
\newtheorem*{remark*}{Remark}
\newtheorem*{remarks*}{Remarks}
\newtheorem*{exercise*}{Exercise}
\newtheorem*{property*}{Property}
\newtheorem*{properties*}{Properties}

\begin{document}

\begin{abstract}
We investigate quantisations of line bundles $\sL$ on derived Lagrangians $X$ over $0$-shifted symplectic  derived Artin $N$-stacks $Y$. In our derived setting, a deformation quantisation  consists of  a curved $A_{\infty}$ deformation of the structure sheaf $\sO_{Y}$, equipped with a curved  $A_{\infty}$ morphism to the ring of   differential operators on $\sL$; for line bundles on smooth Lagrangian subvarieties of smooth symplectic algebraic varieties, this simplifies to deforming  $(\sL, \sO_{Y})$ to a DQ module over a DQ algebroid. 

For each choice of formality isomorphism between the $E_2$ and $P_2$ operads,  we construct a map from the space of non-degenerate quantisations to power series with coefficients in relative cohomology groups of the respective  de Rham complexes.
 When $\sL$ is a square root of the dualising line bundle, this leads to an equivalence between even power series and  certain anti-involutive  quantisations, ensuring that the deformation quantisations always exist for such line bundles. 
This gives rise to a dg category of algebraic Lagrangians, 
an algebraic Fukaya category 
of the form envisaged by Behrend and Fantechi.
We also sketch a  generalisation of these quantisation results to Lagrangians on  higher $n$-shifted symplectic derived stacks.
\end{abstract}

\title{Quantisation of derived Lagrangians}
\author{J.P.Pridham}
\thanks{This work was supported by  the Engineering and Physical Sciences Research Council [grant number EP/I004130/2].}

\maketitle

\section*{Introduction}

A major source of motivation for the study of shifted symplectic and Poisson structures in derived geometry is the desire to develop and understand quantisations. For  $n >0$, existence of quantisations of $n$-shifted Poisson structures is automatic, following from the formality equivalence $E_{n+1}\simeq P_{n+1}$ of operads. Quantisations of positively shifted symplectic structures thus follow immediately from the equivalence in \cite{poisson, CPTVV} between symplectic and non-degenerate Poisson structures. For lower values of $n$, quantisation is a much harder problem to tackle, or even formulate,  but \cite{DQnonneg, DQvanish} established the existence of quantisations for $0$-shifted and $(-1)$-shifted symplectic structures on derived Artin $N$-stacks. 

A deformation quantisation of a symplectic structure on a geometric object $Y$
is a non-commutative deformation, parametrised by power series in $\hbar$, of the
functions $\sO_Y$ on $Y$, such that the classical limit $\hbar \to 0$ recovers the Poisson
bracket associated to the symplectic structure.
Classically, this means looking at associative deformations $\star_{\hbar}$ of the
multiplication on $\sO_Y$, with the Poisson bracket then given by $\{a,b\}:=\lim_{\hbar \to 0} \frac{a\star_{\hbar}  b -b \star_{\hbar}a}{\hbar}$. In derived geometry, $\sO_Y$ is homologically enriched and
symplectic structures can be shifted in the sense that they have non-zero
homological degree, so their quantisations have to be formulated in terms of
more exotic algebraic structures. A deformation quantisation of a Lagrangian
structure $X \to Y$ then involves a deformation of $\sO_Y$ acting on a deformation of
$\sO_X$ in a suitable sense, such that the classical limit recovers the Lagrangian
structure.

The purpose of this paper is simultaneously to generalise the results of \cite{DQnonneg, DQvanish} by formulating and studying quantisations of Lagrangian morphisms  $(X, \lambda) \to (Y,\omega)$ over $0$-shifted symplectic derived stacks $(Y, \omega)$, in the sense of \cite{PTVV}; then \cite{DQnonneg} corresponds to the case where $X$ is empty, and \cite{DQvanish} to the case where $Y$ is a point, forcing the the Lagrangian structure on $X$ to be $(-1)$-shifted symplectic. 

Based on the principle that $n$-shifted quantisations broadly correspond to $E_{n+1}$-algebras, a deformation quantisation of an $n$-shifted Lagrangian structure on $(X \to Y)$ should roughly consist of an $E_{n+1}$-algebra deformation $\tilde{\sO}_{Y}$ of the structure sheaf $\sO_Y$, together with an $\sO_{Y}$-module $\tilde{\sO}_{X}$ in $E_n$-algebras deforming the structure sheaf $\sO_X$.
In the $0$-shifted setting, this would mean seeking  an associative deformation $\tilde{\sO}_{Y}$ of the structure sheaf $\sO_Y$, together with an $\tilde{\sO}_{Y}$-module deformation $\tilde{\sO}_X$ of the structure sheaf $\sO_X$. More precisely, the $E_k$-algebra deformations should  be $BD_k$-algebras; in the $0$-shifted setting, this means that $ \tilde{\sO}_{Y}$ satisfies an almost commutativity condition while  the deformation $\tilde{\sO}_X$  is given by differential operators with orders constrained as in \cite{kravchenko}.

% Introduction
% Page 2, first paragraph. This is unfortunately rather cryptic. I would really
% appreciate a more detailed discussion of shifted Poisson structures and the special
% status of the n = 0 case here.
% Why must the deformation of a line bundle be given by differential operators?
% In [Pri7] that is how E 0 -quantizations of line bundles with (-1)-shifted Poisson
% structure are defined, and some motivation can be found in Remark 1.13 of loc.
% cit. At the least this should be referenced. But more discussion is warrented. For
% example, these E 0 quanizations look very different from algebras over the E 0 -operad
% as in [CG].
The notion of deformation quantisation can be weakened by allowing curvature, and also generalised by deforming $B^{k+1}\bG_m$-torsors in place of the structure sheaf. For  $0$-shifted Lagrangians, both become necessary to ensure the existence of algebraic quantisations. Even on (underived) non-singular varieties,  curvature in quantisations  manifests itself in the  form of DQ algebroids and DQ modules as in \cite{kontsevichDQAlgVar,DAgnoloSchapira}. On a Lagrangian $X$, we look to quantise line bundles in place of the structure sheaf $\sO_X$, and as in the extreme case of  $(-1)$-shifted symplectic structures  considered in \cite{DQvanish}, it is line bundles with a self-duality property (sometimes known as orientation data or spin structures) for which the existence of quantisations is guaranteed.
For our precise formulations of quantised co-isotropic structures, see Definition \ref{QPdef} and Remark \ref{curvedrmk} for local descriptions, Examples \ref{curvedex} for special cases, and 
%%that's where we have $BD_0$-algebras.
Definition \ref{QPdefglobal} and Remark \ref{algdrmk}  globally.

%%Definition \ref{QPdef}, Rmk \ref{curvedrmk} for local descriptions. %%Can think of {QPdef} as quantised co-isotropic structures, but since quantisations admit a natural direct formulation, we are not primarily concerned with co-isotropic structures themselves , although they do arise as an intermediate stage in our construction $Q\cP(A,B;0)/G^1$. %%give ref

Our main result is Theorem \ref{quantpropsd}, which implies that deformation quantisations exist for any $0$-shifted Lagrangian morphism $X \to Y$ of derived Artin stacks, and any line bundle  $\sL$ which is a square root of the dualising complex $K_{X}$. %The quantisations constructed here are curved $A_{\infty}$ (i.e. homotopy associative) morphisms $\tilde{\sO}_{Y} \to \sD_{X}(\sL)\brh$ to the ring of differential operators on $\sL$. They are moreover self-dual in the sense that the diagram admits an involution to the diagram $\tilde{\sO}_{Y}^{\op} \to \sD_{X}(\sL)^{\op}\brh$ of opposite algebras, semilinear with respect to the transformation $\hbar \mapsto -\hbar$; this generalises the property $b\star_{\hbar}a =a\star_{-\hbar}b$ which is often satisfied by star products. 
The quantisations thus constructed have a self-duality property, generalising the property $b\star_{\hbar}a =a\star_{-\hbar}b$  often satisfied by star products.   
In fact, Theorem \ref{quantpropsd} gives a complete classification of self-dual quantisations of a given Lagrangian structure, parametrising them in terms of de Rham cohomology as a torsor for the group
$
 \hbar^2 \H^1(\cone(\DR(Y) \to \DR(X)))\brhh
$,
---  when working over $\Cx$, this  is just the relative cohomology group 
\[
 \hbar^2\H^1(Y(\Cx),X(\Cx);\Cx\brhh)
\]
 of the associated topological spaces $X(\Cx), Y(\Cx)$ (or to be precise, simplicial spaces when $X,Y$ are stacks), with coefficients in $\Cx\brhh$. 
   
  When $Y$ is a smooth variety, and $X$ a smooth Lagrangian subvariety, this theorem and Proposition \ref{prop3} recover the classification in \cite{BaranovskyGinzburgKaledinPecharich} of quantisations of the pair $(Y,X)$, as explained in Remark \ref{cfBGKP}, but our
derived Lagrangians $X$ can also  be derived enhancements of singular schemes or stacks even when $Y$ is a smooth variety.

Much study of derived quantisation questions has been motivated by the desire to associate  a dg category to $0$-shifted symplectic derived schemes $(Y, \omega)$ in algebraic and complex analytic settings with similar properties to the Fukaya category, as outlined in \cite{BehrendFantechiIntersections, joyceFukaya, BBDJS}. The driving philosophy is that the derived intersection (or rather, homotopy fibre product) of $0$-shifted Lagrangians is a $(-1)$-shifted symplectic derived stack, so carries a perverse sheaf of vanishing cycles, and that there should be a dg category whose $\Hom$-complexes are given by appropriate shifts of derived global sections of these sheaves. The motivation  is explained in some detail in \cite[Remark 6.15]{BBDJS}, with vanishing cycles resembling an analogue of Lagrangian Floer cohomology. However, there are serious difficulties in trying to upgrade these complexes to a dg category. Our quantisation results allow us to solve this problem in 
\S \ref{fukayasn}
by attacking it from the opposite direction. Fixing a suitable quantisation $\tilde{\sO}_Y$ of $(Y, \omega)$, Theorem \ref{quantpropsd} guarantees that compatible quantisations $( \tilde{\sO}_Y,\tilde{\sL})$ exist for oriented Lagrangians $(X,\sL)$ over $Y$, leading to a natural dg category of the associated $\tilde{\sO}_Y$-modules in Definition \ref{fukayadef}. We then show in Corollary \ref{cfvanish2} that the  $\Hom$-complexes  indeed come from vanishing cycles after inverting $\hbar$.

Our approach to proving Theorem \ref{quantpropsd} will be familiar from \cite{poisson,DQvanish, DQnonneg}. For each quantisation $\Delta$, we define a map $\mu$ from generalised Lagrangian structures, defined in terms of power series in de Rham cohomology $\cone(\DR(Y) \to \DR(X))$, to a quantised form of relative Poisson cohomology, giving a filtered quasi-isomorphism when $\Delta$ is non-degenerate. To each non-degenerate quantised co-isotropic structure $\Delta$, there is an associated element $\hbar^2 \frac{\pd \Delta}{\pd \hbar}$, and hence a power series $\mu^{-1}(\hbar^2 \frac{\pd \Delta}{\pd \hbar})$ whose constant term is a Lagrangian structure.  Obstruction calculus shows that this induces an equivalence between self-dual quantisations and even power series.  

Our main new technical ingredient in this paper is in defining the map $\mu$, where we consider the natural morphism
$
 \CCC^{\bt}(\sO_Y) \to \CCC^{\bt}(\sD_{X/Y}(\sL))
$
%%problem here that $\CCC$ is sometimes used for cyclic (messy, not that clearcut).
of $E_2$-algebras  induced by the action of the Hochschild complex $\CCC^{\bt}(\sO_Y)$  on the ring of differential operators $\sD_{X/Y}(\sL)$. Via formality, we may regard these $E_2$-algebras as $P_2$-algebras, and then each quantisation $(\tilde{\sO}_Y,\tilde{\sL})$ defines a commutative diagram  of the form
\[
 \begin{CD}
 \DR(Y) @>>> \DR(X)\\
 @VVV      @VVV\\
  \CCC^{\bt}(\tilde{\sO}_Y) @>>> \CCC^{\bt}(\sD_{X/Y}(\tilde{\sL}))
 \end{CD}
\] 
from the de Rham complexes, with the left-hand side recovering the compatibility map from \cite{DQnonneg}. The morphism $\mu(-,\Delta)$ is then given by composing with the natural map $\CCC^{\bt}(\sD_{X/Y}(\tilde{\sL})) \to \sD_{X/Y}(\tilde{\sL})$ and taking cones to give a map from de Rham cohomology to a form of quantised relative Poisson cohomology.

The structure of the paper is as follows. 

In Section \ref{centresn}, we establish some technical background results on Hochschild complexes. Under the well-known principle (see for instance \cite{GM,Kon,Man,ddt1}) that deformation problems in characteristic $0$ are governed by differential graded Lie algebras (DGLAs), the DGLAs in \cite{DQnonneg,DQvanish} governing $0$-shifted and $(-1)$-shifted quantisations were constructed from Hochschild complexes and rings of differential operators, respectively. Our perspective for quantisations of $0$-shifted co-isotropic structures on a morphism $X \to Y$ is that the governing  DGLA comes from Hochschild complex $\CCC^{\bt}(\sO_Y)$  acting on $\sD_{X/Y}(\sL)$ via the quasi-isomorphism $\sD_{X/Y} \to \CCC^{\bt}(\sO_Y, \sD_X)$. 

When equipped with a PBW filtration degenerating to Poisson cohomology, we first show that Hochschild complexes  of almost commutative algebras become  almost commutative brace algebras in a suitable sense (\S \ref{bracesn}). This allows us to construct suitable semidirect products of Hochschild complexes from morphisms of almost commutative algebras in \S \ref{semidirectsn}.   Section \ref{affinesn} then applies these constructions to Hochschild complexes acting on rings of differential operators, allowing us to construct a form of quantised relative Poisson cohomology (Definition \ref{TQpoldef0}),  leading to a space $Q\cP(A,B;0)$ of quantisations associated to a morphism $A \to B$ of commutative bidifferential bigraded algebras (i.e. a map $\Spec B \to \Spec A$ of stacky derived affines in the sense of \cite{poisson}), and more generally a space  $Q\cP(A,M;0)$ for each line bundle $M$ over $B$ (Definition \ref{QPdef}).

Section \ref{compatsn} contains the key technical construction  (Definition \ref{mudef}) of the compatibility map $\mu$ between generalised Lagrangians and quantised co-isotropic structures. The main results of this section are Proposition \ref{QcompatP1}, giving a map from non-degenerate quantisations to generalised Lagrangians, and Proposition \ref{compatcor2}, which gives an equivalence between Lagrangians and non-degenerate co-isotropic structures. Proposition \ref{quantprop} then shows that the obstruction to quantising a co-isotropic structure is first order.

In Section \ref{stacksn}, these constructions are globalised via the method introduced in \cite{poisson}. \S \ref{sdsn} then introduces the notion of self-duality, enabling us to eliminate the first order obstruction and thus lead to Theorem \ref{quantpropsd}, the main comparison result. In \S \ref{higherrmk}, we then explain how the methods and results of the paper should adapt to Lagrangians on positively shifted symplectic stacks. 

Section \ref{fukayasn} describes algebraic (and complex analytic) analogues of the Fukaya category based on self-dual quantisations of line bundles on derived Lagrangians, and establishes a few key properties. The main definition is given in Definition \ref{fukayadef} and the relation with vanishing cycles in Proposition \ref{cfvanish1} and Corollary \ref{cfvanish2}, with Proposition \ref{Lagcorrprop} establishing functoriality with respect to Lagrangian correspondences. Many of these structural results rely on additivity properties established in \S \ref{QIntHom} investigating the interaction of quantisation with intersection and $\Hom$, which may be of independent interest.  

I would like to thank the anonymous referee for helpful comments.

\tableofcontents

\subsubsection*{Notation and terminology}

Throughout the paper, we will usually denote chain differentials by $\delta$. The graded vector space underlying a chain (resp. cochain) complex $V$ is denoted by $V_{\#}$ (resp. $V^{\#}$). Since we often have to work with chain and cochain structures separately, we denote shifts as subscripts and superscripts, respectively, so $(V_{[i]})_n:= V_{i+n}$ and $(V^{[i]})^n:= V^{i+n}$.

Given an associative algebra $A$ in chain complexes, and  $A$-modules $M,N$ in chain complexes, we write $\HHom_A(M,N)$ for the cochain complex given by
\[
 \HHom_A(M,N)^i= \Hom_{A_{\#}}(M_{\#[i]},N_{\#}),
\]
with differential $ f\mapsto \delta_N \circ f \pm f \circ \delta_M$.

We refer to associative algebras in chain complexes as DGAAs (i.e. differential graded associative algebra), and commutative algebras in chain complexes as CDGAs (i.e. commutative differential graded algebras); these are assumed unital unless stated otherwise.
 We will also refer to coassociative coalgebras in chain  complexes over a CDGA $R$ as DGACs over $R$; these are co-unital  unless stated otherwise. 
From Section \ref{affinesn} onwards, we will be working with double complexes $V^{\bt}_{\bt}$ combining both chain and cochain gradings, where the chain and cochain differentials are denoted by $\delta$ and $\pd$ respectively. We refer to unital commutative (resp. associative) algebras in double complexes as stacky CDGAs (resp. stacky DGAAs), regarding the cochain differential $\pd$ as stacky structure and the chain differential $\delta$ as derived structure.

\begin{definition}
 Given a chain cochain complex $V$, define the cochain complex $\hat{\Tot} V \subset \Tot^{\Pi}V$ as a subset of the product total complex by
\[
 %(\hat{\Tot} V)^m := (\prod_{i \le 0} V^i_{i-m}) \oplus (\bigoplus_{i>0}   V^i_{i-m}),
(\hat{\Tot} V)^m := (\bigoplus_{i < 0} V^i_{i-m}) \oplus (\prod_{i\ge 0}   V^i_{i-m})
\]
with differential $\pd \pm \delta$. This is sometimes referred to as the Tate realisation.
\end{definition}

Here and elsewhere, we use the symbol $\pm$ to denote the  sign in the total complex of a double complex, or in induced constructions such as tensor powers of, and monomial operations on, chain complexes, noting that internal tensor products are total complexes of external tensor products.  The sign is determined by the property that $\pm$ takes the value $+$ when all inputs have degree $0$.   The symbol $\mp$ then denotes the opposite sign. 
%% $d \pm \pd \pm \delta$ parsed as $(d \pm \pd) \pm \delta$, but we don't seem to do that here.

\begin{definition}
 Given a stacky DGAA $A$ and $A$-modules $M,N$ in chain cochain complexes, we define  internal $\Hom$s
$\cHom_A(M,N)$  by
\[
 \cHom_A(M,N)^i_j=  \Hom_{A^{\#}_{\#}}(M^{\#}_{\#},N^{\#[i]}_{\#[j]}),
\]
with differentials  $\pd f:= \pd_N \circ f \pm f \circ \pd_M$ and  $\delta f:= \delta_N \circ f \pm f \circ \delta_M$,
where $V^{\#}_{\#}$ denotes the bigraded vector space underlying a chain cochain complex $V$. 

We then define the  $\Hom$ complex $\hat{\HHom}_A(M,N)$ by
\[
 \hat{\HHom}_A(M,N):= \hat{\Tot} \cHom_A(M,N).
%\HHom_A(M,N):= \Tot^{\Pi} \tau^{\le 0}\cHom_A(M,N)
\]
\end{definition}
Note that  there is a multiplication $\hat{\HHom}_A(M,N)\ten \hat{\HHom}_A(N,P)\to \hat{\HHom}_A(M,P)$; beware that the same is not true for the product total complexes $\Tot^{\Pi} \cHom_A(M,N)$ in general.

When we need to compare chain and cochain complexes, we  make use of the equivalence  $u$ from chain complexes to cochain complexes given by $(uV)^i := V_{-i}$, and refer to this as rewriting the chain complex as a cochain complex (or vice versa). On suspensions, this has the effect that $u(V_{[n]}) = (uV)^{[-n]}$.

We will denote symmetric and cosymmetric powers by $S^p_B(M)=\Symm^p_B(M):= (M^{\ten_B p})_{\Sigma_p} $ and $\CoS_B^p(M) =\Co\Symm^p_B(M):= (M^{\ten_B p})^{\Sigma_p}$, respectively given by co-invariants and invariants of the symmetric group action. We also write $\Symm_B(M) = \bigoplus_{p \ge 0}S_B^p(M)$ and $\Co\Symm_B(M) = \bigoplus_{p \ge 0}\CoS_B^p(M)$.

\section{The centre of an almost commutative algebra}\label{centresn}

The purpose of this section is to show that the Hochschild complex of an almost commutative algebra is almost commutative as a brace algebra, and to study the resulting  almost commutative brace algebra constructions.
The primary motivation is to ensure that formality equivalences $E_2 \simeq P_2$ then turn these Hochschild complexes into filtered $P_2$-algebras (i.e. Gerstenhaber algebras) for which the Lie bracket has weight $-1$.

\subsection{Almost commutative algebras}

\subsubsection{Homological algebra of complete filtrations}\label{filtrnsn}

We now introduce a formalism for working with complete filtered complexes. Although we make little explicit use of these characterisations in the rest of the paper, they  implicitly feature in the reasoning for  complete filtered functors  to have given properties.

\begin{definition}\label{xidef}
 Given a vector space $V$ with a decreasing filtration $F$, the  Rees module $\xi(V,F)$ is given by 
$\xi(V,F):= \bigoplus_p F^pV \hbar^{-p} \subset V[\hbar, \hbar^{-1}]$. This has the structure of a $\bG_m$-equivariant (i.e. graded) $\Z[\hbar]$-module, setting $\hbar$ to be of weight $-1$ for the $\bG_m$-action.
\end{definition}

 The functor $\xi$ gives an equivalence between exhaustively  filtered vector spaces and flat $\bG_m$-equivariant $\Z[\hbar]$-modules --- see \cite[Lemma 2.1]{mhs2} for instance. %\ref{mhs2-flatfiltrn}
 
We will be interested in filtrations which are complete, in the sense that $V = \Lim_i V/F^i$. Via the Rees constructions, this amounts to looking at the inverse limit over $k$ of the categories of $\bG_m$-equivariant $\Z[\hbar]/\hbar^k$-modules. However, Koszul duality provides a much more efficient characterisation, as modules over the  Koszul dual  $\Z[\dbar]\simeq \oR\HHom_{\Z[\hbar]}(\Z,\Z)$ of $\Z[ \hbar]$, as follows.
%%resn of $\Z$ over $\Z[\hbar]$ is $\Z[\hbar,\dbar^*]$ with $\delta \dbar^*=\hbar$, so $\dbar$ has opposite weight to $\hbar$.

\begin{definition}
Define the $\bG_m$-equivariant dg algebra $\Z[\dbar]$ by letting $\dbar$ be a formal variable of chain degree $-1$, satisfying $\dbar^2=0$, and having weight $1$ with respect to the $\bG_m$-action. We say that a morphism of   $\Z[\dbar]$-modules in graded chain complexes is a weak equivalence if it is a quasi-isomorphisms of the underlying chain complexes, forgetting $\dbar$.
\end{definition}

\begin{definition}
 For a filtered chain complex $(V,F)$, the corresponding $\bG_m$-equivariant  $\Z[\dbar]$-module $\g\fr_FV$ is  given in weight $i$ by
\[
 \g\fr^i_FV:= \cone(F^{i+1}V \to F^iV),
\]
with $\dbar \co \g\fr_F^iV \to \g\fr^{i+1}_FV_{[-1]}$ given by the identity on $F^{i+1}V$ (and necessarily $0$ elsewhere).
\end{definition}
There is an obvious quasi-isomorphism from $\g\fr_FV$ to the associated graded $\gr_FV$, but the latter does not have a natural $\dbar$-action.

 There is a homotopy inverse functor  to $\g\fr$ %%is given by $\oR\HHom_{\Z[\dbar]}(\Z,-)$, 
 which 
can be realised 
explicitly as follows:
\begin{definition}
 Given a $\Z[\dbar]$-module $E$ in $\bG_m$-equivariant chain complexes, define the chain complex $\ff(E)$ to be the semi-infinite total complex
\[
\ff(E):= (\bigoplus_{i<0} E(i) \oplus \prod_{i \ge 0} E(i), \delta \pm \dbar),
\]
equipped with the complete exhaustive filtration
\[
 F^p \ff(E):= (\prod_{i \ge p} E(i), \delta \pm \dbar). 
\]
\end{definition}
This clearly maps weak equivalences to filtered quasi-isomorphisms. %%not right adjoint, as other not colax.

In summary:
\begin{lemma}
 The functors $\g\fr$ and $\ff$ define an equivalence between the relative category of $\Z[\dbar]$-modules in  $\bG_m$-equivariant chain complexes and the relative category of complete exhaustively filtered chain complexes and filtered quasi-isomorphisms.
\end{lemma}
\begin{proof}
  Given a $\Z[\dbar]$-module $E$ in $\bG_m$-equivariant chain complexes, we have
  \[
  \g\fr^p_F\ff(E) = (\cone(\prod_{i \ge p+1} E(i) \to \prod_{i \ge p} E(i)) , \delta \pm \dbar),
  \]
with $\dbar \co \g\fr_F^p\ff(E)  \to \g\fr^{p+1}_F\ff(E)_{[-1]}$ the identity on $\prod_{i \ge p+1} E(i)$. The canonical $\bG_m$-equivariant quasi-isomorphism $E \to \g\fr_F\ff(E)$ of $\Z[\dbar]$-modules  is then given by $e \mapsto (\pm\dbar e,e)$ in each weight.

Conversely, given a complete exhaustively filtered chain complex $(V,F)$, we have
\[
 F^p\ff(\g\fr_FV)=(\prod_{i \ge p} \cone(F^{i+1}V \to F^iV) , \delta \pm \dbar),
\]
where $\dbar$ is the identity on the respective copies of $F^{i+1}V$. The canonical filtered quasi-isomorphism $\{ F^p\ff(\g\fr_FV)\to F^pV\}_p$ is then given by summing the elements in the targets of the cones, the sum converging because the filtration is complete.  
\end{proof}

One way of thinking of the category of $\Z[\dbar]$-modules is that we are allowed to split the filtration on a filtered complex, but only at the expense of having a component $\dbar$ of the differential which does not respect the grading. The associated graded complex is then simply given by forgetting the action of $\dbar$. 

Another way of understanding this equivalence is to observe that a cofibrant resolution of $\Z[\dbar]$ as an associative algebra in chain complexes is given by the free algebra $\Z\<\dbar_1, \dbar_2, \ldots\>$ with $\dbar_m$ of chain degree $-1$ and weight $m$, with differential $\delta$ given by  $\delta\dbar_m =-  \sum_{i+j=m} \dbar_i\dbar_j$. Thus the structure of a $\Z\<\dbar_1, \dbar_2, \ldots\>$-module on a chain complex $E$ is the same as a differential $\delta + \sum \dbar_i$ on $\bigoplus_{i<0} E(i) \oplus \prod_{i \ge 0} E(i)$ respecting the filtration and agreeing with $\delta$ on the associated graded.

\begin{definition}
 Given a ring $k$, a linear algebraic group $G$ over $k$, and a $G$-equivariant  commutative algebra $R$ in chain complexes over $k$,   define the category $dg\Mod_G(R)$ to consist of $G$-equivariant $R$-modules in chain complexes.
\end{definition}

Thus the Rees construction $\xi(V,M)$ of a filtered $R$-module $M$ lies in $dg\Mod_{\bG_m}(R[\hbar])$, while $ \g\fr_FM \in dg\Mod_{\bG_m}(R[\dbar])$. When $G$ is linearly reductive, standard arguments show that there is a cofibrantly generated model structure on $dg\Mod_G(R)$ in which fibrations are surjections and weak equivalences are quasi-isomorphisms of the underlying chain complexes.

The dg algebra $R[\dbar]$ has the natural structure of  a dg Hopf $R$-algebra, by setting  $\dbar$ to be primitive, so the comultiplication $R[\dbar] \to R[\dbar]\ten_RR[\dbar]$ sends $\dbar$ to $\dbar \ten 1 +1\ten \dbar$.
\begin{definition}
 We define a closed symmetric monoidal structure $\ten_R$ on the category $dg\Mod_{\bG_m}(R[\dbar])$ by giving the chain complex $M\ten_RN$ an $R[\dbar]$-module structure via the comultiplication on the Hopf algebra $R[\dbar]$.

% The internal $\Hom$s are  given by 
% \[
%  \cHom(U,V)(i):= \HHom_{R, \bG_m}(U,V(i)).
% \]
\end{definition}
With respect to this structure, the functors $\g\fr$ and $\ff$ are both lax monoidal. By way of comparison,
note that for the usual tensor product of filtered complexes over $k$, we have $\gr_F(U\ten_k V) = \gr_F(U)\ten_{k}\gr_F(V)$. 

\subsubsection{Koszul duality for almost commutative rings}

From now on, we fix a commutative algebra $R$ in chain complexes over $\Q$. %%might cut the next 2 as redundant.
We refer to associative algebras in chain complexes as DGAAs, and commutative algebras in chain complexes as CDGAs.
 We will also refer to  to coassociative coalgebras in chain  complexes over $R$ as DGACs over $R$.

\begin{definition}
 We say that a complete filtered  DGAA $(A,F)$ is almost commutative if $\gr_FA$ is a CDGA. Similarly, a  filtered DGAC $(C,F)$ is said to be almost cocommutative if the comultiplication on $\gr_FC$ is cocommutative.
\end{definition}

\begin{remark}
An  almost commutative DGAA $(A,F)$ can be regarded as an algebra in filtered complexes for the filtered operad given by the PBW filtration on the associative operad $\Ass$, which is given by powers of the augmentation ideal of $T(V) \to \Symm(V)$.
The Rees construction $\xi(A,F)$ is thus automatically an algebra for the $\bB\bD_1$-operad over $[\bA^1/\bG_m]$ as described in \cite[\S 3.5.1]{CPTVV} (or \cite[\S 2.4.2]{CostelloGwilliamVol2} for its completion $BD_1$, dropping $\bG_m$-equivariance). Explicitly, this means that $\xi(A,F)$ is a $\bG_m$-equivariant DGAA over $R[\hbar]$ equipped with a Lie bracket $[-,-]$ of $\bG_m$-weight $-1$ which is a biderivation and satisfies $\hbar[a,b]=ab\mp ba$. 

Since we only wish to consider complete filtrations, we are effectively studying algebras $\g\fr(A,F)$ over the operad $ \g\fr(BD_1)$ in $dg\Mod_{\bG_m}(\Q[\dbar])$, where we write $BD_1$ for the complete filtered operad associated to $\bB\bD_1$. 
\end{remark}

\begin{definition}\label{bardef}
 We write $\b$ for the bar construction from possibly non-unital DGAAs over $R$ to ind-conilpotent DGACs over $R$. Explicitly, this is given by taking the tensor coalgebra
\[
 \b A:= T(A_{[-1]})= \bigoplus_{i \ge 0} (A_{[-1]})^{\ten_R i}, 
\]
with chain differential given on cogenerators $A_{[-1]}$ by combining the chain differential and multiplication on $A$.
Write $\b_+ A$ for the subcomplex $T_+(A_{[-1]})=\bigoplus_{i > 0} A_{[-1]}^{\ten_R i}$.

Let $\Omega_+$ be the left adjoint to $\b_+$, given by the tensor algebra
\[
 \Omega_+ C: = \bigoplus_{j > 0} (C_{[1]})^{\ten_R i},
\]
 with chain differential given on generators $C_{[1]}$ by combining the chain differential and comultiplication on $C$. We then define $\Omega C:= R \oplus \Omega_+C$ by formally adding a unit.
\end{definition}

\begin{definition}\label{betadef}
 Given an almost commutative DGAA $(A,F)$, we define the filtration $\beta F$ on $\b A$ by convolution with the Poincar\'e--Birkhoff--Witt filtration $\beta$. Explicitly, there is a shuffle multiplication $\nabla$ on $(\b A)_{\#}$ given on cogenerators by the identity maps $(A\ten R) \oplus (R\ten A) \to A $, making $(\b A)_{\#}$ into a Hopf algebra. Writing $F$ as an increasing filtration,  we then set 
$\beta^j\b A$ to be the image of the $j$-fold shuffle product $(\b_+ A)^{\ten j}\to \b A$ (i.e. $b_1\ten \ldots \ten b_j \mapsto b_1\nabla b_2\nabla\ldots\nabla b_j$), and
\[
 (\beta F)_i\b A: = \sum_j   F_{i+j}\cap \beta^j\b A.
\]
\end{definition}

\begin{lemma}\label{betanice}
The filtration $\beta F$ makes $\b A$ into an almost cocommutative DGAC.
\end{lemma}
\begin{proof}
 The filtration $\beta$ is automatically preserved by the comultiplication, making $(\b A)_{\#}$ a filtered coalgebra, and so $(\beta F)$ also gives a filtered coalgebra structure. To see that $\b A$ is a filtered DGAC, it only remains to show that the spaces  $(\beta F)_i\b A$ are closed under the chain differential. Since the latter is a coderivation, it suffices to check that it induces a filtered map on cogenerators.

The filtration induced by $\beta$ on cogenerators  $A_{[-1]}$ (regarded as a quotient of  $ \b A$)  is just $A_{[-1]}= \gr^{\beta}_1A_{[-1]}$, so 
$(\beta F)_i A_{[-1]}=F_{i+1}A_{[-1]}$. We also get  $\beta^1(A_{[-1]}^{\ten 2})=A_{[-1]}^{\ten 2} $, $\beta^2(A_{[-1]}^{\ten 2})=\L^2A_{[-1]}$, and $\beta^3(A_{[-1]}^{\ten 2})=0$, so 
\[
 (\beta F)_i(A_{[-1]}^{\ten 2})= F_{i+1}(A_{[-1]}^{\ten 2})+ F_{i+2}(\L^2A)_{[-2]}.
\]

On cogenerators, the differential $\delta_{\b A}$ is given by 
\[
 (A_{[-1]}^{\ten 2}) \oplus A_{[-1]} \xra{( \cdot_A ,\delta_A)} A_{[-2]},
\]
where $\cdot_A$ and $\delta_A$ denote the multiplication and chain differential on $A$.  Both $\cdot_A$ and $\delta_A$ 
automatically preserve $F$, so the only remaining condition to ensure that $\delta_{\b A}((\beta F)_i) \subset(\beta F)_i$
is that multiplication sends $F_{i+2}(\L^2A)$ to $F_{i+1}A$ --- this is precisely the condition that $\gr_FA$ be commutative.

It remains to show that the filtered DGAC $\b A$ is  almost cocommutative. Observe that  $\gr_{\beta}^1(\b A)_{\#}$ is the cofree (ind-conilpotent) graded Lie coalgebra $(\Co\Lie_R A)_{\#}$, and that the PBW filtration $\beta$ on $(\b A)_{\#}$ is then just induced from the constant filtration $\beta=\beta^1$ on $(\Co\Lie_R A)_{\#}$ by regarding $(\b A)_{\#}$
as its universal (ind-conilpotent) enveloping coalgebra. The filtration $\beta F$ on $\b A$ is similarly induced from its corestriction to $(\Co\Lie_R A)_{\#}$, where $(\beta F)_i=F_{i+1}$. In particular, on associated gradeds this implies that  
\[
 \gr^{\beta F}\b A_{\#} \cong \Co\Symm_R(\gr^F_{*+1}\Co\Lie_R A)_{\#},
\]
a cofree (ind-conilpotent) graded Poisson coalgebra, so the  comultiplication on $\gr^{\beta F}\b A$ is indeed cocommutative.
% 
% Here, the multiplication map $\gr^{\beta F}_i(A\ten A)\to \gr^{\beta F}_i(A)$ is the map 
% \[
%  \gr^F_{i+1}\Symm^2(A) \oplus \gr^F_{i+2}\L^2A \to \gr^F_{i+1}A
% \]
% given by multiplication on the first factor and Lie bracket on the second. Thus
%  the chain differential on  $\gr^{\beta F}\b A_{\#}$  involving both product and Lie bracket on $\gr^FA$. 
\end{proof}

In fact, observe that we can characterise $\beta F$ as the smallest almost cocommutative filtration on $\b A$ for which the induced filtration on cogenerators is $(\beta F)_i A_{[-1]}=F_{i+1}A_{[-1]}$. %%no contradiction when $F$ trivial: get $A_{[-1]}$ of weight $1$ and then bracket has to be weight $\le -1$ for almost commutativity, so Lie alg all weight $1$.

\begin{definition}\label{betastardef}
 Given an almost cocommutative DGAC $(C,F)$ over $R$, define the filtration $\beta^*F$ on $\Omega C$ and $\Omega_+C$ by convolution with the PBW filtration. Explicitly, define a comultiplication $\Delta$ on $T(C_{[1]})$ to be the algebra morphism sending $c \in C_{[1]}$ to $c\ten 1 + 1 \ten c$, and let $\beta^*_r:= \ker (\Delta^{(r+1)}\co T(C_{[1]}) \to T_+(C_{[1]})^{\ten r+1})$ be the kernel of the iterated comultiplication. We then set
\[
 (\beta^* F)_i\Omega C: = \sum_j   F_{i-j}\cap \beta^*_j\Omega C,
\]
and similarly for $\Omega_+C$. We then define $\hat{\Omega}_+C$ to be the completion with respect to $\beta^*$.
\end{definition}

\begin{lemma}\label{betastarnice}
The filtration $\beta^* F$ makes $\hat{\Omega} A$ into an almost commutative DGAA.
\end{lemma}
\begin{proof}
The constructions $(\b, \beta)$ and $(\Omega, \beta^*)$ are dual to each other, so
 the proof of Lemma \ref{betanice}  adapts  after taking shifts and duals.
\end{proof}

\begin{definition}\label{bBD1def}
 Define the functors $\b_{BD_1}$ and $\Omega_{BD_1}$ by $\b_{BD_1}(A,F):= (\b A, \beta F)$ and $\Omega_{BD_1}(C,F):= (\hat{\Omega} C, \beta^*F)$; define $\b_{BD_1,+}$ and $\Omega_{BD_1,+}$ similarly.
\end{definition}

\begin{lemma}
The functor  $\Omega_{BD_1,+}$ is left adjoint to the functor $\b_{BD_1,+}$ from complete non-unital almost commutative DGAAs $A$ over $R$ to non-counital almost cocommutative DGACs $C$ over $R$.
\end{lemma}
\begin{proof}
 Given $A$ and $C$, the sets $\Hom_{DGAA}(\Omega_+C,A)$ and $\Hom(C, \b A)$ can both be identified with the set 
\[
 \{f \in F_1\HHom_R(C,A)^1 ~:~ [\delta, f] + f\smile f =0\},
\]
where the product $\smile$ combines multiplication on $A$ with comultiplication on $C$.
\end{proof}
Observe that the product $\smile$ makes the complex $\HHom_R(C,A)$ into an almost commutative DGAA, so $F_1\HHom_R(C,A)$ is closed under the commutator, hence a DGLA.

\begin{lemma}\label{barcobarprop1}
 If $A$ is a complete filtered non-unital almost commutative DGAA with $\gr_FA$  flat over $R$, then the co-unit $\vareps_A\co \Omega_{BD_1,+}\b_{BD_1,+}A \to A$ of the adjunction is a filtered quasi-isomorphism.
\end{lemma}
\begin{proof}
 It suffices to show that $\vareps$ gives quasi-isomorphisms on the  graded algebras associated to the filtrations. The functors $\gr_{\beta}\b_{BD_1,+}$ and $\gr_{\beta^*}\Omega_{BD_1,+}$ are then just the bar and cobar functors for the Poisson operad, equipped with a $\bG_m$-action setting the commutative multiplication to be of weight $0$ and the Lie bracket of weight $-1$. For $\hbar$ a formal variable of weight $-1$, the graded Poisson operad can be written as $\Com \circ \hbar \Lie$, where $(\hbar \cP)(i):= \hbar^{i-1}\cP(i)$ for any operad $\cP$. The $\bG_m$-equivariant Koszul dual of the graded Poisson operad is then $(\Com \circ \hbar \Lie)^! = (\hbar^{-1} \Com) \circ \Lie = \hbar^{-1} (\Com  \circ \hbar \Lie) $, so it is self-dual after a shift in filtrations. This shift is precisely the difference between PBW and lower central series, so $\gr \vareps$ is a graded quasi-isomorphism by Koszul duality for the Poisson operad.
\end{proof}

\subsection{Hochschild complexes}

Recall that we are fixing a CDGA $R$ over $\Q$.

\begin{definition}\label{HHdef0}
 For an almost commutative DGAA $(A,F)$ over $R$ and a filtered $(A,F)$-bimodule $(M,F)$ in chain  complexes for which the left and right $\gr^FA$-module structures on $\gr^FM$ agree, we define the filtered chain complex
\[
 \CCC_{R, BD_1}(A,M)
\]
to  be the completion of the cohomological  Hochschild complex $\CCC_R(A,M)$ (rewritten as a chain complex)  with respect to the filtration $\gamma F$  defined as follows. We may identify $ \CCC_R(A,M)$ with the subcomplex of 
\[
 \HHom_R(\b A, \b(A \oplus M_{[1]})) 
\]
 consisting of  coderivations extending the zero coderivation  on $\b A$. The hypotheses on $M$ ensure that $A \oplus M$ is almost commutative (regarding $M$ as a square-zero ideal), so we have filtrations $\beta F$ on $\b A$ and $\b(A \oplus M_{[1]})$. We then define $(\gamma F)_i$ to consist of coderivations sending $(\beta F)_j \b A$ to $(\beta F)_{i+j-1}\b(A \oplus M)$.

Since a coderivation is determined by its value on cogenerators, and the cogenerators of the bar construction have weight $1$ with respect to the PBW filtration $\beta$, we may regard $(\gamma F)_i \CCC_R(A,M)_{\#}$ as the subspace of   $  \HHom_R(\b A,  M)^{\#}$ consisting of maps sending $(\beta F)_j\b A $ to $F_{i+j} M$.

We also define the subcomplex $\CCC_{R, BD_1,+}(A,M) $ to be the kernel of $\CCC_{R, BD_1}(A,M) \to M$, or equivalently  $\HHom_R(\b_+ A,  M)^{\#}$. 

\end{definition}

\begin{remark}\label{HKRrmk}
 When the filtrations $F$ are trivial in the sense that $A= \gr^F_0A$, $M=\gr^F_0M$,  we simply write $\gamma := \gamma F$, and  observe that $\gamma_0 \CCC_R(A,M)=M$, while $\gamma_1 \CCC_R(A,M) $ is just the Harrison cohomology complex. When $A$ is moreover cofibrant as a CDGA, observe that the HKR isomorphism gives a filtered levelwise quasi-isomorphism $(\CCC_R(A,M), \tau^{\HH}) \to (\CCC_{R, BD_1}(A,M),\gamma)$, where $\tau^{\HH}$ denotes good truncation in the Hochschild direction as featured in \cite[Definition \ref{DQnonneg-HHdef}]{DQnonneg}.
\end{remark}

\begin{lemma}\label{HHaclemma}
 If $ \phi \co (A,F)\to (D,F)$ is a morphism of almost commutative DGAAs over $R$, then  $\CCC_{R,BD_1}(A,D)$ is an almost commutative DGAA under the cup product, and $\CCC_{R,BD_1}(A,D) \to D$ is a morphism of almost commutative DGAAs. 
\end{lemma}
\begin{proof}
 This just follows because  $\gr^{\gamma F}\CCC_R(A,D)^{\#} = \HHom( \gr^{\beta F}\b A, \gr^FD)^{\#}  $,  with $\gr^{\beta F}\b A $ cocommutative and $\gr^FD$ commutative.
\end{proof}

\subsubsection{Brace algebra structures}\label{bracesn} 

Recall that a brace algebra $B$ over $R$ is an  $R$-cochain  complex  equipped with a cup product in the form of a chain  map
\[
 B\ten B \xra{\smile} B,
\]
 and braces in the form of  maps
\[
 \{-\}\{-,\ldots,-\}_r \co B \ten B^{\ten r}\to B^{[-r]}
\]
satisfying the conditions of \cite[\S 3.2]{voronovHtpyGerstenhaber} (where brace algebras are called homotopy $G$-algebras) with respect to the  differential. There is a  brace operad $\Br$ in cochain complexes, whose algebras are brace algebras. The commutator  of the brace $\{-\}\{-\}_1$ is a Lie bracket, so for any brace algebra $B$, there is a natural DGLA structure on $B^{[1]}$. The brace operad is weakly equivalent to the rationalisation of the little discs operad, so brace algebras are a model for $E_2$-algebras in cochain complexes.
%% $B\ten B \to B[-1]$, so $B[1]\ten B[1] \to B[1]$.

\begin{definition}\label{acbracedef}
 Define an decreasing  filtration $\gamma$ on the brace operad $\Br$ by putting the cup product in $\gamma^0$ and the braces $\{-\}\{-,\ldots,-\}_r $
in $\gamma^r$.

Thus a (brace, $\gamma$)-algebra $(A,F)$ in filtered complexes is a brace algebra for which the cup product respects the filtration, and the $r$-braces  send $F_i$ to $F_{i-r}$. We refer to (brace, $\gamma$)-algebras  as almost commutative brace algebras.
\end{definition}
Beware that the filtration $\gamma$ is not the same as that featuring in \cite[Definition 5.3]{safronovPoissonRednCoisotropic}, since we assign higher weights to higher braces.

In an almost commutative brace algebra $A$, the brace $\{-\}\{-\}_1$ is of weight $-1$; since it gives a homotopy between the cup product and its opposite, it follows that the commutator of the cup product is of weight $-1$, so $A$ is almost commutative as a DGAA. Moreover, a brace algebra structure on $A$ induces a dg bialgebra structure on $\b A$, as in \cite[\S 3.2]{voronovHtpyGerstenhaber}, and because $\beta^r \b A \subset (A_{[-1]})^{\ten \ge r}$,  the multiplication on $\b A$ given by braces  preserves the filtration $\beta F$ on  $\b_{BD_1} A$, so it is a filtered bialgebra (with almost cocommutative comultiplication).

%%algebras over the filtered Swiss Cheese operad consist of almost comm $E_2$ algebra acting on almost comm $E_1$-algebra. Forgetful functor to almost comm $E_1$-algebras, and right adjoint must be Hochschild. NO! not functorial, so we're only saying it exists 

\begin{lemma}\label{HHaclemma2}
For any almost commutative DGAA $A$ over $R$,  there is a natural almost commutative brace algebra structure on $\CCC_{R, BD_1}(A)$ over $R$.  In particular, $\CCC_{R, BD_1}(A)_{[-1]}$ is a filtered DGLA over $R$, and its associated graded DGLA is abelian.  
\end{lemma}
\begin{proof}
The formulae of \cite[\S 3]{voronovHtpyGerstenhaber} define a brace algebra structure on $\CCC_R(A)$. By Lemma \ref{HHaclemma}, we know that $(\CCC_R(A), \gamma F)$ is an almost commutative DGAA, so it suffices to show that the brace operations have the required weights.

Given $f \in (\gamma F)_p\HHom(\b A, A)$ and $g_i \in  (\gamma F)_{q_i}\HHom(\b A,A)$, each $g_i$ corresponds to a coalgebra  coderivation $\tilde{g}_i$ on $\b A$ sending $(\beta F)_j\b A$ to $(\beta F)_{j+ q_{i}-1}\b A$. 

The element $\{f\}\{g_1, \ldots, g_m\}\in \HHom(\b A,A)$ is the composition
\[
 \b A \xra{\Delta^{(m)}} (\b A)^{\ten m} \xra{ \tilde{g}_1\ten \ldots\ten \tilde{g}_m}( \b A)^{\ten m} \xra{\nabla} \b A \xra{f} A,
\]
where $\Delta^{(m)}$ is the iterated coproduct, and $\nabla$  the shuffle product. The definition of $\beta$ ensures that $\nabla$ preserves the filtration $\beta F$, so we  have
\[
 \{f\}\{g_1, \ldots, g_m\}\in (\gamma F)_{(p+q_1+\ldots +q_m-m)}\HHom(\b A,A).
\]
\end{proof}

\begin{definition}\label{braceopdef}
 Given a brace algebra $C$, define the opposite brace algebra $C^{\op}$ to have the same elements as $C$, but multiplication $b^{\op}\smile c^{\op} := (-1)^{\deg b\deg c} (c\smile b)^{\op}$ and brace operations
given by the multiplication $(\b C^{\op}) \ten (\b C^{\op})\to \b C^{\op}$ induced by the isomorphism $(\b C^{\op})\cong (\b C)^{\op}$. Explicitly,
\[
 \{b^{\op}\}\{c_1^{\op}, \ldots, c_m^{\op}\}:=  \pm\{b\}\{c_m, \ldots, c_1\}^{\op},
\]
where $\pm= (-1)^{m(m+1)/2 + (\deg f-m)(\sum_i \deg c_i -m) +  \sum_{i<j}\deg c_i\deg c_j}$.
\end{definition}
Observe that when a filtered brace algebra $C$ is almost commutative, then so is $C^{\op}$.

\begin{lemma}\label{involutiveHH} 
Given DGAAs $A,D$ over $R$, there is an anti-involution  
\[
 -i \co \CCC_R(A,D)^{\op} \to \CCC_R(A^{\op},D^{\op})
\]
of DGAAs given by
\[
 i(f)(a_1, \ldots, a_m) = - (-1)^{\sum_{i<j}  \deg a_i \deg a_j} (-1)^{m(m+1)/2}f(a_m^{\op}, \ldots , a_1^{\op})^{\op}.
\]

When $A=D$, the anti-involution $-i$ is a morphism of brace algebras, and in particular
 $i \co \CCC_R(A)_{[-1]} \to \CCC_R(A)_{[-1]}$ is a morphism of DGLAs.
Whenever $A$ is a  cofibrant  CDGA over $R$, the map $i$ corresponds under the HKR isomorphism to the involution  which acts on  $\HHom_A(\Omega^p_{A/R},A)$ as scalar multiplication by $(-1)^{p-1}$.
\end{lemma}
\begin{proof}
 This is effectively \cite[\S 2.1]{braunInvolutive}, adapted along the lines of \cite[Lemma \ref{DQnonneg-involutiveHH}]{DQnonneg}, together with the observation that $-i$ acts on braces in the prescribed manner.
\end{proof}

\subsubsection{Semidirect products}\label{semidirectsn}

\begin{lemma}\label{swisslemma}
 Given a morphism $\phi \co A \to D$  of almost commutative filtered DGAAs over  $R$, the almost commutative DGAA $\CCC_{R, BD_1}(A,D)$ is a brace module (in the sense of \cite[Definition 3.3]{safronovPoissonRednCoisotropic}\footnote{In particular, this must incorporate a  DGAA homomorphism, as in the final arXiv version (but not the journal publication) of \cite{safronovPoissonRednCoisotropic}.}) over the  
 almost commutative  brace algebra $\CCC_{R, BD_1}(A)$ of Hochschild cochains. 
 
This gives rise to  a  morphism
\[
\b_{BD_1,+}\CCC_{R,BD_1} (A)  \to\b_{BD_1,+} \CCC_{R,BD_1}(\CCC_{R,BD_1} (A,D))  
\]
of almost cocommutative bialgebras.
\end{lemma}
\begin{proof}
 Given $g_1, \ldots, g_m\in \CCC_{R,BD_1}(A)$ and $f \in \CCC_{R,BD_1}(A,D)$, the  brace operation $\{f\}\{g_1, \ldots, g_m\}$ is well-defined as an element of $\CCC_{R,BD_1}(A,D)$. Reasoning as in \cite[\S 3.2]{voronovHtpyGerstenhaber}, this combines with the morphism $\phi_* \co \CCC_{R,BD_1}(A)\to \CCC_{R,BD_1}(A,D)$
to give an action 
\[
 M_{\bt,\bt} \co  \b_{BD_1}\CCC_{R,BD_1} (A,D)\ten_R \b_{BD_1}\CCC_{R,BD_1}(A)\to \b_{BD_1}\CCC_{R,BD_1}(A,D)
\]
of almost cocommutative dg coalgebras, associative with respect to the brace multiplication of \cite{voronovHtpyGerstenhaber}. This respects the filtrations for the same reason that the multiplication does on the bar construction of an almost commutative brace algebra (Definition \ref{acbracedef}). 

Thus $\CCC_{R,BD_1}(A,D)$ is a brace $ \CCC_{R,BD_1}(A)$-module. % in the sense of \cite[Definition 3.3]{safronovPoissonRednCoisotropic}. 
On restricting to cogenerators, the multiplication above gives a map
\begin{align*}
 \b_{BD_1}\CCC_{R,BD_1} (A,D)\to &\HHom( \b_{BD_1}\CCC_{R,BD_1}(A),\CCC_{R,BD_1}(A,D))\\
&\cong \CCC_{R,BD_1}(\CCC_{R,BD_1} (A,D)),
\end{align*}
and as in \cite[Proposition 4.2]{safronovPoissonRednCoisotropic}, this induces a morphism
\[
 \b_{BD_1,+}\CCC_{R,BD_1} (A)  \to\b_{BD_1,+} \CCC_{R,BD_1}(\CCC_{R,BD_1} (A,D))
\]
 of almost cocommutative bialgebras, compatibility with the filtrations being automatic from the description above.
\end{proof}

For an $E_2$-algebra $C$ in chain complexes to act on an $E_1$-algebra $E$  is the same as a morphism from $C$ to the Hochschild complex of $E$. This is what we now construct for Hochschild complexes in the almost commutative setting, so that we will have an almost commutative brace algebra acting on an almost commutative DGAA (or equivalently a $BD_2$-algebra acting on a $BD_1$-algebra).
 Proposition \ref{barcobarprop1} then combines with the adjunction property to give  morphisms 
\[
 \CCC_{R,BD_1} (A) \xla{\sim} \Omega_{BD_1,+}\b_{BD_1,+}\CCC_{R,BD_1} (A)  \to \CCC_{R,BD_1}(\CCC_{R,BD_1} (A,D)),
\]
of almost commutative DGAAs, and we need to enhance this to keep track of the brace algebra structures:

\begin{lemma}\label{barcobarprop2}
  If $A$ is a complete filtered non-unital almost commutative brace algebra  over $R$, then there is a natural almost commutative brace algebra structure on the  DGAA $\Omega_{BD_1,+}\b_{BD_1,+}A$. If  $\gr_FA$ is moreover  flat over $R$, then there is a zigzag of filtered quasi-isomorphisms of almost commutative brace algebras  between $A$ and $\Omega_{BD_1,+}\b_{BD_1,+}A$.
\end{lemma}
\begin{proof}
As in \cite{kadeishviliCobarBialg}, there is a natural brace algebra structure on $\Omega_+C$ for any bialgebra $C$; although stated there only for characteristic $2$, the proof holds more generally, as observed in \cite{youngBraceBar}.

We now show that when $C$ is almost cocommutative, the resulting brace algebra structure on  $\Omega_{BD_1,+}C$ is almost commutative.
 For $c \in C$, the brace operation
\[
 \{c\}\{-\}\co \Omega(C) \to \Omega(C)
\]
is defined by first taking the element $\sum_r \Delta^{(r)}c \in TC$, then applying the multiplication from $C$ internally within each subspace $C^{\ten r}$. Since $\Delta$ is almost cocommutative and $\Omega C$ almost commutative, it follows that when $c \in F_pC$, we get $\{c\}\{(\beta^*F)_i\Omega C\} \subset  (\beta^*F)_{i+p}\Omega C$. Equivalently, for $y \in (\beta^*F)_i\Omega C$, the map $\{-\}\{y\}$ sends $(\beta^*F)_{p}C= F_{p-1}C$ to $(\beta^*F)_{i+p-1}\Omega C$.

We automatically have $\{c\}\{\}_0=c$, and the higher braces $\{c\}\{-\}_n \co \Omega(C)^{\ten n} \to \Omega(C)$ are then set to be $0$ for $c \in C$, and extended to the whole of $\Omega C$ via the identities
\[
 \{xz\}\{y_1, \ldots, y_n\} = \sum_{i=0}^n \pm x\{y_1, \ldots, y_i\}z\{y_{i+1}, \ldots, y_n\}.
\]

In particular, this means that $\{-\}\{y\}$ is a derivation, so must map $(\beta^*F)_{p}\Omega C$ to $(\beta^*F)_{i+p-1}\Omega C$, since it does so on generators. We can then describe  higher braces  $\{-\}\{y_1, \ldots, y_n\}$ as the composition
\[
 \Omega(C) \xra{\Delta^{(n)}} \Omega(C)^{\ten n} \xra{\{-\}\{y_1\}\ten \ldots \ten \{-\}\{y_n\} }\Omega(C)^{\ten n} \to \Omega(C),
\]
the final map being given by multiplication. By the construction of $\beta^*$, the map $\Delta^{(n)}$ preserves the filtration $(\beta^*F)$, so for $y_i \in (\beta^*F)_{q_i}\Omega C$, we have
\[
 \{-\}\{y_1, \ldots, y_n\}\co (\beta^*F)_p\Omega(C)\to (\beta^*F)_{(p+q_1+ \ldots +q_n -n)}\Omega C,
\]
making  $\Omega_{BD_1,+}C$  almost commutative

Taking $C= \b_{BD_1,+}A$ gives an almost commutative brace algebra $\Omega_{BD_1,+}\b_{BD_1,+}A$ and an almost commutative DGAA quasi-isomorphism $\Omega_{BD_1,+}\b_{BD_1,+}A \to A$ by Lemma \ref{barcobarprop1}, but this is not a brace algebra morphism in general. If we let $\Omega_{\Br,+}$ be the left adjoint to $\b_{BD_1}$ as a functor from almost commutative brace algebras to almost cocommutative bialgebras, then it suffices to establish a filtered brace algebra quasi-isomorphism $\Omega_{BD_1,+}\b_{BD_1,+}A\to \Omega_{\Br,+}\b_{BD_1,+}A $. If we disregard the filtrations, this is the main result of \cite{youngBraceBar} (which refers to brace algebras as $\cS_2$-algebras), and the filtered case  follows by observing that the homotopy of \cite[Theorem 3.3]{youngBraceBar} preserves the respective filtrations.
\end{proof}

Combining Lemmas \ref{swisslemma} and \ref{barcobarprop2} gives:
\begin{proposition}\label{swissprop} 
For any morphism $\phi \co A \to D$  of almost commutative filtered DGAAs over  $R$, there is a canonical zigzag
\[
 \CCC_{R, BD_1}(A) \xla{\sim} \tilde{C} \to \CCC_{R, BD_1}(\CCC_{R, BD_1}(A,D))
\]
of almost commutative brace algebras over $R$, where the first map is a quasi-isomorphism.
\end{proposition}

\begin{definition}\label{semidirectdef}
Given an  almost commutative  brace algebra $C$ over $R$, and an almost commutative DGAA $E$ over $R$ which is a left brace $C$-module compatibly with the filtrations, define the semidirect product $E_{[1]} \rtimes C$ to be the  almost commutative non-unital brace algebra given by the homotopy fibre product of the diagram
\[
 \tilde{C} \to \CCC_{R, BD_1}(E) \la \CCC_{R,BD_1,+}(E),
\]
for the brace algebra resolution $\tilde{C}$ of $C$ mapping to $\CCC_{R,BD_1}(E)$ via Lemma \ref{barcobarprop2} and  the proof of Lemma \ref{swisslemma}.
\end{definition}

\begin{remark}\label{swissrmk}
Observe that we have a natural morphism $ E_{[1]} \rtimes C \to C$ of non-unital  brace algebras, with homotopy fibre given by the homotopy kernel of $\CCC_{R,BD_1,+}(E) \to \CCC_{R,BD_1}(E) $. As a complex, this kernel is just $E_{[1]}$, and the underlying DGLA is just the DGLA underlying the DGAA $E$.  For more discussion of the map $\CCC_{R,+}(E) \to \CCC_R(E) $ of $E_2$-algebras, see \cite[\S 2.7]{kontsevichOperads}.

\end{remark}

The following lemma, although not essential for our constructions, greatly simplifies their interpretation.
\begin{lemma}\label{semidirectlemma}
 In the situation of Proposition \ref{swissprop}, the DGLA underlying the shifted brace algebra $(E_{[1]} \rtimes C)_{[-1]}$ is naturally quasi-isomorphic to the complex $\cone(C \to E)$ equipped with the Lie bracket
 \[
  [(c,e),(c',e')]:=(\{c\}\{c'\}_1\mp \{c'\}\{c\}_1, \{e\}\{c'\}_1\mp \{e'\}\{c\}_1 +ee'\mp e'e).
 \]
\end{lemma}
\begin{proof}
A routine check shows that this is indeed a DGLA.
%%details: 
%% in Safronov's notation, $[[c,c'],e]=e\{[c,c']\} = e\{c\{c'\}\}-e\{c'\{c\}\}=  e\{c\}\{c'\}-e\{c'\}\{c\}= [c,[c',e]]-[c',[c,e]]$
%% ditto, $[c,[e,e']] =(ee')\{c\} - (e'e)\{c\}= e\{c\}e'+ e(e'\{c\})-e'\{c\}e-e'(e\{c})$ while 
%$[e,[e',c]]= e(e'\{c\})- e'\{c}e$ and $[e',[e,c]]=e'(e\{c\})-e\{c\}e'$, so all match up.
%%$d[e,c]=\delta (e\{c\})= (\delta e)\{c\}+e\{\delta c\} +e\cdot c -c\cdot e = [de,c] +[e,\delta c] +e\cdot c -c\cdot e$, 
%%Pavel confirms we need DGA HM, not just left mod str.
%% Thus we want cone map $\theta \co C \to E$ to satisfy $e\cdot c -c\cdot e = e\theta(c)-\theta(c)e $. Since presumably $\theta(c)=c\cdot 1$, we're asking for $c \cdot e := (c \cdot 1)e$.

%%Pavel p26 says our thing is brace module.

%%possible this might in fact be a brace algebra. Certainly seems OK for $E_{[1]}$ to be non-unital brace algebra, with $\{\}\{\}$ the product and all others $0$. However, do we even have DGA structure? No, can't force $\theta$ to be derivation, because get factors of $2$. Actually, we could, just by setting right module structure to be 0, but then it won't be a unital DGA. Higher brace axioms then seem to fail, though.

We then observe that inclusion gives a canonical quasi-isomorphism $\CCC_{R,BD_1,+}(E)_{[-1]} \to \cone (\CCC_{R,BD_1}(E)\to E)$ of DGLAs, and that the natural map $\cone (\CCC_{R,BD_1}(E)\to E) \to  \CCC_{R,BD_1}(E)_{[-1]}$ is surjective, giving a model
\[
 \tilde{C}_{[-1]}\by_{\CCC_{R,BD_1}(E)_{[-1]}}\cone (\CCC_{R,BD_1}(E)\to E) \cong \cone(\tilde{C} \to E)
\]
for the DGLA underlying $(E_{[1]} \rtimes C)_{[-1]}$. This in turn is quasi-isomorphic to $\cone(C \to E) $ via the quasi-isomorphism $\tilde{C} \to C$.
 \end{proof}

\section{Defining quantisations for derived  co-isotropic structures}\label{affinesn}

In this section, we  develop a precise notion of  quantisation for derived co-isotropic structures in a stacky affine setting. 
Recall that we are fixing a CDGA $R$ over $\Q$.

\subsection{Stacky thickenings of derived affines}\label{stackyCDGAsn}
 
We now recall some definitions and lemmas from \cite[\S \ref{poisson-Artinsn}]{poisson}, as summarised in \cite[\S \ref{DQvanish-bicdgasn}]{DQvanish}. By default, we will  regard the  CDGAs in derived algebraic geometry   as chain complexes $\ldots \xra{\delta} A_1 \xra{\delta} A_0 \xra{\delta} \ldots$ rather than cochain complexes --- this will enable us to distinguish easily between derived (chain) and stacky (cochain) structures.  

\begin{definition}
A stacky CDGA is  a chain cochain complex $A^{\bt}_{\bt}$ equipped with a commutative product $A\ten A \to A$ and unit $\Q \to A$.  Given a chain  CDGA $R$, a stacky CDGA over $R$ is then a morphism $R \to A$ of stacky CDGAs. We write $DGdg\CAlg(R)$ for the category of  stacky CDGAs over $R$, and $DG^+dg\CAlg(R)$ for the full subcategory consisting of objects $A$ concentrated in non-negative cochain degrees.
\end{definition}

When working with chain cochain complexes $V^{\bt}_{\bt}$, we will usually denote the chain differential by $\delta \co V^i_j \to V^i_{j-1}$, and the cochain differential by $\pd \co V^i_j \to V^{i+1}_j$.
On a first reading, readers interested primarily in DM (as opposed to Artin) stacks may ignore the stacky part of the structure and consider only CDGAs  $A_{\bt}= A^0_{\bt}$ throughout this section.

\begin{example}\label{DstarBGex}
We now recall an important example  of a class of stacky CDGAs from \cite[Example \ref{poisson-DstarBG}]{poisson}.  
 Given a Lie algebra $\g$ of finite rank acting as derivations on a derived affine scheme $Y$, we write $O([Y/\g])$ for the stacky CDGA given by the
 Chevalley--Eilenberg double complex 
\[
  O(Y) \xra{\pd} O(Y)\ten \g^{\vee} \xra{\pd} O(Y)\ten \Lambda^2\g^{\vee}\xra{\pd} \ldots 
\]
of $\g$ with coefficients in the chain $\g$-module $O(Y)$. We think of this as a form of derived Lie algebroid.
\end{example}

\begin{definition}
 Say that a morphism $U \to V$ of chain cochain complexes is a levelwise quasi-isomorphism if $U^i \to V^i$ is a quasi-isomorphism for all $i \in \Z$. Say that a morphism of stacky CDGAs is a levelwise quasi-isomorphism if the underlying morphism of chain cochain complexes is so.
\end{definition}
There is a model structure on chain cochain complexes over $R$ in which weak equivalences are levelwise quasi-isomorphisms and fibrations are surjections --- this follows by identifying chain cochain complexes with the category $dg\Mod_{\bG_m}(R[\pd]/\pd^2)$ of \S \ref{filtrnsn}, for instance, for $\pd$ of chain degree $0$ and weight $1$, with $\pd^2=0$.

The following is \cite[Lemma \ref{poisson-bicdgamodel}]{poisson}:
\begin{lemma}\label{bicdgamodel}
There is a cofibrantly generated model structure on stacky CDGAs over $R$ in which fibrations are surjections and weak equivalences are levelwise quasi-isomorphisms. 
\end{lemma}

There is a  denormalisation functor $D$ from non-negatively graded CDGAs to cosimplicial algebras, combining Dold--Kan denormalisation of a cochain complex with the Eilenberg--Zilber shuffle product (for an explicit description, see \cite[Definition \ref{ddt1-nabla}]{ddt1};  it has a  
 left adjoint $D^*$, described explicitly in \cite[Definition \ref{DQDG-Dstardef}]{DQDG}. 
Given a cosimplicial  CDGA $A$, $D^*A$ is then a stacky CDGA in non-negative cochain degrees. By  \cite[Lemma \ref{poisson-Dstarlemma}]{poisson}, $D^*$ is a left Quillen functor from the Reedy model structure on cosimplicial CDGAs to the model structure of Lemma \ref{bicdgamodel}.

Since   $DA$ is a pro-nilpotent extension of $A^0$, when $\H_{<0}(A)=0$ we think of the simplicial hypersheaf  $\oR \Spec DA$ as a stacky derived thickening of the derived affine scheme $\oR \Spec A^0$.  Stacky CDGAs  arise as formal completions of derived Artin $N$-stacks along affine atlases, as in \cite[\S \ref{poisson-stackyCDGAsn}]{poisson}. When $X$ is a  $1$-geometric derived Artin stack (i.e. has affine diagonal), the formal completion of a smooth affine $1$-atlas $U \to X$ is given by the relative de Rham complex
\[
 O(U) \xra{\pd} \Omega^1_{U/X} \xra{\pd} \Omega^2_{U/X}\xra{\pd}\ldots ,
\]
which arises by applying the functor $D^*$  %of \cite[\S \ref{poisson-stackyCDGAsn}]{poisson} 
to the \v Cech nerve of $U$ over $X$. The construction of Example \ref{DstarBGex} is the special case of this construction corresponding to the atlas $Y \to [Y/G]$ when $G$ is an algebraic group with Lie algebra $\g$. 

By \cite[Theorem \ref{stacks2-bigthm} and Corollary \ref{stacks2-Dequivcor}]{stacks2}, every strongly quasi-compact derived Artin $N$-stack $\fX$ admits a  simplicial resolution $X_{\bt}$ by derived affines of a special form, called a DG Artin hypergroupoid. We then have an associated stacky CDGA $D^*O(X)$, which we can think of as a formal completion of $\fX$ along $X_0$, and more generally stacky CDGAs $D^*O(X^{\Delta^j})$, regarded as completions along the derived affine schemes $X_j$. 

The following is \cite[Corollary \ref{poisson-gooddescent}]{poisson}, showing that a derived Artin stack can be recovered from the stacky CDGAs $D^*O(X^{\Delta^j}) $; this should be thought of as a resolution by derived Lie algebroids.
\begin{lemma}\label{gooddescent}
For any simplicial presheaf  $F$ on $DG\Aff(R)$ and any Reedy fibrant simplicial derived affine $X$, there is a canonical weak equivalence 
\[
 \ho \Lim_{j\in\Delta} \map( \Spec DD^*O(X^{\Delta^j}), F) \to \map (X, F).
\]
\end{lemma}

\begin{definition}\label{hfetdef}
 A morphism  $A \to B$ in $DG^+dg\CAlg(R)$ is said to be  homotopy formally \'etale when the map
\[
 \{\Tot \sigma^{\le q} (\oL\Omega_{A}^1\ten_{A}^{\oL}B^0)\}_q \to \{\Tot \sigma^{\le q}(\oL\Omega_{B}^1\ten_B^{\oL}B^0)\}_q
\]
on the systems of brutal cotruncations is a pro-quasi-isomorphism (i.e. an essentially levelwise quasi-isomorphism in the sense of \cite[\S 2.1]{isaksenStrict}), where $\sigma^{\le q}$ denotes the brutal cotruncation
\[
 (\sigma^{\le q}M)^i := \begin{cases} 
                         M^i & i \ge q, \\ 0 & i<q.
                        \end{cases}
\]
\end{definition}

In particular, as in \cite[\S \ref{poisson-bidescentsn}]{poisson}, for a derived Artin hypergroupoid $X$ the maps $\pd^i \co D^*O(X^{\Delta^j}) \to D^*O(X^{\Delta^{j+1}})$ and $\sigma^i \co   D^*O(X^{\Delta^{j+1}})\to D^*O(X^{\Delta^j})$ are homotopy formally \'etale. Thus $D^*O(X^{\Delta^{\bt}})$ can be thought of as a DM hypergroupoid in stacky CDGAs, and Lemma \ref{gooddescent} allows us to think of this as a resolution for $\fX$. For a more refined statement, see  \cite[Theorem \ref{smallet-mainthm}]{smallet}. 

% Combining \cite[Proposition \ref{poisson-replaceprop}]{poisson} with  \cite[Theorem \ref{stacks2-bigthm} and Corollary \ref{stacks2-Dequivcor}]{stacks2}, every strongly quasi-compact derived Artin $N$-stack over $R$ can be resolved by a derived DM hypergroupoid (a form of homotopy formally \'etale cosimplicial diagram) in $DG^+dg\CAlg(R)$.

%\subsection{The Hochschild complex and differential operators on a stacky CDGA}

The constructions of \S \ref{centresn} all adapt to chain cochain complexes, by just regarding the cochain structure as a $\bG_m$-equivariant  $\Q[\pd]/\pd^2$-module structure; quasi-isomorphisms are only considered in the chain direction. We refer to  associative (resp. brace) algebras in chain cochain complexes as stacky DGAAs (resp.  stacky brace algebras), and have the obvious notions of almost commutativity for filtered stacky DGAAs and filtered stacky brace algebras. We define bar constructions $\b$ generalising Definition \ref{bardef}  by taking shifts  exclusively in the chain direction.

\begin{definition}\label{HHdefa}
 For a stacky DGAA $A$ over $R$ and an $A$-bimodule $M$ in chain cochain complexes, 
 we define the internal cohomological  Hochschild complex $\C\C_R(A,M)$ to be the chain cochain subcomplex of 
\[
 \cHom_R(\b A, \b(A \oplus M_{[1]})) 
\]
 consisting of  coderivations extending the zero derivation  on $\b A$, where the algebra structure on $A \oplus M_{[1]}$ is defined so that $M_{[1]}$ is a square-zero ideal. 

Since a coderivation is determined by its value on cogenerators, the complex $\C\C_R(A,M)$ is given explicitly by
\[
 \C\C_R(A,M)_{\#}:=\prod_n \cHom_R( A^{\ten n}, M)_{[n]}, 
\]
with chain differential $\delta \pm b$, for the 
Hochschild differential $b$ given by
\begin{align*}
 (b f)(a_1, \ldots , a_n) = &a_1 f(a_2, \ldots, a_n)\\
 &+ \sum_{i=1}^{n-1}(-1)^i f(a_1, \ldots, a_{i-1}, a_ia_{i+1}, a_{i+2}, \ldots, a_n)\\
&+ (-1)^n f(a_1, \ldots, a_{n-1})a_n.
\end{align*} 

We simply write $\C\C_R(A)$ for $\C\C_R(A,A)$. 

When $(A,F)$ is almost commutative and $(M,F)$ is a filtered $A$-bimodule for which the left and right $\gr^FA$-module structures on $\gr^FM$ agree, we define the filtered chain cochain complex
\[
 \C\C_{R, BD_1}(A,M)
\]
by endowing  $\C\C_R(A,M)$ with the filtration $\gamma F$ of Definition \ref{HHdef0}, and completing with respect to it. 
\end{definition}

\subsection{Differential operators}

We now fix a  stacky CDGA $B$ over a  CDGA $R$, and recall the definitions of differential operators from \cite[\S \ref{DQvanish-biquantsn}]{DQvanish}.
%\subsubsection{Differential operators}

\begin{definition}\label{Diffdef}
Given  $B$-modules $M,N$ in chain  cochain complexes, inductively define the 
filtered chain cochain complex $\cDiff(M,N)= \cDiff_{B/R}(M,N)\subset \cHom_R(M,N)$ of differential operators from $M$ to $N$  by setting
\begin{enumerate}
 \item $F_0 \cDiff(M,N)= \cHom_B(M,N)$,
\item $F_{k+1} \cDiff(M,N)=\{ u \in \cHom_R(M,N)~:~  [b,u]\in F_{k} \cDiff(M,N)\, \forall b \in B \}$, where $[b,u]= bu- (-1)^{\deg b\deg u} ub$.
\item $\cDiff(M,N)= \LLim_k F_k\cDiff(M,N)$.
\end{enumerate}
We simply write $\cDiff_{B/R}(M):= \cDiff_{B/R}(M,M)$.

We then define the filtered cochain complex $\hat{\Diff}(M,N)= \hat{\Diff}_{B/R}(M,N)\subset \hat{\HHom}_R(M,N)$ by $\hat{\Diff}(M,N):= \LLim_k \hat{\Tot} F_k\cDiff(M,N)$.
\end{definition}

\begin{definition}
Given  a $B$-module $M$ in chain cochain complexes, write $\sD(M)= \sD_{B/R}(M):= \hat{\Diff}_{B/R}(M,M)$,   which we regard as a  sub-DGAA of $\hat{\HHom}_R(M,M)$. We simply write $\sD_B= \sD_{B/R}$ for  $\sD_{B/R}(B,B)$ and $\cDiff_{B/R}$ for $\cDiff_{B/R}(B,B)$.
\end{definition}

The definitions ensure that the associated gradeds $\gr^F_k\cDiff_B(M,N)$ have the structure of $B$-modules. As in \cite{DQvanish},  there are   maps
\[
 \gr^F_{k} \cDiff(M,N) \to \cHom_B(M\ten_B\CoS^k_B\Omega^1_B,N)
\]
for all $k$, which are isomorphisms when $B$ is cofibrant. [Here,  $\CoS$ denotes cosymmetric powers, as in the notation section.]

The following is \cite[Definition \ref{DQvanish-bistrictlb}]{DQvanish}:
\begin{definition}\label{bistrictlb}
 Define a strict line bundle over $B$ to be a $B$-module $M$ in chain cochain complexes such that $M^{\#}_{\#}$ is a projective module of rank $1$ over the bigraded-commutative algebra $B^{\#}_{\#}$ underlying $B$. 
\end{definition}

The motivating examples of  strict line bundles, and the only ones we will need to consider for our applications in \S \ref{lbsn}, are  the double complexes $B_c$   defined as follows. Given $c \in \z^1\z_0B$, we just set $B_c^{\#}$ to be the $B$-module  $B^{\#}$ (so the  chain differential is still  $\delta$), and then we set the cochain differential to be $\pd +c$.

\subsection{Relative quantised polyvectors}

\begin{definition}\label{QPoldef}
Given a morphism $\phi \co A\to B$ of cofibrant stacky CDGAs  over $R$ and a strict line bundle $M$ over $B$,  we define the DGLA $Q\widehat{\Pol}(A,M;0)^{[1]}$ of $0$-shifted relative quantised polyvectors as follows. By giving $A$ and $B$ trivial filtrations, we have a composite morphism
\[
 A \to B \to \cDiff_{B/R}(M)
\]
of almost commutative stacky DGAAs.

Definition \ref{semidirectdef} and Proposition \ref{swissprop} then adapt to double complexes to give us a non-unital almost commutative stacky brace algebra
\[
 \C:= \C\C_{R,BD_1}(A, \cDiff_{B/R}(M))_{[1]} \rtimes \C\C_{R,BD_1}(A), 
\]
and we then form the complex
 \[
 Q\widehat{\Pol}_R(A,M;0) := \prod_{p \ge 0} \hat{\Tot}(\gamma F)_p \C\hbar^{p-1},
 \]
 with  $Q\widehat{\Pol}_R(A,M;0)^{[1]}$ becoming a DGLA  with bracket given by the commutator of the brace $\{-\}\{-\}_1$, closure under this operator following from the definition of the filtration $\gamma$ in Definition \ref{acbracedef}.

We define filtrations $\tilde{F}$ and $G$ on $Q\widehat{\Pol}_R(A,M;0)$ by
\begin{align*}
 \tilde{F}^iQ\widehat{\Pol}_R(A,M;0) &:= \prod_{p \ge i}\hat{\Tot}(\gamma F)_p\C\hbar^{p-1},\\
G^jQ\widehat{\Pol}_R(A,M;0) &:= Q\widehat{\Pol}_R(A,M;0)\hbar^j.
\end{align*}
\end{definition}
Note that almost commutativity of the brace algebra $\C$  implies that $[\tilde{F}^iQ\widehat{\Pol}, \tilde{F}^jQ\widehat{\Pol}] \subset \tilde{F}^{i+j-1}Q\widehat{\Pol}$ and $[G^iQ\widehat{\Pol}, G^jQ\widehat{\Pol}]\subset G^{i+j}Q\widehat{\Pol}$.

\begin{remark}\label{extremecasesrmk}

 When $B=0$, observe that $\sD_{B/R}=0$, so we just have $Q\widehat{\Pol}_R(A,0;0) \simeq \prod_{p \ge 0}(\hat{\Tot} \gamma_p\C\C_{R,BD_1}(A)\hbar^{p-1})$, which admits a filtered quasi-isomorphism from the complex $Q\widehat{\Pol}_R(A,0)$ of  $0$-shifted quantised polyvectors from \cite[Definition \ref{DQnonneg-qpoldef}]{DQnonneg} as in Remark \ref{HKRrmk}.

By Remark \ref{swissrmk}, there is always a projection $ Q\widehat{\Pol}_R(A,M;0)[1]\to Q\widehat{\Pol}_R(A,0;0)[1]$, and the homotopy fibre over $0$ is equivalent to the filtered $L_{\infty}$-algebra underlying the DGAA $Q\widehat{\Pol}_A(M,-1):= \prod_{p \ge 0}F_p\sD_{B/A}(M)\hbar^{p-1}$ of \cite[Definition \ref{DQvanish-qpoldef}]{DQvanish} when $B$ is cofibrant over $A$. The equivalence follows  because the HKR isomorphism for $A$ ensures that $ \cDiff_{B/A} \to \C\C_R(A,  \cDiff_{B/R})$  is a filtered quasi-isomorphism. 

\end{remark}

The following is standard:
\begin{definition}\label{mcPLdef}
 Given a   DGLA $L$, define the the Maurer--Cartan set by 
\[
\mc(L):= \{\omega \in  L^{1}\ \,|\, d\omega + \half[\omega,\omega]=0 \in   L^{2}\}.
\]

Following \cite{hinstack}, define the Maurer--Cartan space $\mmc(L)$ (a simplicial set) of a nilpotent  DGLA $L$ by
\[
 \mmc(L)_n:= \mc(L\ten_{\Q} \Omega^{\bt}(\Delta^n)),
\]
where 
\[
\Omega^{\bt}(\Delta^n)=\Q[t_0, t_1, \ldots, t_n,\delta t_0, \delta t_1, \ldots, \delta t_n ]/(\sum t_i -1, \sum \delta t_i)
\]
is the commutative dg algebra of de Rham polynomial forms on the $n$-simplex, with the $t_i$ of degree $0$.

Given a pro-nilpotent  DGLA $L= \Lim_i L_i$, define $\mmc(L):= \Lim_i \mmc(L_i)$.
\end{definition}

\begin{definition}\label{QPdef}%%wants english defn - done?
Given a morphism $\phi \co A\to B$ of cofibrant stacky CDGAs  over $R$ and a strict line bundle $M$ over $B$, define the space $Q\cP(A,M;0)$ of  $0$-shifted quantisations (or  of quantised  co-isotropic structures) for the pair $(A,M)$ to be the space
\[
 \mmc(\tilde{F}^2Q\widehat{\Pol}(A,M;0)^{[1]})
\]
of Maurer--Cartan elements of the pro-nilpotent DGLA $\tilde{F}^2Q\widehat{\Pol}(A,M;0)^{[1]}$.
\end{definition}

Replacing $\tilde{F}^2Q\widehat{\Pol}(A,M;0)$ with its quotient by $G^k$ gives a space $ Q\cP(A,M;0)/G^k$. On the quotient $Q\cP(A,M;0)/G^1$, the choice of line bundle $M$ does not affect the space, since $\gr^F\sD_{B/R}(M)\cong \gr^F\sD_{B/R}$;  we then refer to $\cP(A,B;0) :=Q\cP(A,M;0)/G^1$ %%or A,B
as  the space of $0$-shifted co-isotropic structures on $A \to B$, in the sense that $A$ carries a $0$-shifted Poisson structure with respect to which $B$ is co-isotropic.

\begin{remark}\label{curvedrmk}
Expanding out the definitions, a quantised co-isotropic structure, i.e.  an element of  $Q\cP(A,M;0)$, is a Maurer--Cartan element of the pro-nilpotent DGLA
\[
 \tilde{F}^2Q\widehat{\Pol}_R(A,M;0) = \prod_{p \ge 2}\hat{\Tot}(\gamma F)_p (\C\C_{R,BD_1}(A, \cDiff_{B/R}(M))_{[1]} \rtimes \C\C_{R,BD_1}(A))_{[-1]} \hbar^{p-1},
\]
and Lemma \ref{semidirectlemma} allows us to express  this semidirect product construction in terms of a DGLA structure on the cone.

Projection to the second factor gives a natural map $Q\cP(A,M;0) \to Q\cP(A,0;0) $, and the homotopy fibre over an element $\Delta_A$ is then given by the Maurer--Cartan space of the filtered DGLA underlying the DGAA
\[
 \prod_{p \ge 2}\hat{\Tot}(\gamma F)_p (\C\C_{R,BD_1}(A, \cDiff_{B/R}(M)), \delta \pm \{-\}\{\Delta_A\}_1).
\]

The term $\Delta_A \in \hat{\Tot}\C\C_{R,BD_1}(A)\brh $ gives rise to a curved almost commutative $A_{\infty}$-deformation $\tilde{A}$ of $\hat{\Tot}A$ over $R\llbracket \hbar \rrbracket$, via the canonical map
\[
\hat{\Tot}\C\C_{R,BD_1}(A)^{[1]} \to\CCC_{R,BD_1}( \hat{\Tot}A)^{[1]}.
\]
The $A_{\infty}$-structure consists of operations $m_n \co \hat{\Tot}A^{\ten n} \to \hat{\Tot}A_{[n-2]} \brh$ for all $n \ge 0$ deforming the multiplication on $A$, with $m_1$ being the differential and $m_0$ the curvature; these satisfy higher associativity conditions.

The remaining term  $\Delta_M \in \hat{\Tot}(\C\C_{R,BD_1}(A, \cDiff_{B/R}(M)) , \delta \pm \{-\}\{\Delta_A\}_1)$ then gives rise to the data of a curved almost commutative $A_{\infty}$-morphism $\tilde{A} \to \hat{\Tot}\cDiff_{B/R}(M)\llbracket \hbar \rrbracket$ deforming the map $\hat{\Tot} A \to \hat{\Tot}\cDiff_{B/R}(M)$, via the canonical map
\[
 \hat{\Tot}\C\C_{R,BD_1}(A, \cDiff_{B/R}(M)) \to \CCC_{R,BD_1}(\hat{\Tot}A, \hat{\Tot}\cDiff_{B/R}(M)).
\]
The $A_{\infty}$-morphism $\Delta_M$ consists  of maps $f_n \co \hat{\Tot}\tilde{A}^{\ten_{R\brh} n} \to  \hat{\Tot}\cDiff_{B/R}(M)_{[n-1]}\llbracket \hbar \rrbracket$ for all $n \ge 0$ deforming the composite $A \to B \to \cDiff_{B/R}(M)$, with $f_0$ deforming the differential $\delta$. These satisfy compatibility  conditions with the respective $A_{\infty}$-structures; in particular, $(f_0)^2=\sum_n f_n(m_0, \ldots, m_0)$,
so this gives $\tilde{M}:=(\hat{\Tot}M\brh,\pd \pm \delta +f_0)$ the structure of a curved $\tilde{A}$-module. 

However, there are additional restrictions on the resulting deformations: if we filter $\hat{\Tot} V$ by setting  $\Fil^p \hat{\Tot}V := \Tot^{\Pi} V^{\ge p}$, then each component of the $A_{\infty}$-structure  $m$ or $A_{\infty}$-morphism  $f$ must be bounded in the sense that for some integer $r$,  each $\Fil^p$ is mapped to $\Fil^{p+r}$. When the stacky CDGAs are bounded in the cochain direction, as occurs when they originate from $1$-geometric derived Artin stacks, these boundedness restrictions are vacuous (cf. \cite[Example \ref{DQnonneg-quantex}]{DQnonneg}), but there are still restrictions arising because  $\hat{\Tot}$ does not preserve cofibrant objects. 
\end{remark}

\begin{remark}\label{curvedrmk2}
There are similar descriptions for the quotient spaces $ Q\cP(A,M;0)/G^k$ given by truncating the structures. In particular, a $0$-shifted co-isotropic structure, i.e. an element of  $\cP(A,B;0)$, is a Maurer--Cartan element of the pro-nilpotent DGLA
\begin{align*}
 \tilde{F}^2Q\widehat{\Pol}_R(A,M;0)/G^1 &= \prod_{p \ge 2}\hat{\Tot}\gr^{\gamma F}_p (\C\C_{R,BD_1}(A, \cDiff_{B/R}(M))_{[1]} \rtimes \C\C_{R,BD_1}(A))_{[-1]} \hbar^{p-1}.
 \end{align*}

The term $\Delta_A \in \hat{\Tot}\gr^{\gamma}\C\C_{R,BD_1}(A) $ gives rise to a strong homotopy  Poisson algebra structure $\varpi$ on $\hat{\Tot}A$ extending its commutative algebra structure. 

As in \cite{DQvanish}, we have  isomorphisms $\gr^F_j\cDiff_{B/R}(M)\cong \cHom_B( \CoS_B^j\Omega^1_{B/R},B)$, so $ \prod_p \hat{\Tot}\gr^F_p\cDiff_{B/R}(M)$ is the $P_1$-algebra $\widehat{\Pol}(B,-1)$ of $(-1)$-shifted polyvectors. 
The 
remaining term  $\Delta_B \in \hat{\Tot}\gr^{\gamma F}\C\C_{R,BD_1}(A, \cDiff_{B/R}(M))$ then gives rise to the data of a Maurer--Cartan element $\pi \in \widehat{\Pol}(B,-1)$, i.e. a $(-1)$-shifted Poisson structure in the sense of \cite{poisson}, together with the data of a s.h. Poisson algebra morphism
\[
 A_{\varpi} \to T_{\pi}\widehat{\Pol}(B,-1),
\]
where $T_{\pi}\widehat{\Pol}(B,-1)$ is defined in the same way as $\widehat{\Pol}(B,-1)$, but with chain differential $\delta +[\pi,-]$.

Note that this is essentially the approach the formulation of co-isotropic structures (for all shifts) proposed by Costello and Rozenblyum, and developed by Melani and Safronov in \cite{melanisafronovI,melanisafronovII} after this paper was first written.

As in the quantised case, there are additional restrictions on the resulting deformations in terms of bounds in the cochain direction. 

\end{remark}

\begin{examples}\label{curvedex}
Here is how the descriptions of quantised co-isotropic structures simplify in settings where we need not worry about the subtleties resulting from $\hat{\Tot}$:
\begin{enumerate}
\item \label{curvedexsmooth}
If $A$ and $B$ are smooth $R$-algebras concentrated in degrees $(0,0)$, then objects of $Q\cP(A,M;0)$  just correspond to  $R\brh$-deformations $\tilde{A}$ of $A$ as an associative algebra, equipped with $R\brh$-algebra homomorphisms $\phi\co \tilde{A}\to \prod_{p \ge 0}F_p\cDiff_{B/R}(M)\hbar^p=:\hat{\xi}(\cDiff_{B/R}(M),F)$ deforming the composite $A \to B=F_0\cDiff_{B/R}(M)$. 

However, curvature does manifest itself on the level of morphisms, with an isomorphism between two objects $(\tilde{A},\phi), (\tilde{A}',\phi')$ consisting of an isomorphism $ \theta \co\tilde{A}\cong  \tilde{A}'$ deforming $\id_A$, together with an element of
\begin{eqnarray*}
 &\exp(\prod_{p \ge 1}F_{p+1}\cDiff_{B/R}(M)\hbar^{p})&\\
&=\{g \in 1+ \hbar \cDiff_{B/R}(M)\brh ~:~ g \hat{\xi}(\cDiff_{B/R}(M),F)g^{-1} \subset \hat{\xi}(\cDiff_{B/R}(M),F) \}&
\end{eqnarray*} %%justification for this is first to take log, giving $b \in \cD\brh$, then look at commutators. Inductively dedcue that $[b_i,F_0]\subset F_i$, so $b_i \in F_{i+1}$.
intertwining  $\phi'\circ \theta$ and $\phi$. 

There are also $2$-isomorphisms in the form of elements of $1+ \hbar \tilde{A}=\exp(\hbar\tilde{A})$ intertwining  $1$-morphisms in the obvious way. The space  $Q\cP(A,M;0)$ is then equivalent to the nerve of this $2$-groupoid. 
%%can write cdn as $\exp(b)^{-1}\frac{\pd \exp(b)}{\pd \hbar} \in \prod_{p \ge 2}F_p\cDiff_{B/R}(M)\hbar^{p-2}$.

\item \label{curvedexsmooth2} In the special case of the previous example where the morphism $A \to B$ is surjective, we can relax a condition by just requiring that $\phi\co \tilde{A}\to \End_R(M)\brh$, since almost commutativity of $\tilde{A}$ then combines with surjectivity to guarantee that the image of $\tilde{A}$ is contained in  $\hat{\xi}(\cDiff_{B/R}(M),F)$. 

If in addition the map $\tilde{A}\to \hat{\xi}(\cDiff_{B/R}(M),F)$ is surjective, as happens when the underlying Poisson structure is non-degenerate, then we may also relax the condition on the intertwiner $g$ to say that $g \in 1+\hbar \End_R(M)\brh$, since it must automatically then lie in $\exp(\prod_{p \ge 1}F_{p+1}\cDiff_{B/R}(M)\hbar^{p})$. For such fixed $\tilde{A}$, this means that the space  $Q\cP(A,M;0)\by^h_{Q\cP(A,0)}\{\tilde{A}\}$ of quantisations lifting $\tilde{A}$ is equivalent to the nerve of the groupoid of $\tilde{A}$-modules deforming the $A$-module $M$.

\medskip
 \item\label{curvedexDM}
 We can generalise (\ref{curvedexsmooth}) to consider the case where  $A$ and $B$ are functions on derived affine schemes,   so $A=A^0_{\bt}$ and $B=B^0_{\bt}$ are chain complexes concentrated in non-negative degrees.  Again, curvature does not manifest itself in deformations of $A$, but  the curved $A_{\infty}$-morphism $f\co \tilde{A}\to \prod_{p \ge 0}F_p\cDiff_{B/R}(M)\hbar^p $ includes  a Maurer--Cartan element $f_0$ in the target.

In this setting, we may use the bar-cobar adjunction $\Omega_{BD_1} \dashv \b_{BD_1}$ to replace $A_{\infty}$-structures with genuinely associative structures. 
Explicitly, an element of $Q\cP(A;0)$  corresponds to a flat $\hbar$-adically complete DGAA $A'$ over $R\brh$ which is commutative modulo $\hbar$ and equipped with a quasi-isomorphism $A'/\hbar \simeq A$. Beware that the DGAA $A'$ must be genuinely commutative modulo $\hbar$ --- a quasi-isomorphism $A'/\hbar \simeq A$ alone does not suffice to make $A'$ almost commutative.

An element of $Q\cP(A,M;0)$ then combines this with a Maurer--Cartan element 
  $f_0 \in \prod_{p \ge 1}F_p\cDiff_{B/R}(M)_{-1}\hbar^{p-1}$ deforming $\delta$, together with an $R\brh$-algebra homomorphism $f_+\co A'\to (\prod_{p \ge 0}F_p\cDiff_{B/R}(M)\hbar^p, \delta \ad_{f_0})$.

As in \cite[Remark \ref{DQvanish-BVrmk}]{DQvanish} and \cite[Definition 9]{kravchenko}, the term $f_0$ is providing a twisted version of a $BD_0$-algebra structure (or $BD$-algebra structure in the terminology of \cite[\S 2.4]{CostelloGwilliamVol2}). Explicitly, if $A=M$ and $f_0(1)=0$, then $f_0$ is precisely  a commutative $BV_{\infty}$-algebra in the sense of \cite{BraunLazarevHtpyBV} 
(call this a $BD_{0,\infty}$-algebra to disambiguate). In particular, this means that any local  generators for $M$ in the kernel of $f_0$ give rise locally to  quasi-isomorphisms to $BD_0$-algebras. 

[Explicitly, using \cite[Proposition 2]{kravchenko}  we can characterise a $BD_{0,\infty}$-algebra as being a shifted  $L_{\infty}$-algebra structure $([-]_n)_{n \ge 1}$ over $R\brh$, where $[-]_1=f_0$ , together with a graded-commutative product satisfying 
\begin{align*}
[a_1, \ldots, a_{n-1},bc]_n= &[a_1, \ldots, a_{n-1},b]_nc \pm b[a_1, \ldots, a_{n-1},c]_n\\ 
&+\hbar[a_1, \ldots, a_{n-1},b,c]_{n+1}.
\end{align*}
To see that the canonical map from the $BD_{0,\infty}$ operad to the $BD_0$ operad is indeed a filtered quasi-isomorphism, just observe that on reducing modulo $\hbar$ it becomes the resolution $\Com \circ s\Lie_{\infty}\to \Com \circ s\Lie$ of the $P_0$ operad.]
  
The spaces of morphisms in the $\infty$-groupoid   $Q\cP(A,M;0)$ then have contributions from intertwiners generalising the underived situation, in addition to the usual higher homotopies given by localising at quasi-isomorphisms. 
\medskip
\item\label{curvedexsmoothstacky} 
Generalising in the opposite direction, we can consider the case where there is stacky structure but no derived structure, so $A=A_0^{\bt}$ and $B=B_0^{\bt}$ are cochain complexes concentrated in non-negative degrees, with $A_0^0,B_0^0$ smooth and $A_0^{\#},B_0^{\#}$ freely generated over them by graded projective modules. Then an element of $Q\cP(A;0)$ can be described as an almost commutative curved $A_{\infty}$-deformation $\tilde{A}$ of $A$, which can more precisely be encoded as a deformation $\tilde{C} $ of the coderivation on the $\hbar$-adic completion of the $R[\hbar]$-DGAC  
$ \xi(\b_{BD_1}A)$. An element of $Q\cP(A,M;0)$ combines this with the structure of an almost commutative curved $A_{\infty}$-morphism $\tilde{A} \to \prod_{p \ge 0}F_p\cDiff_{B/R}(M)\hbar^p$, which can more precisely be characterised as a dg coalgebra  morphism from $\tilde{C}$  to the $\hbar$-adic completion of  $ \xi(\b_{BD_1}(\cDiff_{B/R}(M),F ))$. 

In this setting, curvature manifests itself immediately on the level of objects, so the structural differentials $m_1$ and $f_0$ lifting $\delta$  in our induced deformations of $A$ and $M$ need not square to $0$; in particular, this means that $f_0$ need not define a $(-1)$-shifted quantisation of the line bundle $M$ on $B$ in the sense of \cite{DQvanish}. Also beware that the nature of the model structure we chose in Lemma \ref{bicdgamodel} means that these constructions are not invariant under cochain quasi-isomorphisms, in contrast to the situation with chain quasi-isomorphisms in the purely derived setting above. 
\end{enumerate}

\end{examples}

\begin{definition}\label{TQpoldef0}
 Define the filtered tangent space to relative quantised polyvectors by 
\begin{align*}
 TQ\widehat{\Pol}(A,M;0)&:= Q\widehat{\Pol}(A,M;0)\oplus \hbar Q\widehat{\Pol}_R(A,M;0)\eps,\\
\tilde{F}^jTQ\widehat{\Pol}(A,M;0)&:= \tilde{F}^jQ\widehat{\Pol}(A,M;0)\oplus \hbar\tilde{F}^j Q\widehat{\Pol}(A,M;0) \eps,
\end{align*}
for $\eps$ of degree $0$ with $\eps^2=0$. Then  $TQ\widehat{\Pol}(A,M;0)^{[1]} $ is a DGLA, with  Lie bracket  given by $ [u+v\eps, x+y\eps]= [u,x]+ [u,y]\eps + [v,x]\eps$. 

Write $TQ\cP(A,M;0)$ for the space
\[
 \mmc(\tilde{F}^2TQ\widehat{\Pol}(A,M;0)^{[1]});
\]
this is effectively the tangent bundle on the space of quantised co-isotropic structures.
\end{definition}

\begin{definition}\label{TQPoldef}%%wants english defn ---done
 Given a Maurer--Cartan element $\Delta \in \mc( Q\widehat{\Pol}_R(A,M;0) )$, define the centre $T_{\Delta} Q\widehat{\Pol}_R(A,M;0)$ of $(A,M,\Delta)$ to be the non-unital  brace algebra
\[
 (\hbar Q\widehat{\Pol}_R(A,M;0)_{\#}, \delta_{Q\Pol} + [\Delta,-]).
\]
Note that the twisted differential $\delta_{Q\Pol} + [\Delta,-]$ is necessarily square-zero by the Maurer--Cartan conditions, and necessarily compatible with the brace operations by properties of the bracket.

We define filtrations $\tilde{F}$ and $G$ on $T_{\Delta}Q\widehat{\Pol}_R(A,M;0)$ by
\begin{align*}
 \tilde{F}^iT_{\Delta}Q\widehat{\Pol}_R(A,M;0)_{\#} &:=\hbar\tilde{F}^iQ\widehat{\Pol}_R(A,M;0)_{\#},\\ 
G^jT_{\Delta}Q\widehat{\Pol}_R(A,M;0) &:= \hbar^jT_{\Delta}Q\widehat{\Pol}_R(A,M;0).
\end{align*}

Note that $(T_{\Delta} Q\widehat{\Pol}_R(A,M;0), \tilde{F})$ is  a non-unital almost commutative brace algebra over $R$. 
\end{definition}

Observe that $T_{\Delta}Q\cP(A,M;0):= \mmc(\tilde{F}^2T_{\Delta}Q\widehat{\Pol}(A,M;0)^{[1]})$ is  just the fibre of $TQ\cP(A,M;0) \to Q\cP(A,M;0)$ over $\Delta$; we think of this as the tangent space at $\Delta$ of the space of quantised co-isotropic structures. We regard  the cohomology of $T_{\Delta} Q\widehat{\Pol}_R(A,M;0)$ as a form of  relative quantised Poisson cohomology.

%%Definition 2.18: What does this definition do? I don’t think I see why this section
% gives, for example, the correct Definition 3.12. It generalizes a similar construction
% in [Pri4], but more motivation would be appreciated.
\begin{definition}\label{Qsigmadef}
Given $\Delta \in Q\cP(A,M;0)$, define $\sigma(\Delta)\in \z^2(\tilde{F}^2T_{\Delta}Q\widehat{\Pol}(A,M;0))$ to be 
\[
 -\pd_{\hbar^{-1}}\Delta = \hbar^{2}\frac{\pd \Delta}{\pd \hbar}.
\]

More generally, define the global section  $\sigma \co Q\cP(A,M;0)\to TQ\cP(A,M;0)$ of the tangent bundle to be the map induced by the morphism $Q\widehat{\Pol}(A,M;0)\to TQ\widehat{\Pol}(A,M;0)$ of filtered DGLAs given by   $\Delta  \mapsto \Delta -\pd_{\hbar^{-1}}\Delta\eps$.  
\end{definition}
We can think of $\sigma$ as giving us elements $[\sigma(\Delta)]$ in quantised Poisson cohomology associated to quantisations $\Delta$, and it generalises the corresponding constructions \cite[Definition \ref{DQvanish-Qsigmadef}]{DQvanish} and \cite[Definition \ref{DQnonneg-Qsigmadef}]{DQnonneg} for $(-1)$-shifted and $0$-shifted quantisations. %Note that in the simplest cases, when $\Delta=\Delta_2\hbar$ is purely linear in $\hbar$, we just get $\sigma(\Delta_2\hbar)=\Delta_2\hbar^2$. 

 As in \cite[\S  \ref{poisson-bipoisssn}]{poisson}, we will usually consider stacky CDGAs    $A \in DG^+dg\CAlg(R)$ satisfying the following properties, since we can resolve derived Artin stacks by stacky CDGAs of this form, which can be thought of models for  derived higher  Lie algebroids:
\begin{assumption}\label{biCDGAprops}
\begin{enumerate}
 \item for any cofibrant replacement $\tilde{A}\to A$ in the model structure of Lemma \ref{bicdgamodel}, the morphism $\Omega^1_{\tilde{A}/R}\to \Omega^1_{A/R}$ is a levelwise quasi-isomorphism,
\item  the $A^{\#}$-module $(\Omega^1_{A/R})^{\#}$ in   graded chain complexes is cofibrant (i.e. it has the left lifting property with respect to all surjections of $A^{\#}$-modules in graded chain complexes),
\item there exists $N$ for which the chain complexes $(\Omega^1_{A/R}\ten_AA^0)^i $ are acyclic for all $i >N$.
\end{enumerate}
\end{assumption}

The following lemma breaks down the complex of quantised relative polyvectors into manageable pieces given by powers of tangent complexes.
\begin{lemma}\label{gradedcalclemma}
 If $A$  and $B$ are both cofibrant and satisfy Assumption \ref{biCDGAprops}, then $\gr_G^i\tilde{F}^pQ\widehat{\Pol}(A,M;0)$ is quasi-isomorphic to the cocone of 
\[ 
 \prod_{j \ge p} \hat{\HHom}_A(\Omega^{j-i}_{A/R},A)\hbar^{j-1}[i-j] \to \prod_{j \ge p} \hat{\HHom}_B(\oL\CoS^{j-i}_B\bL^{B/A},B)\hbar^{j-1} 
\]
coming from the connecting homomorphism $S \co   \oL\Omega^1_{B/A}=\bL^{B/A}\to\Omega^1_{A/R}[1]$.

Moreover,  $\gr_G^i\tilde{F}^pT_{\Delta}Q\widehat{\Pol}(A,M;0)$ is quasi-isomorphic to  $\hbar\gr_G^i\tilde{F}^pQ\widehat{\Pol}(A,M;0)$.
\end{lemma}
\begin{proof}
By construction, Remark \ref{swissrmk} and Lemma \ref{semidirectlemma}, the complex underlying $\gr_G^i\tilde{F}^pQ\widehat{\Pol}$ is the homotopy kernel of 
\[
\prod_{j \ge p} \hat{\Tot}\gr^{\gamma}_{j-i}\C\C^{\bt}_R(A)\hbar^{j-1} \to \prod_{j \ge p} \hat{\Tot}\gr^{\gamma F}_{j-i}\C\C^{\bt}_R(A, \cDiff_{B/R} )\hbar^{j-1},  
\]
so is just given by the cocone of that morphism.
Since $B$ is assumed cofibrant, we have isomorphisms 
\[
 \gr^F_{k} \cDiff_{B/R} \to \cHom_B(\CoS^k_B\Omega^1_{B/R},B).
\]
 The bar-cobar resolution for $A$ as a  commutative algebra then  gives quasi-isomorphisms
\begin{align*}
 \cHom_A(\Omega^{j-i}_{A/R},A)[i-j]&\to \gr^{\gamma}_{j-i}\C\C^{\bt}_R(A)\\
 \cHom_A(\CoS^{j-i}_B(\cocone(\Omega^1_{B/R} \to\Omega^1_{A/R}\ten_AB)) ,B)&\to \gr^{\gamma F}_{j-i}\C\C^{\bt}_R(A, \cDiff_{B/R} ).
\end{align*}
Since $\cocone(\Omega^1_{B/R} \to\Omega^1_{A/R}\ten_AB) $ is a model for the cotangent complex $\bL^{B/A}$, the results follow.
\end{proof}

Given an element $\Delta \in Q\cP(A,M;0)$, we write $\Delta_A$ for the image in $Q\cP(A,0)$ and $\Delta_B$ for the image in $\hat{\Tot}\C\C_{R,BD_1}(A, \cDiff_{B/R})$. If we write $\Delta =  \sum_{j \ge 2} \Delta_j \hbar^{j-1}$, then by working modulo $G^1+\tilde{F}^3$, Lemma \ref{gradedcalclemma} allows us to identify $\Delta_2=(\Delta_{2,A},\Delta_{2,B})$ with a closed element of degree $0$ in the cocone of
\[
 \hat{\HHom}_A(\Omega^2_{A/R},A) \to \oR\hat{\HHom}_B(\oL\CoS^2_B\bL^{B/A},B)[2].
\]

Now $\Delta_{2,A}$ defines a closed element of the first space, and since the composition of this map with 
\[
 \hat{\HHom}_B(\oL\CoS^2_B\bL^{B/A},B) \to \hat{\HHom}_B(\Omega^1_{B/R}\ten_B^{\oL}\bL^{B/A},B)
\]
is homotopic to $0$, $\Delta_{2,B}$ defines a closed element of the latter.

We then have a diagram
\[
 \begin{CD}
\Omega^1_{A/R} @>>> \Omega^1_{B/R}\\
@V{\Delta_{2,A}^{\flat}}VV @VV{\Delta_{2,B}^{\flat}}V\\
\hat{\HHom}_A(\Omega^1_{A/R},A) @>{S}>> \oR\hat{\HHom}_B(\bL^{B/A},B)[1]
 \end{CD}
\]
commuting up to a canonical homotopy coming from $\Delta_{2,B}$.

\begin{definition}\label{Qnondegdef}
Say that a quantisation $\Delta$ of the pair $(A,M)$  is non-degenerate if the maps 
\begin{align*}
 \Delta_{2,A}^{\flat}\co  \Tot^{\Pi} (\Omega_{A/R}^1\ten_AA^0) &\to \hat{\HHom}_A(\Omega^1_A, A^0)\\
\Delta_{2,B}^{\flat}\co  \Tot^{\Pi} (\Omega_{B/R}^1\ten_BB^0) &\to \oR\hat{\HHom}_B(\bL^{B/A}, B^0)^{[1]}
\end{align*}
are quasi-isomorphisms and $\Tot^{\Pi} (\Omega_{A/R}^1\ten_AA^0)$ (resp. $\Tot^{\Pi} (\Omega_{B/R}^1\ten_BB^0)$) is a perfect complex over $A^0$ (resp. $B^0$).
\end{definition}

In other words, a non-degenerate quantisation gives an equivalence between the cotangent and tangent complexes of $A$, and between the cotangent complex of $B$ and the derived normal bundle of $B$ over $A$.

\section{Compatibility of quantisations and isotropic structures}\label{compatsn} 

In this section, we introduce generalised isotropic structures, develop the notion of compatibility between a quantisation and a generalised isotropic structure, and give some preliminary existence results for quantisations of Lagrangians.

\subsection{Morphisms from the de Rham algebra}

%%instead of the following, do we just have $\DR(A/R) \to \DR(\hat{\Tot}A/R)$? Yes, but that doesn't guarantee map to $\oL \DR(\hat{\Tot}A/R)$.

\begin{definition}\label{DRdef}
 Given a stacky CDGA $A$ over $R$, define the stacky de Rham algebra of $A$ to be the  complete filtered stacky CDGA
\[
\cD\cR(A/R)^n_i:= \prod_{j\ge 0} (\Omega^j_A)^n_{i+j}
\]
with filtration $F^p \cD\cR(A/R)= \prod_{j\ge p} (\Omega^j_A)_{[j]}$, cochain differential $\pd$ and chain differential $\delta \pm d$, where $d$ is the de Rham differential, and the differentials $\pd,\delta$ are induced from those on $A$. 

We then write $\DR(A/R):= \hat{\Tot}\cD\cR(A/R)$.
\end{definition}
In particular, beware that the de Rham differential is absorbed in the chain (derived) structure, not the cochain (stacky) structure.

\begin{lemma}\label{liftlemma}
 Given a morphism $A \to \gr_F^0B$ of stacky CDGAs over $R$, with $A$ cofibrant and $(B,F)$ a complete filtered stacky CDGA, there is an associated filtered stacky CDGA morphism $\cD\cR(A/R) \to F^0B$ over $R$, unique up to coherent homotopy.
\end{lemma}
\begin{proof}
 We may assume that $A$ is cofibrant, and then $\cD\cR(A)$ is cofibrant as a complete filtered stacky CDGA, in the sense that it has the left lifting property with respect to  surjections of complete filtered stacky CDGAs over $R$ which are levelwise filtered quasi-isomorphisms. 
 
 For any complete filtered $A$-module $(M,F)$, we may regard $M$ as a $\cD\cR(A)$-module via the projection $\cD\cR(A) \to A$. 
  Now,  $\cD\cR(A)$ is generated as a complete filtered algebra by $A \oplus (\Omega^1_A)_{[1]}$, so ($A$ being cofibrant) the  double complex  
$ \cHom_{ \cD\cR(A),\Fil}(\Omega^1_{\cD\cR(A)/R}, M)$  
of filtered derivations from $\cD\cR(A)$ to $M$ is given by 
\[
\cone(\cHom_{ A,\Fil}(\Omega^1_{A/R}, F^1M) \to \cHom_{ A,\Fil}(\Omega^1_{A/R}, F^0M)), 
\]
the cone being taken in the chain direction; here, the first term comes from restriction of a derivation to $(\Omega^1_A)_{[1]}\subset \cD\cR(A)$ and the second from restriction to $A\subset \cD\cR(A)$.
Thus  $ \cHom_{ \cD\cR(A),\Fil}(\Omega^1_{\cD\cR(A)/R}, M)$ is  levelwise acyclic  when  $M=F^1M$.
  
%   Since $\cD\cR(A)$ is generated as a complete filtered algebra by $A \oplus (\Omega^1_A)_{[1]}$, which has $0$ intersection with $F^2\cD\cR(A)$, 
% when  $M=F^1M$ there are no restrictions on the images of generators under a filtered derivation $\cD\cR(A) \to M$.  The double complex  
% $ \cHom_{ \cD\cR(A),\Fil}(\Omega^1_{\cD\cR(A)/R}, M)$  
% of filtered derivations from $\cD\cR(A)$ to $M$ 
% is thus
%  levelwise acyclic  when  $M=F^1M$, by the construction of $\cD\cR(A)$.

Now, the double complex $ \cHom_{ \cD\cR(A),\Fil}(\Omega^1_{\cD\cR(A)/R},\gr_F^rB )$  governs the obstruction theory to lifting maps from $\cD\cR(A)$ along the square-zero extension $F^0B/F^{r+1}B \to F^0B/F^rB$. Thus the acyclicity above for all $r>0$ gives the required equivalence of mapping spaces
\[
 \map_{\Fil}(\cD\cR(A), B) \simeq \map(A, \gr_F^0B)
\]
of filtered stacky CDGAs and of stacky CDGAs, respectively.
 \end{proof}

The following is a slight generalisation of \cite[Lemma \ref{poisson-keylemma}]{poisson}:
\begin{lemma}\label{mulemma}
 Take  a cofibrant stacky CDGA $A$ over $R$, a complete filtered  CDGA $B$ over $R$, and a    filtered morphism $\phi \co \DR(A/R)\to B$. Then for any derivation $\pi\in \mc(F^1\DDer_R(B))$, there is an associated filtered CDGA morphism
\[
 \mu(-,\pi) \co \DR(A/R) \to (B, \delta + \pi)
\]
given by $\mu(a,\pi)=\phi(a)$  and $\mu(df,\pi) = \phi(df) + \pi\phi(f)$ for $a,f \in A$.
\end{lemma}
\begin{proof}
 The formulae clearly define a filtered morphism $ \mu(-,\pi) \co \DR(A)^{\#} \to B^{\#}$ of graded algebras, since $ \phi\circ d + \pi \circ \phi$ defines a derivation on $A$ with respect to $\phi \co A \to B$.  We therefore need only check that $\mu$ is a chain map. We have
\begin{align*}
 \delta\mu(a,\pi) &= \phi(\delta a)+\phi(da)\\
\pi\mu(a,\pi) &=\pi\phi(a)\\
(\delta+\pi)\mu(a,\pi)&= \mu(\delta a +da,\pi),
\end{align*}
and the calculation above applied to $a=f$ and using that $(\delta +\pi)^2=0$  gives
\begin{align*}
 (\delta+\pi)\mu(df,\pi) &= - (\delta+\pi)\mu(\delta f, \pi)\\
&= - (\delta+\pi)\phi(\delta f)\\
&= -\phi(d\delta f) -\pi\phi(\delta f)\\
&= \mu(-d\delta f, \pi)\\
&=\mu( (\delta -d)df, \pi),
\end{align*}
as required.
\end{proof}

Combining Lemmas \ref{liftlemma} and \ref{mulemma} gives:
\begin{lemma}\label{mulemma2}
 Take   a morphism $\phi \co A \to \gr_F^0B$ of stacky CDGAs over $R$,  with $A$ cofibrant and $B$ a complete filtered stacky CDGA. Then for any  $\pi\in \mmc(\hat{\Tot} F^1\cDer_R(B))$, there is an associated morphism 
\[
 \mu(-,\pi) \co \DR(A/R) \to (\hat{\Tot}{B},\delta + \pi),
\]
of 
 filtered CDGAs, unique up to coherent homotopy.
\end{lemma}

\subsection{The compatibility map}

We  now develop the notion of compatibility between de Rham data and  quantisations of a pair $(A \to B)$, generalising the notion of compatibility between generalised $0$-shifted pre-symplectic structures and $E_1$ quantisations from \cite{DQnonneg}. We begin by recalling some observations from \cite[\S \ref{DQnonneg-formalitysn}]{DQnonneg}. 

As explained succinctly in \cite{petersenGTformality}, a choice of  Levi decomposition  of the Grothendieck--Teichm\"uller group (equivalently, a Drinfeld $1$-associator) over $\Q$ gives a formality quasi-isomorphism  $E_2 \simeq P_2$. Writing $\tau$ for the good truncation filtration $\tau_{\ge p}$ on a homological operad, a formality quasi-isomorphism automatically gives a filtered quasi-isomorphism $(E_2, \tau) \simeq (P_2, \tau)$. The filtration $\tau$ on $P_2$ gives the commutative multiplication weight $0$ and the Lie bracket weight $-1$, and we refer to $(P_2, \tau)$-algebras in complete filtered complexes as almost commutative $P_2$-algebras.

Likewise, the map in  \cite{voronovHtpyGerstenhaber} from the $E_2$ operad to the brace operad $\Br$ must preserve the good truncation filtrations. Finally, note that the good truncation filtration is contained in the filtration $\gamma$ on $\Br$ from Definition \ref{acbracedef}, since all operations of homological degree $r$ lie in $\gamma^r$, so in particular the closed operations do so. Thus every almost commutative brace algebra can be thought of as an $(E_2, \tau)$-algebra.

\begin{definition}
 Given a  Levi decomposition  $w \in \Levi_{\GT}(\Q)$ of the Grothendieck--Teichm\"uller group $\GT$ over $\Q$, we denote by $p_w$ the resulting $\infty$-functor  from almost commutative brace algebras   to almost commutative $P_2$-algebras over $\Q$ (after simplicially localising both categories at filtered quasi-isomorphisms), induced by the filtered quasi-isomorphism $(E_2, \tau) \simeq (P_2, \tau)$ as above. 
\end{definition}
For readers who prefer an honest functor, we can take a cofibrant resolution $\Omega \b (P_2,\tau)$ of the filtered operad $(P_2,\tau)$ via Koszul duality as in \cite{lodayvalletteoperads}, and we can then choose a morphism  $\Omega \b (P_2,\tau) \to (\Br,\gamma)$ of filtered operads realising the $\infty$-morphism. This induces a functor from almost commutative brace algebras to $\Omega \b (P_2,\tau)$-algebras, and  bar-cobar duality for $(P_2,\tau)$ then gives us a functor from  $\Omega \b (P_2,\tau)$-algebras to $(P_2,\tau)$-algebras, and we can take the composite to be $p_w$

Since the natural $\infty$-morphism from the Lie operad to the  $E_2$ operad is given in each arity by inclusion of the top weight term for the decreasing filtration, it follows that  
   the $\infty$-functor $p_w$ automatically commutes with the fibre functors $A \mapsto F_1A$ to the underlying filtered DGLAs, 

\begin{definition}
 For any of the definitions from \S \ref{affinesn}, we add the subscript $w$ to indicate that we are replacing $\C\C_{R,BD_1}(A) $ with $p_{w}\C\C_{R,BD_1}(A)$ in the construction. 
\end{definition}

Since the DGLAs underlying $\C\C_{R,BD_1}(A) $ and $p_{w}\C\C_{R,BD_1}(A)$  are filtered quasi-isomorphic, in particular we have canonical weak equivalences $Q\cP_w(A,0) \simeq Q\cP(A,0)$. Properties of the filtration $\tilde{F}$ then ensure that the complexes $T_{\Delta}Q\widehat{\Pol}_w(A,0)$ are filtered $(P_2,\tau)$-algebras.

\begin{definition}\label{mudef}
Given a choice $w \in \Levi_{\GT}(\Q)$ of Levi decomposition for $\GT$ and an element $\Delta \in  Q\cP(A,M;0)_w/G^j$, define 
\[
 \mu_w(-,\Delta) \co \cocone(\DR(A/R) \to \DR(B/R))\llbracket\hbar\rrbracket/\hbar^j \to T_{\Delta}Q\widehat{\Pol}_w(A;M,0)/G^j
\]
as follows.

Since $[B ,F_i\cDiff_{B/A}] \subset F_{i-1}\cDiff_{B/A}$, we have a map $B \to \gr_{\gamma F}^0 \C\C_{R,BD_1}(  \cDiff_{B/A})$. If we assume $B$ to be cofibrant, then combined with the weak equivalence $\cDiff_{B/A}\to \C\C_{R,BD_1}(A, \cDiff_{B/R})$,  this gives a commutative diagram
\[
 \begin{CD}
A @>>> B \\
@VVV @VVV \\
\gr^0_{\tilde{\gamma}} (p_w\C\C_{R,BD_1}(A))\llbracket\hbar\rrbracket/\hbar^j) @>>> \gr^0_{\widetilde{\gamma F}} (p_w\C\C_{R,BD_1}( \C\C_{R,BD_1}(A, \cDiff_{B/R}) )\llbracket\hbar\rrbracket/\hbar^j), 
 \end{CD}
\]
 where the filtrations on the bottom row are taken to be $(\widetilde{\gamma F})^p:= \prod_{i \ge p} (\gamma F)_i \hbar^i$ and the final term is weakly equivalent to $\gr^0_{\widetilde{\gamma F}} (p_w\C\C_{R,BD_1}(\cDiff_{B/A} )\llbracket\hbar\rrbracket/\hbar^j)$.

Applying Lemma \ref{mulemma2} to this diagram 
and the Maurer--Cartan elements on the bottom line induced by $\Delta$ yields a diagram
\[
\begin{CD}
 \DR(A) @>{\mu_w(-,\Delta)}>>(\hat{\Tot}\widetilde{\gamma}^0 (p_w\C\C_{R,BD_1}(A)\llbracket\hbar\rrbracket/\hbar^j), \delta+ [\Delta_A,-])\\
@VVV @VVV\\
\DR(B ) @>{\mu_w(-,\Delta)}>> (\hat{\Tot}\widetilde{\gamma F}^0  (p_w\C\C_{R,BD_1}( \cDiff_{B/A})\llbracket\hbar\rrbracket/\hbar^j), \delta + [\Delta_B,-])\\
@AAA @AAA \\
0 @>>>(\hat{\Tot}\widetilde{\gamma F}^0  (p_w\C\C_{R,BD_1,+}(  \cDiff_{B/A})\llbracket\hbar\rrbracket/\hbar^j), \delta + [\Delta_B,-]),
\end{CD}
 \]
determined up to coherent homotopy, and taking homotopy limits of the columns gives the desired map.
\end{definition}

\begin{remark}\label{cfDQvanish}
When $B=0$, this recovers the definition of $\mu_w$ from \cite[Definitions \ref{DQnonneg-mudef} and  \ref{DQnonneg-muwdef}]{DQnonneg}. When $R=A$, this definition is slightly different from that in \cite[Definition \ref{DQvanish-QPolmudef}]{DQvanish}. The construction there relied on a filtered  DGAA resolution $\DR'(B/R)$ of $\DR(B/R)$, with  \cite[Lemma \ref{DQvanish-mulemma1}]{DQvanish} giving a non-commutative analogue of Lemma \ref{mulemma2}. 

Instead, Definition \ref{mudef} effectively constructs the map $\mu_w \co \DR(B/R) \to T_{\Delta}\sD_{B/R}$ in this setting by first taking
\[
 \DR(B/R) \to p_w\hat{\Tot}\C\C_{R,BD_1}( \cDiff_{B/R})
\]
using the commutative structure underlying a $P_2$-algebra, then applying the projection $\C\C_{R,BD_1}( \cDiff_{B/R})\to \cDiff_{B/R} $. The map $\mu_w$ then  converges more quickly than the map $\mu$ in \cite{DQvanish}, but depends on a choice of formality isomorphism.

This raises the question of whether the construction of \cite{DQvanish} could be adapted to unshifted symplectic structures, giving equivalences not relying on formality. This would mean establishing an analogue of Lemma \ref{liftlemma} giving a universal property for $\DR(B/R)$ within a suitable category of filtered $E_2$-algebras. The filtered DGAA $\DR'(B/R)$ is not almost commutative, but the left and right $A$-module structures on $\gr_F\DR'(B/R)$ agree. Similarly, $\DR(B/R)$ will not have the desired universal property in $BD_2$-algebras, but the analogy raises the possibility that  it might do so in some larger category. 

%%Representing object for $k[V]$ on almost comm prob $\cU$ applied to free Lie alg in $V,dV$, with bracket of weight $1$, so even $E_1$ goes wrong as $[dx,dx] \in F^3$. If instead we just require that left and right module structures agree, we get bracket of weight $0$, $dx$ weight $1$, but $[x,-]$ weight $1$, so $[x,dx], [dx,dx]\in F^2$. 
\end{remark}

\subsubsection{Generalised Lagrangians}

We now fix a  cofibrant stacky CDGA $A$ over $R$, and a cofibration $A \to B$ of stacky CDGAs over $R$.

\begin{definition}
Recall that  a $0$-shifted pre-symplectic structure $\omega$ on $A/R$ is an element
\[
 \omega \in \z^{2}F^2\DR(A/R).
\]
It is called symplectic if $\omega_2 \in \z^0\Tot^{\Pi}\Omega^2_{A/R}$ induces a quasi-isomorphism
\[
 \omega_2^{\sharp} \co \hat{\HHom}_A(\Omega^1_{A/R}, A^0)\to  \Tot^{\Pi} (\Omega_{A/R}^1\ten_AA^0) 
\]
and  $\Tot^{\Pi} (\Omega_{A/R}^1\ten_AA^0)$ is a perfect complex over $A^0$.

An isotropic structure on $B$ relative to $\omega$ is  an element $(\omega, \lambda)$ of 
\[
 \z^{2}\cocone(F^2\DR(A/R)\to F^2\DR(B/R))
\]
lifting $\omega$. This structure is called Lagrangian if $\omega$ is symplectic and the image $\bar{\lambda}_2$ of $\lambda$ in  $\z^{-1}\Tot^{\Pi}\Omega^1_{B/R}\ten_B\Omega^1_{B/A}  $ 
 induces a quasi-isomorphism
\[
 \lambda_2^{\sharp} \co \hat{\HHom}_B(\Omega^1_{B/A}, B^0)\to  \Tot^{\Pi} (\Omega_{B/R}^1\ten_BB^0)^{[-1]} 
\]
and  $\Tot^{\Pi} (\Omega_{B/A}^1\ten_BB^0)$ is a perfect complex over $B^0$.
\end{definition}

%$S \co \Omega^1_{A/R}[-1] \to \bL^{B/A}$

\begin{definition}\label{tildeFDRdef}
Define a decreasing filtration $\tilde{F}$ on $ \DR(A/R)\llbracket\hbar\rrbracket$ by 
\[
 \tilde{F}^p\DR(A/R):= \prod_{i\ge 0} F^{p-i}\DR(A/R)\hbar^{i}.
\]

Define a further filtration $G$ by $ G^k \DR(A/R)\llbracket\hbar\rrbracket = \hbar^{k}\DR(A/R)\llbracket\hbar\rrbracket$.
\end{definition}

\begin{definition}\label{GPreSpdef}
 Define the space of generalised $0$-shifted isotropic structures on the pair $(A,B)$ over $R$ to be the simplicial set
\[
 G\Iso(A,B;0):= \mmc( \tilde{F}^2\cone(\DR(A/R)\llbracket\hbar\rrbracket  \to \DR(B/R)\llbracket\hbar\rrbracket)), 
\]
where we regard the cochain complexes as a  DGLA with trivial bracket.

Also write $G\Iso(A,B;0)/\hbar^{k}$ for the obvious truncation in terms of $\DR[\hbar]/\hbar^k$, 
 so $ G\Iso(A,B;0)= \Lim_k G\Iso(A,B;0)/\hbar^{k} $. Write $\Iso := G\Iso/\hbar$, the space of $0$-shifted isotropic structures.

Set $G\Lag(A,B;0) \subset G\Iso(A,B;0)$ to consist of the points whose images in $\Iso(A,B;0)$ are   Lagrangians on symplectic structures --- this is a union of path-components. Write $\Lag:=G\Lag/\hbar$, the space of $0$-shifted Lagrangians, and $\Sp(A,0):=\Lag(A,0;0)$, the space of $0$-shifted symplectic structures.
\end{definition}
Thus the components of $G\Iso(A,B;0)$ are just elements in 
\[
 \H^1\cone(\tilde{F}^2\DR(A/R)\brh  \to \tilde{F}^2\DR(B/R)\brh), %= \H^1\cone( F^2\DR(A/R)\oplus \hbar F^1 \DR(A/R) \oplus \hbar^2 \DR(A/R)  \to F^2\DR(B/R)
\]
where $\tilde{F}^2\DR(A/R)\llbracket\hbar\rrbracket = F^2\DR(A/R)\oplus \hbar F^1 \DR(A/R) \oplus \hbar^2 \DR(A/R)\brh$ and similarly for $\tilde{F}^2\DR(B/R)\brh$, so we can think of these as power series in certain relative cohomology groups.
Equivalence classes of $n$-morphisms in $G\Iso(A,B;0)$ are then given by elements in $\H^{1-n}$ of the same complex.

\subsubsection{Compatible structures}

In addition to our morphism $A \to B$, we now fix a strict line bundle $M$ over $B$, in the sense of Definition \ref{bistrictlb}.

\begin{definition}\label{Qcompatdef}
We say that a generalised  isotropic    structure $(\omega,\lambda)$ and a quantisation $\Delta$ of the pair $(A,M)$  are  $w$-compatible (or a $w$-compatible pair) if 
\[
 [\mu_w(\omega,\lambda; \Delta)] = [-\pd_{\hbar^{-1}}(\Delta)] \in  \H^1(\tilde{F}^2T_{\Delta}Q\widehat{\Pol}_w(A,M;0)) \cong \H^1(\tilde{F}^2T_{\Delta}Q\widehat{\Pol}(A,M;0)),
\]
where $\sigma=-\pd_{\hbar^{-1}}$ is the canonical tangent vector of Definition \ref{Qsigmadef}. 
\end{definition}

This definition is chosen to lift the notion of compatibility between Poisson and symplectic structures from \cite[\S \ref{poisson-compsn}]{poisson}, in such a way that compatibility becomes a one-to-one correspondence for non-degenerate structures. As we will see, when $\Delta$ is non-degenerate it is fairly straightforward to solve for $(\omega,\lambda)$ in terms of $\Delta$   because $\mu_w(-; \Delta)$ is a filtered quasi-isomorphism.   The other direction, associating quantised co-isotropic structures to generalised isotropic structures,  will require indirect arguments in terms of obstruction theory, as in the unquantised setting. By analogy with \cite{KhudaverdianVoronov}, this correspondence can be thought of as a form of Legendre transformation.

\begin{definition}\label{vanishingdef}
Given a simplicial set $Z$, an abelian group object $A$ in simplicial sets over $Z$,  a space $X$ over $Z$ and a morphism  $s \co X \to A$ over $Z$, define the homotopy vanishing locus of $s$ over $Z$ to be the homotopy limit of the diagram
\[
\xymatrix@1{ X \ar@<0.5ex>[r]^-{s}  \ar@<-0.5ex>[r]_-{0} & A \ar[r] & Z}.
\]
\end{definition}

\begin{definition}\label{Qcompdef}%%wanted english defn
Define the space $Q\Comp_w(A,M;0)$ of quantised compatible pairs  to be the homotopy vanishing locus of  
\[
 (\mu_w - \sigma) \co G\Iso(A,B;0) \by Q\cP_w(A,M;0) \to TQ\cP_w(A,M;0)
\]
over $Q\cP_w(A,M;0)$

We define a cofiltration on this space by setting $ Q\Comp_w(A,M;0)/G^j$ to be the homotopy vanishing locus of  
\[
 (\mu_w - \sigma) \co (G\Iso(A,B;0)/G^j)  \by (Q\cP_w(A,M;0)/G^j)  \to TQ\cP_w(A,M;0)/G^j 
\]
over $Q\cP_w(A,M;0)/G^j $.
\end{definition}

Thus an element of $Q\Comp_w(A,M;0)$ consists of data $(\omega, \lambda, \Delta,\alpha)$, where $(\omega, \lambda)$ is a  generalised isotropic structure, $\Delta$ a quantisation of $(A,M)$, and $\alpha$ a homotopy between $\mu_w(\omega,\lambda,\Delta)$ and $\sigma(\Delta)$.

\begin{definition}
 Define $Q\Comp_w(A,M;0)^{\nondeg} \subset Q\Comp_w(A,M;0)$ to consist of $w$-compatible quantised pairs $(\omega, \Delta)$ with $\Delta$ non-degenerate. This is a union of path-components, and by \cite[Lemma \ref{poisson-compatnondeg}]{poisson} any pre-symplectic form compatible with a non-degenerate quantisation is symplectic. The same argument shows that any isotropic pair compatible with a non-degenerate quantisation is Lagrangian, so there is a natural projection 
\[
 Q\Comp_w(A,M;0)^{\nondeg}\to G\Lag(A,B;0)
\]
as well as the canonical map
\[
 Q\Comp_w(A,M;0)^{\nondeg} \to Q\cP_w(A,M;0)^{\nondeg}.
\]
\end{definition}

\subsection{The equivalences}

The essential idea of the following proposition is that non-degeneracy of a quantisation $\Delta$ ensures that $\mu_w(-,\Delta)$ is a filtered quasi-isomorphism, so the generalised Lagrangian data $(\omega, \lambda)$ associated to $\Delta$ are given by
\[
 -\mu_w(-,\Delta)^{-1} (\pd_{\hbar^{-1}}\Delta).
\]
\begin{proposition}\label{QcompatP1} 
For any Levi decomposition $w$ of $\GT$,  the canonical map
\begin{eqnarray*}
    Q\Comp_w(A,M;0)^{\nondeg} \to  Q\cP_w(A,M;0)^{\nondeg}\simeq  Q\cP(A,M;0)^{\nondeg}          
\end{eqnarray*}
 is a weak equivalence. In particular, $w$ gives rise to  a morphism
\[
  Q\cP(A,M;0)^{\nondeg} \to G\Lag(A,B;0)
\]
(from non-degenerate quantisations to generalised Lagrangians)
in the homotopy category of simplicial sets.
\end{proposition}
\begin{proof}
The proof of \cite[Proposition \ref{poisson-compatP1}]{poisson} adapts to this context, along much the same lines as \cite[Proposition \ref{DQnonneg-QcompatP1}]{DQnonneg}.

For any $\Delta \in Q\cP_w(A,M;0)^{\nondeg}$, the homotopy fibre of $Q\Comp_w(A,M;0)^{\nondeg} $ over $\Delta$ is just the homotopy fibre of
\[
\mu_w(-,\Delta)  \co  G\Iso(A,B;0)  \to T_{\Delta}Q\cP_w(A,M;0)
\]
over $-\pd_{\hbar^{-1}}(\Delta)$, so it suffices to show that this map is a weak equivalence.

By construction, the map
\[
 \mu_w(-,\Delta) \co \cocone(\DR(A/R) \to \DR(B/R))\llbracket\hbar\rrbracket \to T_{\Delta}Q\widehat{\Pol}_w(A,M;0)
\]
 is a morphism of complete $\tilde{F}$-filtered $R\llbracket\hbar\rrbracket$-CDGAs. Moreover, it maps $\hbar^k\tilde{F}^p\cocone(\DR(A/R) \to \DR(B/R))\llbracket\hbar\rrbracket$ to $   G^k\tilde{F}^pT_{\Delta}Q\widehat{\Pol}_w(A,M;0)$. 
 Non-degeneracy of $\Delta_2$  implies that $\mu_w(-,\Delta)$ induces  quasi-isomorphisms
\begin{align*}
  \Tot^{\Pi}\Omega^{p-k}_A\hbar^{k}[k-p] &\to \hat{\HHom}_A(\Omega^{p-k}_{A}, A)\hbar^{p}[k-p]\\
 \Tot^{\Pi} \Omega^{p-k}_B\hbar^{k}[k-p] &\to \hat{\HHom}_B(\CoS^{p-k}\bL^{B/A}, B)\hbar^{p}
  \end{align*}
on the associated gradeds $\gr_G^k\gr_{\tilde{F}}^p$.  We therefore have a quasi-isomorphism of bifiltered complexes, so we have isomorphisms on homotopy groups:
\begin{eqnarray*}
 \pi_jG\Iso(A,B;0)  &\to& \pi_jT_{\Delta}Q\cP(A,M;0)\\
 \H^{2-j}(\tilde{F}^2 \cocone(\DR(A/R) \to \DR(B/R))\llbracket\hbar\rrbracket) &\to&  \H^{2-j}(\tilde{F}^2T_{\Delta}Q\widehat{\Pol}(A,M;0)).\qedhere
\end{eqnarray*}
\end{proof}

Write $\widehat{\Pol}(A,B;0):=Q\widehat{\Pol}(A,M;0)/G^1$, with a filtration $F$ given by the image of the filtration $\tilde{F}$, then also write $\Comp$, $\cP$, $\Lag$ and $\Iso$ for $Q\Comp_w/G^1$, $Q\cP/G^1$, $G\Lag/G^1$ and $ G\Iso/G^1$, respectively. Note that since $\widehat{\Pol}(A,B;0)$ is already a $P_2$-algebra, the space $\Comp$ is  independent of the Levi decomposition $w$ of $\GT$.

The following proposition establishes an equivalence between Lagrangians and non-degenerate co-isotropic Poisson structures in the $0$-shifted setting:
\begin{proposition}\label{compatcor2}
The canonical maps
\begin{eqnarray*}
   \Comp(A,B;0)^{\nondeg} &\to&  \cP(A,B;0)^{\nondeg} \\ 
\Comp(A,B;0)^{\nondeg} &\to& \Lag(A,B;0)              
\end{eqnarray*}
 are weak equivalences.
\end{proposition}
\begin{proof}
The first equivalence is given by observing that the equivalences in  Proposition \ref{QcompatP1} respect the cofiltration $G$.  For the second equivalence, we adapt the proofs  of \cite[Corollary \ref{poisson-compatcor1} and Proposition \ref{poisson-level0prop}]{poisson}, establishing the equivalence by induction on the filtration $F$.

The space $\Lag(A,B;0)/F^3$ is just given by elements $(\omega, \lambda)$ in the cocone of $\hat{\Tot}\Omega^2_{A/R}\to \hat{\Tot}\Omega^2_{B/R}$ which are non-degenerate in the sense that the induced map $(\omega, \lambda)^{\sharp}$ induces a quasi-isomorphism
\[
\begin{CD}
\hat{\HHom}_A(\Omega^1_{A/R},A^0) @>{S}>> \hat{\HHom}_B(\Omega^1_{B/A}, B^0)[1]\\
@V{\omega^{\sharp}}VV @VV{\lambda^{\sharp}}V \\
 \hat{\Tot}(\Omega^1_{A/R}\ten_AA^0) @>>> \hat{\Tot}(\Omega^1_{B/R}\ten_BB^0)
\end{CD}
\]
%$S \co   \oL\Omega^1_{B/A}=\bL^{B/A}\to\Omega^1_{A/R}[1]$
of diagrams. 
Since $ \cP(A,B;0)/F^3$ is given by elements $(\varpi,\pi)$ in the cocone of $S \co \hat{\HHom}_A(\Omega^2_{A/R},A) \to \hat{\HHom}_B(\CoS^2_B\Omega^1_{B/A}, B)[2] $, the essentially unique Poisson structure compatible with $(\omega, \lambda)$ is just given by the image of $(\omega, \lambda)$  under the symmetric square of the homotopy inverse of $(\omega, \lambda)^{\sharp}$, so
\[
 \Comp(A,B;0)^{\nondeg}/F^3 \xra{\sim} \Lag(A,B;0)/F^3.   
\]

By \cite[Proposition \ref{poisson-obsDGLA}]{poisson}, we have fibration sequences 
\begin{align*}
 \cP(A,B;0)/F^{p+1} \to &\cP(A,B;0)/F^{p} \to  \mmc (\gr_F^{p}\widehat{\Pol}(A,B;0)^{[2]} )\\
T\cP(A,B;0)/F^{p+1} \to &T\cP(A,B;0)/F^{p} \to  \mmc (\gr_F^{p}\widehat{\Pol}(A,B;0)^{[2]}[\eps] )\\
\Iso(A,B;0)/F^{p+1} \to &\Iso(A,B;0)/F^{p} \to  \mmc (\cocone(\Omega^{p}_{A/R} \to \Omega^p_{B/R})^{[2-p]}),
\end{align*}
for $\eps^2=0$.

We can then combine these, defining a complex $
 M(\omega,\lambda,\varpi,\pi,p) 
 $
 to be the homotopy limit of the diagram
\[
 \begin{CD}
  \hat{\Tot}\Omega^{p}_{A/R}[1-p] @>>> \hat{\Tot}\Omega^{p}_{B/R}[1-p]\\
@V{\L^{p}(\varpi^{\flat})}VV  @VV{\L^{p}(\pi^{\flat})}V\\
\hat{\HHom}_A(\Omega^p_{A/R},A)[1-p] @>S>> \hat{\HHom}_B( \CoS^p_B\Omega^1_{B/A},B)[1] \\  
@A{\nu(\omega, \varpi) - (p-1)}AA @AA{\nu(\lambda, \pi) - (p-1)}A \\
 \hat{\HHom}_A(\Omega^p_{A/R},A)[1-p] @>S>> \hat{\HHom}_B( \CoS^p_B\Omega^1_{B/A},B)[1].
 \end{CD}
\]
 Here $\nu(\omega, \varpi)$ is the tangent map  at $\varpi$ of the map $\mu(\omega, -)$ from Lemma \ref{mulemma} (though in this unquantised setting,   \cite[Lemma \ref{poisson-keylemma}]{poisson} suffices), given by 
\[
 \mu(\omega, \varpi+ \rho \eps) = \mu(\omega, \varpi) + \nu(\omega, \varpi)(\rho)\eps
\]
for  $\eps^2=0$, with  $\nu(\lambda, \pi)$ defined similarly.                        

Taking homotopy fibre products  similarly to \cite[Proposition \ref{poisson-compatobs}]{poisson}, we then have 
 a commutative diagram  
\[
\begin{CD}
 (\Comp(A,B;0)^{\nondeg}/F^{p+1})_{(\omega,\lambda,\varpi, \pi)} @>>>(\Lag(A,B;0)/F^{p+1})_{(\omega,\lambda)}\\
@VVV @VVV \\
(\Comp(A,B;0)^{\nondeg}/F^{p})_{(\omega,\lambda,\varpi,\pi)} @>>>(\Lag(A,B;0)/F^{p})_{(\omega, \lambda)}\\
@VVV @VVV \\
\mmc(M(\omega,\lambda,\varpi,\pi,p)[1])@>>>\mmc(\hat{\Tot}\cocone (\Omega^{p}_{A/R}\to \Omega^{p}_{B/R})[2-p])
\end{CD}
\]
of  fibre sequences.

Arguing as in \cite[Lemma \ref{DQvanish-tangentlemma}]{DQvanish},  if we sum over all $p$ in the diagram defining $M(\omega,\lambda,\varpi,\pi,p)$, then the maps $\nu(\omega, \varpi)$ and  $\nu(\lambda, \pi)$ are derivations with respect to the commutative multiplications, so are determined by generators  $\gr_F^1$, giving $\nu(\omega, \varpi)\simeq p \Lambda^p(\varpi^{\flat} \circ \omega^{\sharp})$ and $\nu(\lambda, \pi)\simeq p\Symm^p(\pi^{\flat} \circ \lambda^{\sharp})$. Since we are in the non-degenerate setting, $\varpi^{\flat} \circ \omega^{\sharp}$ and $\pi^{\flat} \circ \lambda^{\sharp}$ are homotopic to the identity maps on their respective spaces, so $\nu(\omega, \varpi)$ and $\nu(\lambda, \pi)$ are compatibly homotopic to multiplication by $p$. 
Because $p-(p-1)$ is invertible, we then get
\[
 M(\omega,\lambda,\varpi,\pi,p) \simeq \hat{\Tot}\cocone (\Omega^{p}_{A/R}\to \Omega^{p}_{B/R})[1-p].
\]
Substituting in the diagram of fibre sequences then gives 
\begin{align*}
 &(\Comp(A,B;0)^{\nondeg}/F^{p+1}) \\
&\simeq(\Comp(A,B;0)^{\nondeg}/F^{p})\by^h_{(\Lag(A,B;0)/F^{p}) }(\Lag(A,B;0)/F^{p+1}),
\end{align*}
from  which the desired equivalence $(\Comp(A,B;0)^{\nondeg}/F^{p+1})\simeq (\Lag(A,B;0)/F^{p+1})$ follows by induction. 
\end{proof}

\begin{proposition}\label{quantprop}
For any Levi decomposition $w$ of $\GT$, the   maps
\begin{align*}
 &Q\cP_w(A,M;0)^{\nondeg}/G^j  \\
&\to(Q\cP_w(A,M;0)^{\nondeg}/G^2)\by^h_{(G\Lag(A,B;0)/G^2)}(G\Lag(A,B;0)/G^j) \\ 
&\simeq (Q\cP_w(A,M;0)^{\nondeg}/G^2)\by \prod_{2 \le i<j } \mmc(\cone(\DR(A/R) \to \DR(B/R))\hbar^i)
\end{align*}
coming from Proposition \ref{QcompatP1}  are weak equivalences for all $j \ge 2$.
\end{proposition}
\begin{proof}
 The proof of \cite[Proposition \ref{DQvanish-quantprop}]{DQvanish} and \cite[Proposition \ref{DQnonneg-quantprop}]{DQnonneg}  generalises to this setting. For $(\omega,\lambda,\varpi,\pi) \in \Comp(A,B;0)$, we may apply \cite[Proposition \ref{poisson-obsDGLA}]{poisson} to give
 a commutative diagram  
\[
\begin{CD}
 (Q\Comp_w(A,M;0)/G^{j+1})_{(\omega,\lambda,\varpi, \pi)} @>>>(G\Iso(A,B;0)/G^{j+1})_{(\omega,\lambda)}\\
@VVV @VVV \\
(Q\Comp_w(A,M;0)/G^j)_{(\omega,\lambda,\varpi,\pi)} @>>> (G\Iso(A,B;0)/G^{j})_{(\omega, \lambda)}\\
@VVV @VVV \\
\mmc(N(\omega,\lambda,\varpi,\pi,j)[1]) @>>>\mmc(\cone(F^{2-j}\DR(A/R)\to F^{2-j}\DR(B/R)) \hbar^{j})
\end{CD}
\]
of  fibre sequences, for a space $N(\omega,\lambda,\varpi,\pi,j)$ defined as follows.

We set  $N(\omega,\lambda,\varpi,\pi,j)$ to be the homotopy limit of the diagram 
\[
 \begin{CD}
 \cocone( F^{2-j}\DR(A/R)\to F^{2-j}\DR(B/R))\hbar^{j} \\ 
@VV{\mu(-,-,\varpi,\pi)}V \\
F^{2-j}T_{(\varpi,\pi)}\widehat{\Pol}(A,B;0)\hbar^{j} \\
@AA{\nu(\omega,\lambda,\varpi, \pi)+ \pd_{\hbar^{-1}}}A \\
(F^{2-j}\widehat{\Pol}(A,B;0)\hbar^{j},\delta_{\varpi,\pi})= F^{2-j}T_{(\varpi,\pi)}\widehat{\Pol}(A,B;0)\hbar^{j-1},
 \end{CD}
\]
where $\mu$ is given by Definition \ref{mudef} (taking $j=1$) % and writing $ \mu(\omega,\lambda,\varpi, \pi)$ for $ \mu((\omega,\lambda),(\varpi, \pi))$), and 
$\nu(\omega,\lambda,\varpi, \pi)$ is the tangent map of $\mu((\omega,\lambda),(-, -))$ at $(\varpi,\pi)$, given by 
\[
 \mu((\omega,\lambda), (\varpi+\tau \eps, \pi+ \rho \eps)) = \mu((\omega,\lambda), (\varpi, \pi)) + \nu(\omega,\lambda,\varpi, \pi)(\tau, \rho)\eps
\]
 with $\eps^2=0$.

On the associated graded pieces, the map 
$
\gr_F^p\nu(\omega,\lambda,\varpi, \pi)
$
on $\gr_F^p\widehat{\Pol}(A,B;0)\hbar^{j}$ comes from taking cocones of the rows in the  commutative diagram
% \[
%  \begin{CD}
%   \hat{\HHom}_A(\Omega^p_{A},A)^{[-p]}\hbar^{p+j-1} @>{\gr_F^p\nu(\omega, \varpi)}>> \hat{\HHom}_A(\Omega^p_{A},A)^{[-p]}\hbar^{p+j}\\
%   @VVV @VVV\\
%    \hat{\HHom}_B(\CoS^p \cone(\Omega^1_A\ten_AB \to \Omega^1_{B}),B))\hbar^{p+j-1}@>{\gr_F^p\nu(\lambda, \pi)}>>\hat{\HHom}_B(\CoS^p \cone(\Omega^1_A\ten_AB \to \Omega^1_{B}),B))\hbar^{p+j}.
%    \end{CD}
% \]
\[
 \begin{CD}
  \hat{\HHom}_A(\Omega^p_{A},A)^{[-p]}\hbar^{p+j-1} @>>> \hat{\HHom}_B(\CoS^p \cone(\Omega^1_A\ten_AB \to \Omega^1_{B}),B))\hbar^{p+j-1}\\
  @V{\gr_F^p\nu(\omega, \varpi)}VV @VV{\gr_F^p\nu(\lambda, \pi)}V \\
  \hat{\HHom}_A(\Omega^p_{A},A)^{[-p]}\hbar^{p+j} @>>> \hat{\HHom}_B(\CoS^p \cone(\Omega^1_A\ten_AB \to \Omega^1_{B}),B))\hbar^{p+j}.
     \end{CD}
\]

As in the proof of Proposition \ref{compatcor2},  the maps $\gr_F^p\nu(\omega, \varpi)$ and $\gr_F^p\nu(\lambda, \pi)$ are compatibly homotopic to multiplication by $p\hbar$, since we are in the non-degenerate setting.
We thus have
\[
 \gr_F^p(\nu(\omega,\lambda,\varpi, \pi) + \pd_{\hbar^{-1}})\simeq p\hbar -(p+j-1)\hbar= (1-j)\hbar,
\]
and as this is an isomorphism for all $j \ge 2$, 
 the map  $N(\omega, \lambda,\varpi,\pi,j) \to \cocone(F^{2-j}\DR(A/R)\to F^{2-j}\DR(B/R)) \hbar^{j}$ is a quasi-isomorphism,  which inductively gives the required weak equivalences from the fibre sequences above.
\end{proof}

\begin{remarks}\label{quantrmk}
Taking the limit over all $j$, Proposition \ref{quantprop}  gives an equivalence
\begin{align*}
 &Q\cP_w(A,M;0)^{\nondeg}\\
& \simeq (Q\cP_w(A,M;0)^{\nondeg}/G^2)\by \prod_{i \ge 2} \mmc(\cone(\DR(A/R) \to \DR(B/R))\hbar^i);
\end{align*}
in particular, this means that there is a canonical map 
\[
 (Q\cP(A,M;0)^{\nondeg}/G^2) \to Q\cP(A,M;0)^{\nondeg},
\]
dependent on $w\in \Levi_{\GT}$, corresponding to the distinguished point $0$.

Even if $\pi$ is degenerate, a variant of Proposition \ref{quantprop} still holds. Because $\varpi^{\flat} \circ \omega^{\sharp}$ and $\pi^{\flat} \circ \lambda^{\sharp} $ are homotopy idempotent, the map   $\gr_F^p\nu(\omega,\lambda, \varpi, \pi)$ has eigenvalues in the interval $[0,p]$, so we just replace  $(1-j)$ with an operator  having eigenvalues in the interval $[1-p-j, 1-j]$.
 Since this is still a quasi-isomorphism for $j>1$, we have
\begin{align*}
 &Q\Comp_w(A,M;0) \\
&\simeq (Q\Comp_w(A,M;0)/G^2)\by \prod_{i \ge 2} \mmc(\cocone(\DR(A/R)\to \DR(B/R))\hbar^i).
\end{align*}
giving a sufficient first-order criterion  for degenerate quantisations to exist.
\end{remarks}

\section{Global quantisations}\label{stacksn}  

As in \cite[\S \ref{DQvanish-Artinsn}]{DQvanish} and  \cite[\S \ref{DQnonneg-stacksn}]{DQnonneg}, in order to pass from stacky CDGAs to derived Artin stacks, we will exploit a form of \'etale functoriality. We then introduce the notion of self-duality and thus establish the existence of quantisations for derived Lagrangians. 

\subsection{Diagrams of quantised  pairs}\label{Artindiagsn}

\begin{definition} 
Given a small category $I$,  an $I$-diagram $(A,F)$ in almost commutative stacky DGAAs over $R$, and a filtered $A$-bimodule $M$ in $I$-diagrams of chain cochain complexes for which  the left and right $\gr^FA$-module structures on $\gr^FM$ agree, we define the filtered chain cochain complex
\[
 \C\C_{R, BD_1}(A,M)
\]
to be the equaliser of the obvious diagram
\[
\prod_{i\in I} \C\C^{\bt}_{R,BD_1}(A(i),M(i)) \implies \prod_{f\co i \to j \text{ in } I} \C\C^{\bt}_{R,BD_1}(A(i),M(j)),
\]
for the $BD_1$ Hochschild complexes of Definition \ref{HHdefa}.

We then write $\C\C^{\bt}_{R,BD_1}(A):= \C\C^{\bt}_{R,BD_1}(A,A)$,   which  inherits the structure of  a stacky  brace algebra from  each $\C\C^{\bt}_{R,BD_1}(A(i),A(i))$. 
\end{definition}

Note that if $u \co I \to J$ is a morphism of small categories and $A$ is a  $J$-diagram of almost commutative stacky DGAAs over $R$, with $B= A \circ u$, then we have a natural map $ \C\C^{\bt}_R(A) \to \C\C^{\bt}_R(B)$.

In order to ensure that $\C\C^{\bt}_R(A,M)$ has the correct homological properties, we now consider categories  of  the form $[m]= (0 \to 1 \to \ldots \to m)$. Similarly to \cite[Lemma \ref{DQnonneg-calcCClemma}]{DQnonneg}, the construction $\C\C^{\bt}_R(A,M)$ preserves weak equivalences provided we restrict to pairs $(A,M)$ for which  each $A(i)$ is cofibrant as an $R$-module and $M$ is fibrant for the injective model structure (i.e. the maps $M(i) \to M(i+1)$ are all surjective).

As in \cite[\S \ref{DQvanish-Artindiagramsn}]{DQvanish}, we can do much the same for differential operators:
\begin{definition} 
 Given a small category $I$,  an $I$-diagram $B$ of stacky CDGAs over $R$, and   $B$-modules $M,N$ in chain  cochain complexes,  define the 
filtered chain cochain complex $\cDiff_{B/R}(M,N)$
to be the equaliser of the obvious diagram
\[
 \prod_{i \in I} \cDiff_{B(i)/R}(M(i),N(i)) \implies \prod_{f\co i \to j \text{ in } I}   \cDiff_{B(i)/R}(M(i),f_*N(j)),
\]
and write $\cDiff_{B/R}$ for $\cDiff_{B/R}(B,B)$
\end{definition}

If $B$ is an  $[m]$-diagram in  $DG^+dg\CAlg(R)$  which is cofibrant and  fibrant for the injective model structure (i.e. each $B(i)$ is cofibrant in the model structure of Lemma \ref{bicdgamodel} and the maps $B(i) \to B(i+1)$ are surjective), then observe that $\gr^F_k\cDiff_{B/R}$
is a model for the derived $\Hom$-complex $\oR \cHom_B(\CoS^k_B\Omega^k_{B/R},B)$.

The constructions in \S \ref{affinesn} now all carry over verbatim, generalising from morphisms of cofibrant stacky CDGAs to morphisms $A \to B$ of  $[m]$-diagrams of stacky CDGAs which are cofibrant and  fibrant for the injective model structure. In particular, for any such morphism and a strict line bundle $M$ over $B$, we have a DGLA
\[
 Q\widehat{\Pol}(A,M;0)^{[1]}
\]
 of $0$-shifted relative quantised polyvectors  as in Definition \ref{QPoldef}, and a space
\[
 Q\cP(A,M;0)
\]
 of quantisations of the pair $(A,M)$ as in Definition \ref{QPdef}.

In order to identify $Q\cP/G^1$ with  $\cP$, and for  notions such as non-degeneracy to make sense, we have to  assume that for our fibrant cofibrant  $[m]$-diagrams $A,B$ of stacky CDGAs, each $A(j),B(j)$ satisfies Assumption \ref{biCDGAprops}, so there exists $N$ for which the chain complexes $(\Omega^1_{A(j)/R}\ten_{A(j)}A(j)^0)^i $ are acyclic for all $i >N$, and similarly for $B$.

\begin{definition}\label{ICompdef}
Given a morphism $A \to B$ of fibrant cofibrant  $[m]$-diagrams  in stacky CDGAs (for the injective model structure)
define  
\[
 G\Iso(A,B;0):=  G\Iso(A(0),B(0);0)= \Lim_{i\in [m]}  G\Iso(A(i),B(i);0),
\]
for the space $G\Iso$ of generalised isotropic structures of Definition \ref{GPreSpdef}, and define the space $G\Lag(A,B;0)$ of generalised Lagrangians similarly. 

 Given a choice $w \in \Levi_{\GT}(\Q)$ of Levi decomposition for $\GT$, define 
\[
 \mu_w \co G\Iso(A,B;0) \by Q\cP_w(A,M;0) \to TQ\cP_w(A,M;0) 
\]
by setting $\mu_w(\omega,\lambda, \Delta)(i):= \mu_w(\omega(i),\lambda(i), \Delta(i)) \in TQ\cP_w(A(i),B(i);0)$ for $i \in [m]$, and let $ Q\Comp_w(A,M;0)$ be the homotopy vanishing locus of
\[
(\mu_w - \sigma) \co  G\Iso(A,B;0) \by Q\cP_w(A,M;0) \to  TQ\cP_w(A,M;0).
\]
over $Q\cP_w(A,M;0)$.
\end{definition}

As in \cite[\S \ref{poisson-bidescentsn}]{poisson}, if we let $(DG^+dg\CAlg(R)^{[1]})^{\et} \subset DG^+dg\CAlg(R)^{[1]}$ be the wide subcategory of the arrow category with only homotopy formally \'etale morphisms (see Definition \ref{hfetdef}) between arrows, then
for any of the constructions $F$ based on $Q\cP$ (i.e. $Q\cP$, $Q\Comp_w$ and their cotruncations $Q\cP/G^k$, $Q\Comp_w/G^k$), 

\cite[Definition \ref{poisson-inftyFdef}]{poisson} adapts  to give
an $\infty$-functor
\[
 \oR F \co \oL (DG^+dg\CAlg(R)^{[1]})^{\et} \to  \oL s\Set
\]
from the $\infty$-category of stacky CDGAs (simplicially localised at weak equivalences) and homotopy formally \'etale morphisms to the $\infty$-category of simplicial sets.
This construction has the property that 
\[
(\oR F)(\phi \co A\to B) \simeq F(\phi\co A\to B)
\]
for all morphisms $\phi$ of cofibrant stacky CDGAs $A$ over $R$.

Immediate consequences of Propositions \ref{QcompatP1} and \ref{quantprop} are that for any $w \in \Levi_{\GT}(\Q)$, the canonical maps 
 \begin{align*}
& Q\Comp_w(A,M;0)^{\nondeg} \to  Q\cP_w(A,M;0)^{\nondeg}\simeq  Q\cP(A,M;0)^{\nondeg};\\      
& Q\cP_w(A,M;0)^{\nondeg}/G^j\\
 &\to %(Q\cP_w(A,M;0)^{\nondeg}/G^2)\by^h_{(G\Lag(A,B;0)/G^2)}(G\Lag(A,B;0)/G^j) \\ &\simeq 
(Q\cP_w(A,M;0)^{\nondeg}/G^2)\by \prod_{2 \le i<j } \mmc(\cocone(\DR(A/R) \to \DR(B/R))\hbar^i[1])
\end{align*}
 are weak equivalences of $\infty$-functors on the full subcategory of $(\oL DG^+dg\CAlg(R)^{[1]})^{\et}$ consisting of objects satisfying the conditions of Assumption \ref{biCDGAprops}, for all $j \ge 2$.

\subsection{Descent and line bundles}\label{lbsn}

In order to translate our constructions from stacky CDGAs to derived Artin stacks, we now follow the approach set out in \cite[\S \ref{poisson-bidescentsn}]{poisson}, adapted to include line bundles as in \cite[\S \ref{DQvanish-lbsn}]{DQvanish}.

The denormalisation functor $D \co DG^+dg_+\CAlg(R) \to dg_+\CAlg(R)^{\Delta}$ from stacky CDGAs to cosimplicial CDGAs (cf. \cite[Definition \ref{ddt1-nabla}]{ddt1} allows us to extend simplicial functors $F$ on CDGAs to simplicial functors on stacky CDGAs, given by $B \mapsto \ho \Lim_{i\in \Delta} F(D^iB)$.
% \[
% F(B):= \ho \Lim_{i\in \Delta} F(D^iB).
% \]

\begin{definition}
Given a derived Artin $N$-stack $X$, and $A \in DG^+dg\CAlg(R)$, we say that an element  $f \in \ho \Lim_i X(D^iA)$ is homotopy formally \'etale if the induced morphism 
\[
N_cf_0^*\bL^{X/R} \to \{ \Tot \sigma^{\le q} \oL\Omega^1_{A/R}\ten^{\oL}_AA^0\}_q
\]
 from \cite[\S \ref{poisson-Artintgtsn}]{poisson}  is a pro-quasi-isomorphism. 
\end{definition}
In this situation, it makes sense to think of $A$ as a derived Lie algebroid locally isomorphic to $X$.

This allows us to exploit \'etale functoriality of our constructions on stacky CDGAs, allowing them to descend to derived Artin stacks as follows.
3
\begin{definition}
Given a morphism $X \to Y$ of derived Artin $N$-stacks, we  write
 $(dg_+DG\Aff_{\et}^{[1]}\da X/Y)$ for the arrow $\infty$-category in which objects are pairs $(f,z)$, for  morphisms $f \co \Spec B \to \Spec A$ in the simplicial localisation of  $DG^+dg\CAlg(R)^{\op}$ at levelwise quasi-isomorphisms, and
 homotopy formally \'etale elements $z \in \ho \Lim_i X(D^iB)\by^h_{Y(D^iB)}Y(D^iA)$; morphisms in this  $\infty$-category are given by 
spaces of compatible homotopy formally \'etale maps $A \to A'$, $B \to B'$ .
\end{definition}

We now extend the constructions above to line bundles, via $\bG_m$-equivariance exactly as in  \cite[\S \ref{DQvanish-lbsn}]{DQvanish}. When working with CDGAs with no stacky structure, this can be done just by observing that there is a natural $\bG_m$-action on $Q\cP$ given by conjugation, since the derived stack associated to $B\bG_m$ is just the hypersheafification of  of the nerve of the functor $B \mapsto \bG_m(B_0)$. 

However, for stacky CDGAs, we must replace the group $\bG_m(B_0)$ with the groupoid
\[
  \mathrm{TLB}(B):= [\z^1(\z_0B)/(\z_0B^0)^{\by}]
\]
of trivial line bundles, where $f \in (B^0)^{\by}$ acts on $\z^1B$ by addition of $\pd \log f = f^{-1}\pd f$. Here, an element 
$c \in \z^1(\z_0B)$ corresponds to the strict line bundle $B_c=(B^{\#}, \pd +c)$, with invertible elements $f \in (\z_0B^0)$ giving isomorphisms $f\co B_{c+\pd \log f}\to B_c $. %% $(\pd +c)(fm)= f\pd m + m\pd f +cfm = f(\pd m +cm + m\pd \log f)$.
The reason this works is that the nerve of $\mathrm{TLB}$ is essentially the smallest functor which hypersheafifies to recover $B \mapsto \ho \Lim_{i\in \Delta} B\bG_m(D^iB)$.

% 
% As in \cite[\S \ref{DQvanish-lbsn}]{DQvanish}, we
%  say that a morphism $A \to C$ in $DG^+dg\CAlg(R)$ is a covering if $A^0 \to C^0$ is faithfully flat. In particular, this implies that 
% \[                                                                                                                                                                  
% \ho\LLim_i \oR\Spec D^iC \to \ho\LLim_i \oR\Spec D^iA                                                                                                   \]
%  is a surjection of \'etale hypersheaves, for the denormalisation functor $D$ from stacky CDGAs to cosimplicial chain CDGAs %%just evaluate at Henselian local rings with separated closed residue field, and see that only level $0$ matters.
%  When  $X \to Y$ is  a relative trivial derived Artin hypergroupoid in the sense of \cite{stacks2}, $X_0 \to Y_0$ is faithfully flat, so the morphism $D^*O(Y) \to D^*O(X)$ is a covering in the sense above. 

For any morphism $A \to B$ of cofibrant stacky CDGAs over $R$, we can then extend $ Q\cP(A,B;0)$   to a simplicial representation of the 
groupoid $\mathrm{TLB}(B)$ above by sending an object 
%$b \in \z^1(\Omega^{\bt}(\Delta^i) \ten \Tot^{\Pi}A)$ to $\Lim_p \mc( \tilde{F}^2 (Q\widehat{\Pol}(A_b,-1) /\tilde{F}^{i+2})\ten \Omega^{\bt}(\Delta^i) )$.
$c \in \z^1(\z_0B)$ to $Q\cP(A,B_c;0)$, with $(\z_0B^0)^{\by}$ acting via functoriality for strict line bundles. Note that the quotient representation $Q\cP(-,-;0)/G^1= \cP(-,0)$ is trivial;  we also set $G\Iso$ to be a trivial representation $c \mapsto G\Iso(A,B;0)$.

\begin{definition}\label{quotientbyGmdef}
For any of the constructions $F$ of \S \ref{Artindiagsn}, let $\oR (F/^h\bG_m)$ be the $\infty$-functor on $\oL dg\CAlg(R)^{\et}$  given by applying the construction of  \cite[\S \ref{poisson-bidescentsn}]{poisson} to the right-derived functor of the Grothendieck construction
\[
  B \mapsto %\ho\LLim_{b \in [\z^1(\z_0A)/(\z_0A^0)^{\by}]} F(b),
\holim_{\substack{ \lra \\ c \in \mathrm{TLB}(B)}} F(A,B_c),
\]
 then taking  hypersheafification with respect to homotopy formally  \'etale coverings. 
\end{definition}

\begin{definition}\label{QPdefglobal}
 Given a map $f \co X \to Y$ of  strongly quasi-compact derived Artin $N$-stacks over $R$,  a line bundle $\sL$ on $X$ and any of the functors $F$ above, define
$
  F(Y,\sL)
$
 to be the homotopy limit of 
\[
\oR(F/^h\bG_m)(A,B)\by_{\oR (*/^h\bG_m)(B)}^h\{\sL|_{B}\}
\]
 over objects $\Spec B \to \Spec A$ in  the $\infty$-category $(dg_+DG\Aff_{\et}^{[1]}\da X/Y)$.
\end{definition}

\begin{remark}\label{smallerresnrmk}
In many cases, we can take smaller categories than $(dg_+DG\Aff_{\et}^{[1]}\da X/Y)$ on which to calculate the homotopy limit. When the $\bG_m$-action on $F$ is trivial, we can restrict to  compatible hypergroupoid resolutions of $X$ and $Y$ as in \cite[\S \ref{poisson-bidescentsn}]{poisson}, and in general we just need the resolution of $X$ to be compatible with the canonical resolution of $B\bG_m$. When $X$ and $Y$ are derived Deligne--Mumford $N$-stacks, we do not need stacky CDGAs at all, and can just work over $(DG\Aff_{\et}^{[1]}\da X/Y)$. 

When $X$ and $Y$ are $1$-geometric derived Artin stacks, we may just consider the $\infty$-category of commutative diagrams
\[
 \begin{CD}
  U @>f>> X \\
@VVV @VVV \\
V @>g>> Y
 \end{CD}
\]
with $U,V$ derived affines and the maps $f,g$ being smooth; to this we associate the morphism $\Omega^{\bt}_{U/X} \to \Omega^{\bt}_{V/Y}$ of stacky CDGAs as in \S \ref{stackyCDGAsn}, giving an object of $(dg_+DG\Aff_{\et}^{[1]}\da X/Y)$.
\end{remark}

\begin{remark}\label{algdrmk}
 Following Remark \ref{curvedrmk}, we may regard an element of $Q\cP(Y,\sL;0)$  as a sheaf on  $(DG\Aff_{\et}^{[1]}\da X/Y)$ deforming the pair $(\sO_Y,\sL)$, by combining a suitable curved  $A_{\infty}$ deformation $\tilde{\sO}_Y$ of $\sO_Y$ over $R\brh$ with an $f^{-1}\tilde{\sO}_Y$-module $\tilde{\sL}$ deforming $\sL$ over $R\brh$, the deformation being given by  $R$-linear differential operators with restrictions on their orders. 
 
 In fact, there is an $f^{-1}\tilde{\sO}_Y-\sD_X\brh$-bimodule 
 $
\sE_{\hbar}:=(\sL\ten_{\sO_{X}}\sD_{X}\brh, \delta + \Delta_{\tilde{\sL}} \cdot-),
 $
 %which we abusively denote by $ \tilde{\sL}\hten_{\sO_{X}}\sD_{X}$. We can then 
from which we can 
recover $\tilde{\sL}$ as $\sE_{\hbar}\ten_{\sD_{X}\brh}\sO_{X}\brh$.
% \[
%  \tilde{\sL} = (\tilde{\sL}\hten_{\sO_{X}}\sD_{X})\ten_{\sD_{X}}\sO_{X}. %%NB not a right $\O_X$-mod because $\sD_X$ doesn't act $\O_X$-linearly. 
% \]
 In particular, this gives us a functor from right  $\tilde{\sO}_Y$-modules $\sN$ to right $\sD_X\brh$-modules $f^{-1}\sN\ten_{f^{-1}\tilde{\sO}_Y}^{\oL}\sE_{\hbar}$. %$(\tilde{\sL}\hten_{\sO_{X}}\sD_{X})$.
 \end{remark}

\begin{examples}\label{algdex}
Here are some cases where the description simplifies:
\begin{enumerate}
 
\item\label{algdexsmooth}
The simplest case to consider is when $X$ and $Y$ are both smooth (underived) Deligne--Mumford $N$-stacks, so we can work with algebras instead of stacky CDGAs. Then the description of Example \ref{curvedex}.(\ref{curvedexsmooth}) implies that $\tilde{\sO}_Y$ is locally given by an associative deformation of the sheaf $\sO_{Y}$ on the \'etale site of $Y$, but the presence of $2$-automorphisms makes $\tilde{\sO}_Y$ an algebroid deformation, i.e. an  $R\brh$-deformation of $\sO_Y$ regarded as a $2$-sheaf of $R$-linear categories.

Then $\tilde{\sL}$ gives rise to  an $R\brh$-linear functor from the algebroid $f^{-1}\tilde{\sO}_Y$  on the \'etale site of $X$ to the $R\brh$-linear category of  right $\sD_X\brh$-modules on $X$, together with conditions on orders of differential operators which are difficult to characterise directly. However, when $f$ is a closed immersion and  $\tilde{\sO}_Y$ is non-degenerate, Example \ref{curvedex}.(\ref{curvedexsmooth2}) implies that   $\tilde{\sL}$ is just a  $R\brh$-linear functor from the algebroid $f^{-1}\tilde{\sO}_Y$ to the category of complete flat $R\brh$-modules on $X_{\et}$, reducing to the constant functor $f^{-1}\sO_Y \mapsto \sL$ modulo $\hbar$, with no further conditions necessary.

\medskip
\item\label{algdexDM}
Generalising to the case where $X$ and $Y$ are both derived Deligne--Mumford $N$-stacks, the description of Remark \ref{curvedex}.(\ref{curvedexDM}) similarly implies that $\tilde{\sO}_Y$ gives rise to an  associative $R\brh$-deformation $\sA$ of $\sO_{Y}$ as a hypersheaf of $R$-linear dg categories, but this throws away information about almost commutativity, so we cannot recover $\tilde{\sO}_Y$ from the algebroid. There is a similar loss of information associating right $\sD$-modules to   $\tilde{\sL}$. Thus each quantisation gives rise to  (but cannot be recovered from) an  $\infty$-algebroid $\sA$ on $Y$ equipped with  an  $R\brh$-linear $\infty$-functor from $f^{-1}\sA$ to the $R\brh$-linear $\infty$-category of   right $\sD_X\brh$-modules, deforming the constant functor $f^{-1}\sO_Y \mapsto \sL$. 

% \medskip
% \item 
% When $X$ and $Y$ are $1$-geometric derived Artin stacks,  we can use Remark \ref{smallerresnrmk}  to describe   elements of $Q\cP(Y,\sL;0)$ in terms of compatible $1$-atlases $U,V$ of $X,Y$. We see via Remark \ref{curvedrmk} that each quantisation is  a form of  curved  $A_{\infty}$ deformation $\tilde{\sO}_Y$ of the presheaf $V \mapsto \Tot \Omega^{\bt}_{V/Y}$, together with a curved $f^{-1}\tilde{\sO}_Y$-module deformation of the presheaf $U \mapsto \Tot \Omega^{\bt}_{U/X}\ten_{f^{-1}\O_{X}}f^{-1}\sL$ given by suitable $R$-linear differential operators.    

\end{enumerate}
\end{examples}

Adapting \cite[Definition \ref{DQvanish-nondegstack}]{DQvanish} along the lines of Definition \ref{Qnondegdef} gives:
\begin{definition}\label{nondegstack}
Say that a quantisation $\Delta \in Q\cP(Y,\sL;0)/G^k$  is non-degenerate if the induced maps from cotangent complexes to tangent complexes
\begin{align*}
 \Delta_{2,Y}^{\flat}\co  \bL^{Y/R} \to \oR\hom_{\sO_Y}(\bL^{Y/R}, \sO_{Y/R})\\
\Delta_{2,X}^{\flat}\co  \bL^{X/R} \to \oR\hom_{\sO_{X}}(\bL^{X/Y}, \sO_{X})[1]
\end{align*}
 are quasi-isomorphisms and  and $\bL^{X}, \bL^{Y}$ are perfect.
\end{definition}

Propositions \ref{compatcor2} and  \ref{quantprop} now readily generalise (substituting the relevant results from \cite[\S \ref{poisson-Artinsn}]{poisson} to pass from local to global),   showing that the only obstruction to quantising a non-degenerate co-isotropic structure is first-order:
\begin{proposition}\label{prop3}
For any morphism  $X \to Y$ of derived Artin $N$-stacks, any line bundle $\sL$ on $X$ and any $w \in \Levi_{\GT}(\Q)$, the  canonical maps
\begin{align*} 
  \Comp(Y,X;0)^{\nondeg} &\to  \cP(Y,X;0)^{\nondeg} \\ 
\Comp(Y,X;0)^{\nondeg} &\to \Lag(Y,X;0)\\        
Q\Comp_w(Y,\sL;0)^{\nondeg} &\to  Q\cP(Y,\sL;0)^{\nondeg}
\end{align*}
\begin{align*} 
 Q\Comp_w(Y,\sL;0) &\to (Q\Comp_w(Y,\sL;0)/G^2) \by^h_{(G\Iso(Y,X;0)/G^2)} G\Iso(Y,X;0)\simeq \\
&(Q\Comp_w(Y,\sL;0)/G^2)\by \prod_{i \ge 2} \mmc(\cone(\DR(Y/R)\to \DR(X/R) ) \hbar^i)
\end{align*}
are filtered weak equivalences.  

In particular, $w$ gives rise to a  morphism in the homotopy category of simplicial sets
\[
  Q\cP(Y,\sL;0)^{\nondeg} \to G\Lag(Y,X;0)
\]
from the space of quantised co-isotropic structures to the space of generalised Lagrangians, which induces a weak equivalence
\begin{align*}
 Q\cP(Y,\sL;0)^{\nondeg} &\to  (Q\cP(Y,\sL;0)^{\nondeg}/G^2) \by^h_{ G\Lag(Y,X;0)/G^2}G\Lag(Y,X;0)\simeq\\
&(Q\cP(Y,\sL;0)^{\nondeg}/G^2)\by \prod_{i \ge 2} \mmc(\cone(\DR(Y/R)\to \DR(X/R) ) \hbar^i).
\end{align*}
\end{proposition}

\begin{remark}\label{cfBGKP} 
The results of Proposition \ref{prop3} are compatible with those of \cite[Theorem 1.1.4]{BaranovskyGinzburgKaledinPecharich}, which fixes a sheaf $\tilde{\sO}_{Y}$ of associative algebras  quantising a symplectic structure on a smooth variety $Y$, and describes $\tilde{\sO}_{Y}$-module deformations of line bundles  $\sL$ on smooth closed Lagrangians $X\subset Y$. As in Example \ref{algdex}.(\ref{algdexsmooth}), this groupoid corresponds precisely to our space $ Q\cP(Y,\sL;0)^{\nondeg}\by^h_{Q\cP(Y,0)^{\nondeg}}\{\tilde{\sO}_Y\}$ in this specialised setting, although we consider more general quantisations  $\tilde{\sO}_{Y}$. 

In the generality of Proposition \ref{prop3}, the  first order deformation problem is a question of lifting $Q\cP(Y,\sL;0)^{\nondeg}/G^2\to Q\cP(Y,0)^{\nondeg}/G^2$ %over $\cP(Y,\sL;0)^{\nondeg}\to \cP(Y,0)^{\nondeg}$ 
over a Lagrangian structure $\pi \in \Lag(Y,X;0)^{\nondeg}$,  so DGLA obstruction theory applied to the complexes of quantised polyvectors allows us to  read off the  obstruction space as   $  \H^3(\cocone(F^1T_{\pi}\widehat{\Pol}(Y,X;0) \to F^1T_{\pi}\widehat{\Pol}(Y;0)))$, which is isomorphic via the compatibility map $\mu(-,\pi)$ to $  \H^2F^1\DR(X)$. By Proposition \ref{prop3}, the  higher order  deformation problem is then simply a case of lifting an element $u \in \hbar^2\H^2\DR(Y/R)\brh$ (determined by $\tilde{\sO}_Y$) to $\hbar^2\H^1(\cone(\DR(Y)\to \DR(X))\brh$, giving the higher order obstruction as the image of $u$ in $\hbar^2\H^2\DR(X/R)\brh$.

In their restricted setting, \cite{BaranovskyGinzburgKaledinPecharich} indeed show that the potential first order obstruction to quantising $\sL$ over $\tilde{\sO}_Y$   is given by a class $c_1 (\sL) -\half  c_1 (K_X ) -\At(\tilde{\sO}_Y , X) \in \H^2F^1\DR(X)$, with higher order obstructions  a power series in $\hbar^2\H^2\DR(X)\brh$ depending only on $\tilde{\sO}_Y$. 

When $\sL^{\ten 2}$ has a right $\sD$-module structure, the Chern class $c_1 (\sL) -\half  c_1 (K_X )$ vanishes. Moreover, whenever there is an isomorphism $\tilde{\O}_Y\simeq \tilde{\O}_Y^{\op}$ of quantisations which is semilinear with respect to the transformation $\hbar \mapsto -\hbar$, the calculations of  \cite[Remark 5.3.4]{BaranovskyGinzburgKaledinPecharich} show that $\At(\tilde{\O}_Y,X)=0$. Thus their first order obstruction does indeed vanish in the scenario of %\cite[Theorem 1.1.4]{BaranovskyGinzburgKaledinPecharich} 
Theorem \ref{quantpropsd} below, with the higher order obstruction given by Corollary \ref{quantcorsd}. %in fact that theorem implies that for such self-dual line bundles and involutive quantisations, the potential obstruction to quantising $\sL$ over $\tilde{\sO}_Y$ lies in  the image of $\hbar^2\H^2\DR(Y)\brh\to \hbar^2\H^2\DR(X)\brh$ and depends only on $\tilde{\sO}_Y$.

%%worth noting we can recover that obstruction by looking at what happens when we vary the line bundle, as in {quanttorsorpropnonsd}. Effect of $\bG_m$-action on obstruction is the homotopy given by $+d\log a$, and then sheafifying gives obstruction map $c_1$-equivariant, in the sense that $\kappa(\sL\ten \sM)=c_1(\sL) + \kappa(\sM)$. 
\end{remark}

\subsection{Self-duality}\label{sdsn}

In order to eliminate the potential first order obstruction to quantising a generalised Lagrangian in Proposition \ref{prop3}, we now introduce the notion of self-duality, combining the ideas of \cite[\S \ref{DQvanish-sdsn}]{DQvanish} and \cite[\S \ref{DQnonneg-sdsn}]{DQnonneg}.

We  wish to consider line bundles $\sL$ on $X$ equipped with an anti-involutive equivalence $(-)^t \co \sD(\sL) \simeq \sD(\sL)^{\op}$. Such an equivalence is the same as a right $\sD$-module structure on $\sL^{\ten 2}$. Since a dualising line bundle  $K_{X}$ on $X$ naturally has the structure of a right $\sD$-module (see for instance \cite[\S 2.4]{GaitsgoryRozenblyumCrystal} for a proof in the derived setting), we will typically  take  $\sL$ to be a square root of $K_{X}$, when this exists. In this case, the equivalence $\sD(\sL) \simeq \sD(\sL)^{\op}$ comes from the equivalences $ \sL \simeq \sL^{\vee}$ and $\sD(\sE)^{\op} \simeq \sD(\sE^{\vee})$, where  $\sE^{\vee}:= \oR\hom_{\sO_X}(\sE,K_X)$.

\begin{definition}
Given a morphism $\phi \co A\to B$ of cofibrant stacky CDGAs  over $R$ and a strict line bundle $M$ over $B$, equipped with an anti-involution $(-)^t$ of $\cDiff_{B/R}(M)$,  we define an involution $(-)^*$ on  
the DGLA $Q\widehat{\Pol}(A,M;0)[1]$ by  
\[
 \Delta^*(\hbar):= i(\Delta)(-\hbar)^t, 
\]
for the brace algebra anti-involution  
\begin{align*}
 -i &\co (\C\C_{R,BD_1}(A, \cDiff_{B/R}(M))_{[1]} \rtimes \C\C_{R,BD_1}(A))^{\op}\\
&\to \C\C_{R,BD_1}(A, \cDiff_{B/R}(M)^{\op})_{[1]} \rtimes \C\C_{R,BD_1}(A)
\end{align*}
given by applying  Lemma \ref{involutiveHH} to the a.c. brace algebras $ \C\C_{R,BD_1}(A)$ and  $\C\C_{R, BD_1}(\C\C_{R,BD_1}(A, \cDiff_{B/R}(M)))$. 
\end{definition}

Since $(-)^*$ is a quasi-isomorphism of filtered DGLAs, it gives rise to an involutive weak equivalence 
\[
(-)^*\co Q\cP(A,M;0) \to Q\cP(A,M;0)
\]

\begin{definition}\label{selfdualdef}
For a line bundle $\sL$ on $X$ with a right $\sD$-module structure on $\sL^{\ten 2}$, 
we define
the space 
\[
 Q\cP(Y,\sL;0)^{sd}
\]
 of self-dual quantisations to be  the space of  homotopy fixed points of the  $\Z/2$-action  on $Q\cP(Y,\sL;0)$ generated by $(-)^*$.

 Similarly, we write $Q\widehat{\Pol}(A,M;0)^{sd}$ for the fixed points of $ Q\widehat{\Pol}(A,M;0)$ under the corresponding involution, so that 
 \[
 Q\cP(A,M;0)^{sd}\simeq \mmc( Q\widehat{\Pol}(A,M;0)^{sd}[1]).
 \]
 \end{definition}

\begin{lemma}\label{filtsd}
For the filtration $G$ induced on  $\tilde{F}^pQ\widehat{\Pol}(A,M;0)^{sd}$ by the corresponding filtration on $\tilde{F}^p Q\widehat{\Pol}(A,M;0)$, we have
\[
\gr_G^k  \tilde{F}^pQ\widehat{\Pol}(A,M;0)^{sd} \simeq \begin{cases}
                                                         \gr_G^k  \tilde{F}^pQ\widehat{\Pol}(A,M;0) & k \text{ even}\\
0 & k \text{ odd}.
                                                        \end{cases}
\]
\end{lemma}
\begin{proof}
This combines \cite[Lemma \ref{DQvanish-filtsd}]{DQvanish} and \cite[Lemma \ref{DQnonneg-filtsd}]{DQnonneg}. It follows because Lemma \ref{involutiveHH} ensures that the involution  acts trivially on $\gr_G^0Q\widehat{\Pol}(A,M;0)$. It therefore acts as multiplication by $(-1)^k$ on   $ \gr_G^kQ\widehat{\Pol}(A,M;0)= \hbar^k\gr_G^0Q\widehat{\Pol}(A,M;0)$.
\end{proof}

\begin{remark}\label{curvedsdrmk}
Following Remark \ref{algdrmk}, 
 a self-dual quantisation of $(X \xra{\phi}Y,\sL)$ gives rise to a 
curved $A_{\infty}$-deformation $\tilde{\O}_{Y}$ of $\hat{\Tot}\sO_{Y}$ over $R\llbracket \hbar \rrbracket$, equipped with an anti-involution $*$ which is semilinear under the transformation $\hbar \mapsto -\hbar$, together with a curved anti-involutive $A_{\infty}$-morphism $\phi^{-1}\tilde{\O}_{Y} \to \sD_{\sO_{X}/R}(\sL)\llbracket \hbar \rrbracket$.

More is true: by \cite[Proposition \ref{DQnonneg-Perprop}]{DQnonneg}, a quantisation gives a curved $A_{\infty}$ deformation of the  dg category $\per_{\dg}(\sO_{Y})$ of perfect complexes on $Y$, with self-dual quantisations incorporating a semilinear lift of the involution $\oR \hom_{\sO_{Y}}(-, \sO_{Y})$. A self-dual quantisation of the pair $(Y,\sL)$ thus gives a curved semilinearly involutive $A_{\infty}$-deformation of the involutive category  $\per_{\dg}(\sO_{Y})$ fibred over $\per_{\dg}(\sO_{X})$ via the functor
\begin{align*}
 (\per_{\dg}(\sO_{Y}),\oR \hom_{\sO_{Y}}(-, \sO_{Y})) &\to (\per_{\dg}(\sO_{X}),\oR \hom_{\sO_{X}}(-, \sL^{\ten 2}))\\
\sF &\mapsto \phi^*\sF \ten \sL,
\end{align*}
with an additional restriction of the  curvature of the deformation in terms of differential operators.

Adapting \cite[Remark \ref{DQnonneg-sdgerbermk}]{DQnonneg}, we can extend  the input data from the space $\oR\Gamma(X, B\bG_m)$ of line bundles to the space  $\oR\Gamma(Y, B^2\bG_m)\by^h_{\oR\Gamma(X, B^2\bG_m)} \{1\}$ of pairs $(\sG,\sL)$ with $\sG$  a $\bG_m$-gerbe on $Y$, and $\sL$  a trivialisation of $\phi^*\sG$. There is then a notion of self-dual quantisation for pairs $(\sG,\sL)$ with $\sG$ a $\mu_2$-gerbe and $\sL$ a trivialisation of the $\bG_m$-gerbe associated to $\phi^*\sG$, with a right $\sD$-module structure on the line bundle $\sL^{\ten 2}$. In particular, we may consider involutive quantisations of $(\per_{\dg}(\sO_{Y}),\oR \hom_{\sO_{Y}}(-, \sM))$ for any line bundle $\sM$, the criterion for self-duality now being that $\sL^{\ten 2} \ten \phi^*\sM$ be a right $\sD$-module, so that we consider the involution $\oR \hom_{\sO_{X}}(-,  \sL^{\ten 2}\ten \phi^*\sM)$ on $\per_{\dg}(\sO_{X})$. 

The natural example to take for $\sM$ is the dualising line bundle $K_{Y}= \det \bL^{Y}$ when $Y$ is virtually LCI, but when $X$ is Lagrangian, $\phi^*K_{Y}$ will  be trivial, so the resulting quantisations are quite similar. In any case, the $\bG_m$-actions on our filtered DGLAs are all unipotent, so extend to $\bG_m\ten_{\Z}\Q$-actions. Since $\mu_2\ten\Q=0$, this means  there are canonical equivalences between the  spaces of self-dual quantisations for varying $(\sG,\sL)$.%, depending only on the Chern class $c_1(\sL) \in \H^2F^1\DR(X/R)$.
\end{remark}

\begin{definition}
As in \cite[Remark \ref{DQnonneg-oddcoeffsrmk}]{DQnonneg}, write $t \in \GT(\Q)$ for the element which induces the anti-involution of Lemma \ref{involutiveHH}. We then denote by $\Levi_{\GT}^t$ the space of 
Levi decompositions $w$ of $\GT$ with $w(-1)=t$; these  form a torsor for the subgroup $(\GT^1)^t$ of $t$-invariants in the pro-unipotent radical $\GT^1$, and correspond to even Drinfeld associators.
 \end{definition}

\begin{definition}
Define  $G\Lag(Y,X;0)^{sd}$ to be the homotopy fixed points of the involution of $G\Lag(Y,X;0)$ given by $\hbar \mapsto -\hbar$. Explicitly, we set $G\Iso(A,B;0)^{sd}$ to be 
\[
 \mmc( \cone(F^2\DR(A/R)\to F^2\DR(B/R))) \by \prod_{i>0}  \mmc( \cone(\DR(A/R)\to \DR(B/R))\hbar^{2i}),  
\]
with $G\Lag(A,B;0)^{sd}$ the subspace of non-degenerate elements.
\end{definition}

\begin{theorem}\label{quantpropsd}
Take  a morphism $X \to Y$ of strongly quasi-compact Artin $N$-stacks over $R$, and  a line bundle $\sL$ on $X$ with a right $\sD$-module structure on $\sL^{\ten 2}$ (such as when $\sL$ is a square root of $K_{X}$). For any even associator $w \in \Levi_{\GT}^t(\Q)$, the induced map 
\[
  Q\cP(Y,\sL;0)^{\nondeg,sd} \to G\Lag(Y,X;0)^{sd}
\]
(from non-degenerate self-dual quantisations to generalised self-dual Lagrangians)
coming from Proposition \ref{prop3} is a weak equivalence.

In particular, $w$ associates a canonical choice of self-dual quantisation of  $(Y,\sL)$ to   every Lagrangian structure of $X$ over $Y$.
\end{theorem}
\begin{proof}
This is much the same as \cite[Proposition \ref{DQvanish-quantpropsd}]{DQvanish}. Lemma \ref{filtsd} implies that $w$ gives rise to weak equivalences
\begin{align*}
 Q\cP(Y,\sL;0)^{sd}/G^{2i} &\to Q\cP(Y,\sL;0)^{sd}/G^{2i-1}\\
Q\cP(Y,\sL;0)^{sd}/G^{2i+1} &\to (Q\cP(Y,\sL;0)^{sd}/G^{2i})\by^h_{(Q\cP(Y,\sL;0)/G^{2i})}(Q\cP(Y,\sL;0)/G^{2i+1}).
\end{align*}

Combined with Proposition  \ref{prop3}, these  give  weak equivalences from $Q\cP(Y,\sL;0)^{\nondeg,sd}/G^{2i+1}$ to
\[
 (Q\cP(Y,\sL;0)^{\nondeg,sd}/G^{2i})\by \mmc(\hbar^{2i} \cone(\DR(Y/R) \to \DR(X/R))
\]
for all $i>0$. Moreover, \cite[Remark \ref{DQnonneg-oddcoeffsrmk}]{DQnonneg} ensures that for our choice of Levi decomposition $w$, the map $\mu_w$ is equivariant under the involutions $*$, so these equivalences are just given by taking homotopy $\Z/2$-invariants.
The result then follows by induction, the base case holding because $*$ acts trivially on  $ Q\cP(Y,\sL;0)/G^1=\cP(Y,X;0) $, so $ Q\cP(Y,\sL;0)^{sd}/G^1\simeq \cP(Y,X;0)$.
\end{proof}

%%referred to  in {cfBGKP}:
% 
\begin{corollary}\label{quantcorsd}
Take  a $0$-shifted Lagrangian morphism $(X,\lambda) \to (Y, \omega)$ of strongly quasi-compact Artin $N$-stacks over $R$, and  a line bundle $\sL$ on $X$ with a right $\sD$-module structure on $\sL^{\ten 2}$ (such as when $\sL$ is a square root of $K_{X}$). Then for any  self-dual $E_1$-quantisation $\tilde{\sO}_Y$ of the symplectic structure, there exist self-dual quantised Lagrangian structures in $Q\cP(Y,\sL;0)^{\nondeg,sd}$ lifting $\tilde{\sO}_Y \in Q\cP(Y,0)^{\nondeg,sd}$ if and only if the class
\[
 [\mu_w(-, \tilde{\sO}_Y)^{-1}\sigma(\tilde{\sO}_Y) -\omega] \in \hbar^2\H^2\DR(Y/R)\brhh
\]
lies in the kernel of
\[
 \hbar^2\H^2\DR(Y/R)\brhh \to \hbar^2\H^2\DR(X/R)\brhh,
\]
where $w \in \Levi_{\GT}^t(\Q)$ is an even associator.
\end{corollary}
\begin{proof}
 Our  space of interest is the homotopy fibre of the canonical map
\[
 Q\cP(Y,\sL;0)^{\nondeg,sd} \to Q\cP(Y,0)^{\nondeg,sd}\by^h_{\Sp(Y,0)}\Lag(Y,X;0)
\]
over $(\tilde{\sO}_Y,\omega,\lambda)$, where $\Sp(Y,0)= \Lag(Y,\emptyset;0)$ is the space of $0$-shifted symplectic structures as in \cite{poisson,DQnonneg}.

Substituting in Theorem \ref{quantpropsd} twice  (with empty Lagrangian for $Q\cP(Y,0)$), this becomes
\[
 \prod_{i>0}  \mmc( \cone(\DR(Y/R)\to \DR(X/R))\hbar^{2i})\by_{ \mmc( \DR(Y/R)^{[1]}\hbar^{2i})}\{\mu_w^{-1}\sigma(\tilde{\sO}_Y) -\omega\},
\]
and exactness of the sequence $ \H^1\cone(\DR(Y/R)\to \DR(X/R))\to \H^2\DR(Y) \to \H^2\DR(X)$ completes the proof.
\end{proof}

%%and add remark saying that the proof also gives description of lower order deformations?

%\section{Further constructions}

\subsection{Quantisations of higher Lagrangians}\label{higherrmk}

Given a Lagrangian $(X, \lambda)$ with respect to an  $n$-shifted symplectic structure $(Y,\omega)$ for $n > 0$, we now discuss how the techniques of this paper should adapt to give a notion of quantised co-isotropic structures and to establish their existence. The broad picture is that a quantisation takes the form of an $E_{n+1}$-algebra deformation of $\sO_{Y}$ acting on an $E_n$-algebra deformation of $\sO_{X}$.

If we exploit  Koszul duality for $P_{n+1}$-algebras, we may replace the filtered Hochschild complexes of \S \ref{centresn} with Poisson coalgebra coderivations on bar complexes to give $P_{n+2}$-algebras of derived multiderivations acting on $P_{n+1}$-algebras (instead of brace algebras acting on associative algebras); the details of this construction are worked out in \cite{melanisafronovI}, via \cite[\S 3.1]{CalaqueWillwacher}.
Proposition \ref{compatcor2} then generalises to give a variant  proof of the equivalence between $n$-shifted Lagrangians and non-degenerate $n$-shifted co-isotropic  structures, announced by Costello and Rozenblyum and proved by Melani and Safronov \cite{melanisafronovII} with an approach explicitly based on part of the argument in this paper; 
\cite{melanisafronovII}  also established quantisations for $n$-shifted  co-isotropic structures for $n>1$ via formality of the $E_{n+1}$ operad. We  now sketch a parametrisation of quantisations for higher Lagrangians, including the case $n=1$ not addressed in \cite{melanisafronovII}. 
%Following \cite[Remark \ref{DQnonneg-nonfedosovrmk}]{DQnonneg}, these constructions might lead to parametrisations of degenerate $n$-shifted  co-isotropic structures. %% MS1 just works out co-isotropic strs as MC of relative polyvectors. MS2 globalises, gives Lagrangian equivce and quantn.

\subsubsection{Almost commutative $E_k$-algebras}

We begin with the notion of a $BD_k$-algebra as a higher analogue of an almost commutative algebra. There is a  filtration on  the Lie operad given by arity, inducing a filtration on the free Lie algebra generated by any filtered complex. Taking the universal enveloping $E_k$-algebra of this Lie algebra then gives a filtered $E_k$-algebra, and this construction corresponds to  a filtration on the $E_k$ operad. We can  then define the $BD_k$ operad to be the $E_k$ operad equipped with this completed filtration, for $k \ge 1$. 

Explicitly, $BD_1$ is just the operad defined in \cite[\S 3.5.1]{CPTVV}, whose algebras are almost commutative DGAAs. For $k\ge 2$, the operad $BD_k$ is just given by the re-indexed good truncation filtration $F^p BD_k= \tau_{\ge p(k-1)}E_k$ --- this agrees with \cite[\S 3.5.1]{CPTVV} for $k=2$, but differs by the reindexation for higher $k$. In particular, almost commutative brace algebras are equivalent to $BD_2$-algebras.

Informally, an $n$-shifted quantisation of a morphism $A \to B$ of CDGAs consists of a  $BD_{n+1}$-algebra deformation  $\tilde{A}$ of $A$ acting on a  $BD_n$-algebra deformation  $\tilde{B}$ of $B$ in a sense we will now  make precise, assuming an additivity conjecture for $BD_n$-algebras. An $n$-shifted quantisation of a morphism $A \to B$ of stacky CDGAs will be an $n$-shifted quantisation of  $\hat{\Tot}A \to \hat{\Tot}B$ with curvature and subject to additional boundedness constraints.

\subsubsection{Centres}

From now on, we refer to $BD_k$-algebras in complete filtered cochain chain complexes (simplicially localised at levelwise filtered quasi-isomorphisms) as stacky $BD_k$-algebras.  
Adapting \cite[Theorem  5.3.1.14]{lurieHigherAlgebra} from $\infty$-operads to the operads $BD_k$  in filtered chain complexes   will  
give a stacky $BD_k$-algebra 
\[
 \oR\C\C_{BD_k,R}(A,D)
\]
associated to any morphism $A \to D$ of stacky $BD_k$-algebras over $R$, universal  with the property that there is a $BD_k$-algebra morphism $\oR\C\C_{BD_k,R}(A,D)\ten_R^{\oL} A \to D$ in the associated $\infty$-category. Explicitly, these centres will be given by the higher order  Hochschild complexes of  \cite{ginotHigherOrderHH}  equipped with a PBW filtration. %%these have $\smile_i$-products, so probably $\cS_{k+1}$-algebras in the sense of \cite{McClureSmithMultivariable}.
The associated graded $\gr \oR\C\C_{BD_k,R}(A,D)$ is necessarily the centre of the morphism $\gr A \to \gr B$ of graded $P_k$-algebras, so is given by derived $P_k$ multiderivations from $\gr A$ to $\gr B$.

The universal property implies that $\oR\C\C_{BD_k,R}(A):=\oR\C\C_{BD_k,R}(A,A) $ is naturally an $E_1$-algebra in stacky $BD_k$-algebras, i.e. a stacky $E_1 \ten^{\oL}_{BV} BD_k$-algebra for the Boardman--Vogt tensor product $\ten_{BV}$. Moreover, for any morphism $A \to D$,  the centre $\oR\C\C_{BD_k,R}(A,D) $  then becomes an $\oR\C\C_{BD_k}(A)$-module in stacky $BD_k$-algebras.

For any morphism %%remember morphism fixed
$A_1 \by A_2 \to D$, the idempotents in the domain  give a decomposition $D= D_1 \by D_2$, and by universality for each morphism $A \to D$ we thus have 
%% $(R\by A) \ten C \to R \by D$ gives  $C \to R$ and $A \ten C \to D$, so OK.
\[
 \oR\C\C_{BD_k,R}(R\by A, R \by D) \simeq \oR\C\C_{BD_k,R}(R,R) \by \oR\C\C_{BD_k,R}(A,D) = R \by \oR\C\C_{BD_k,R}(A,D).
\]
The centre of $R \by A \to R \by D$ in the category of  augmented stacky  $BD_k$-algebras over $R$ is just
\[
 \oR\C\C_{BD_k,R}(R\by A, R \by D)\by_{  (R \by D)}^h R,
\]
so the reasoning above shows that
\[
 \oR\C\C_{BD_k,R,+}(A,D):= \oR\C\C_{BD_k,R}(A,D)\by^h_D 0
\]
is naturally a non-unital stacky  $BD_k$-algebra, with $\oR\C\C_{BD_k,R,+}(D)$ a non-unital stacky $E_1 \ten^{\oL}_{BV} BD_k$-algebra.

Adapting Lemma \ref{semidirectdef}, we then have:
\begin{definition}
Given a stacky $E_1 \ten^{\oL}_{BV} BD_k$-algebra $C$ over $R$ and a $C$-module  $E$ in stacky $BD_k$-algebras over $R$,
 we define $ E_{[1]} \rtimes C $ to be the non-unital stacky  $E_1 \ten^{\oL}_{BV}BD_k$-algebra
\[
 C\by^h_{ \oR\C\C_{BD_k,R}(E) }\oR\C\C_{BD_k,R,+}(E),
\]
the morphism $C \to \oR\C\C_{BD_k,R}(E)$ existing by universality.
\end{definition}

%%for interpretation, we'd like DGLA underlying $ E_{[1]} \rtimes C$ to be $\cone(C \to E)$, somehow. Easy enough: pass to underlying DGLAs, and then centre of $L$ is $\HHom(\beta L,L)$, in the form of curved $L_\infty$-derivations of $L$, so cone does obvkous thing.
%
%%However, can bypass that by saying hfibre of $\mmc( E_{[1]} \rtimes C) \to \mmc(C)$ over $\omega$ works out as $\mmc(E_{\omega})$ (shifts missing).

\subsubsection{Quantised $n$-shifted relative polyvectors for $n>0$}

Given a morphism $\phi \co A\to B$ of  stacky CDGAs  over $R$, now consider the non-unital $E_1 \ten^{\oL}_{BV} BD_{n+1}$-algebra 
\[
 (\C,F):= \oR\C\C_{BD_{n+1},R}(A, \oR\C\C_{ BD_n,R}(B) )_{[1]} \rtimes \oR\C\C_{BD_{n+1},R}(A) 
\]
in complete filtered cochain chain complexes. Definition \ref{QPoldef} then adapts verbatim to give a complex $Q\widehat{\Pol}(A,B;n) $ equipped with filtrations $\tilde{F}$ and $G$. 

Since we wish $Q\widehat{\Pol}(A,B;n)[n+1]$ to have the structure of a DGLA with $ [\tilde{F}^i,\tilde{F}^j] \subset \tilde{F}^{i+j-1}$ and $[G^i, G^j]\subset G^{i+j}$, and acting as derivations on the bifiltered  $E_1 \ten^{\oL}_{BV} BD_{n+1}$-algebra $\hbar Q\widehat{\Pol}(A,B;n)$, we need to know that $\oR\C\C_{BD_k}(A)$ has the structure of a $BD_{k+1}$-algebra. The analogous statement for $k=1$  is the content of Lemma \ref{HHaclemma2}. In general, the property would follow from the following conjecture:

\begin{conjecture}\label{additivityconj}
 For $k \ge 1$, the additivity isomorphism  $E_{k+1} \simeq E_1 \ten^{\oL}_{BV}E_k$ of \cite[Theorem  5.1.2.2]{lurieHigherAlgebra} induces a map $BD_{k+1} \simeq E_1 \ten^{\oL}_{BV} BD_k$ of operads in complete filtered chain complexes.
\end{conjecture}
Here, $\ten^{\oL}_{BV}$ denotes the derived Boardman--Vogt tensor product, so the conjecture amounts to saying that an $A_{\infty}$-algebra in $BD_k$-algebras is naturally a $BD_{k+1}$-algebra. On passing to associated graded complexes, the equivalence would give $P_{k+1} \to E_1 \ten^{\oL}_{BV} P_k$, which has been proved to be an equivalence by  Rozenblyum (unpublished, cf. \cite[\S 3.4]{CPTVV}) and independently by Safronov \cite{safronovBraces}; thus the map in the conjecture is necessarily an equivalence if it exists. A proof of Conjecture \ref{additivityconj} has also been announced by Rozenblyum (see \cite[comment after Conjecture 3.5.7]{CPTVV}).
For $k \ge 2$, the conjecture  would follow if  additivity is compatible with the action of the Grothendieck--Teichm\"uller group.

The conjecture would also ensure that the centres $\oR\C\C_{BD_k,R}(A,D)$ above all exist by appealing directly to \cite[Theorem  5.3.1.14]{lurieHigherAlgebra} for $k \ge 1$, regarding $BD_k$-algebras as $E_{k-1}$-algebras in $BD_1$-algebras.

The definitions of \S\S \ref{affinesn}, \ref{compatsn} would all then adapt, replacing $Q\widehat{\Pol}(A,M;0)$ with $Q\widehat{\Pol}(A,B;n)$ and taking appropriate shifts. The space $Q\cP(A,B;n)$ of $n$-shifted quantisations of the pair $(A,B)$ could then  just be defined as
\[
\mmc(\tilde{F}^2Q\widehat{\Pol}(A,B;n)[n+1]),
\]
elements of which give rise to curved $E_{n+1}$-algebra deformations of $\hat{\Tot}A$ acting on curved $E_n$-algebra deformations of $B$.

We may define the  space $G\Iso(A,B;n)$ of $n$-shifted generalised isotropic structures to  be
\[
\mmc( \tilde{F}^2\cone(\DR(A/R)\llbracket\hbar\rrbracket  \to \DR(B/R)\llbracket\hbar\rrbracket)[n]),
\]
and for each $w \in \Levi_{\GT}(\Q)$, Definition \ref{mudef}  adapts to give a compatibility map
\[
 \mu_w(-,\Delta) \co \cocone(\DR(A/R) \to \DR(B/R)\llbracket\hbar\rrbracket/\hbar^j \to T_{\Delta}Q\widehat{\Pol}_w(A;B,0)/G^j
\]
for each quantisation $\Delta$, with  Definition \ref{Qcompdef}  then adapting to give a space $Q\Comp_w(A,B;n)$ of compatible quantised pairs. 

Propositions \ref{QcompatP1}, \ref{compatcor2} and \ref{quantprop} would all carry over directly, in particular giving a map from $ Q\cP(A,B;n)^{\nondeg}$ to  $G\Lag(A,B;n)$, the non-degenerate locus in $G\Iso(A,B;n) $. The techniques of \S \ref{stacksn} would then extend these to global constructions for Artin $N$-stacks.

\subsubsection{Self-duality}

The functor $D \mapsto D^{\op}$ sending an almost commutative algebra to its opposite gives an involutive endofunctor of the category of $BD_1$-algebras, and hence of the categories of $E_1 \ten^{\oL}_{BV} BD_k$-algebras. The universal property of centres then gives an anti-involution
\[
 -i \co \oR\C\C_{BD_k,R}(A,D)^{\op} \to \oR\C\C_{BD_k,R}(A^{\op},D^{\op}),
\]
which in the $k=1$ case is the anti-involution $-i$ of Lemma \ref{involutiveHH}. Defining an anti-involutive  $E_1 \ten^{\oL}_{BV} BD_k$-algebra to be a homotopy fixed point of the involutive endofunctor $(-)^{\op}$, the anti-involution above makes $  \oR\C\C_{BD_k,R}(A,D)$ a stacky anti-involutive $BD_k$-algebra whenever $A$ and $D$ are stacky anti-involutive $BD_k$-algebras. In fact, this is necessarily the centre of $A \to D$ in the category of stacky anti-involutive $BD_k$-algebras --- the operad governing anti-involutive $BD_k$-algebras is $BD_k \circ  \Q.(\Z/2)$, regarding the algebra $\Q .(\Z/2)$ as an operad of arity $1$;  the distributivity transformation is given by  anti-involution.

As in \S \ref{sdsn}, we would then have  an involution $(-)^*$ on  
the (conjectural) DGLA $Q\widehat{\Pol}(A,D;0)[n+1]$ given by  
$
 \Delta^*(\hbar):= i(\Delta)(-\hbar)^t, 
$
and we could define $ Q\cP(A,B;n)^{sd}$ to be the fixed points of the resulting $\Z/2$-action, so its points give rise to involutive quantisations.

The proof of Theorem \ref{quantpropsd} will then adapt to give:
\begin{theorem}\label{higherquantpropsd}
Take  a morphism $X \to Y$ of strongly quasi-compact Artin $N$-stacks over $R$. If Conjecture \ref{additivityconj} holds, then for any even associator $w \in \Levi_{\GT}^t(\Q)$, the induced map 
\[
  Q\cP(Y,X;n)^{\nondeg,sd} \to G\Lag(Y,X;n)^{sd}
\]
(from non-degenerate self-dual quantisations to generalised self-dual Lagrangians)
 is a weak equivalence for all $n>0$.

In particular, $w$ associates a canonical choice of self-dual quantisation of  $(Y,X)$ to   every $n$-shifted Lagrangian structure of $X$ over $Y$.
\end{theorem}
For $n>1$, this has been proved without the self-duality conditions by \cite{melanisafronovII} after this paper was first written, by using a direct formality argument.  Their argument also implies these self-dual and curved statements. For $n=1$, existence of deformation quantisations for all co-isotropic structures when $Y$ is locally of finite presentation follows from the  $P_2$-algebra equivalence between  $p_w\C\C_{R,BD_1}(\sO_Y)$ and polyvectors  established more recently  in \cite{DQpoisson}.%\cite[Remarks \ref{DQpoisson-coisormk1} and \ref{DQpoisson-coisormk2}]{DQpoisson}.   

\begin{remark}[Twisted quantisations]
 
One significant difference between Theorems \ref{quantpropsd} and \ref{higherquantpropsd} is that the former incorporates the data of a line bundle. Similar input data are not essential for positively shifted quantisations because a  commutative algebra is canonically isomorphic to its opposite $E_1$-algebra, whereas  $\sO_{X}$ is not in general a right $\sD$-module.

However, by generalising Remark \ref{curvedsdrmk} we still expect a sensible notion of twisted quantisations for $n$-shifted Lagrangians, fibred over the space $\oR\Gamma(Y,B^{n+2}\bG_m)\by^h_{\oR\Gamma(X,B^{n+2}\bG_m) }\{1\}$ of  pairs $(\sG,\sL)$ with $\sG$ a $B^{n+1}\bG_m$-torsor on $Y$, and $\sL$ a trivialisation of $\phi^*\sG$ on $X$. Self-dual (i.e. involutive) quantisations would then be parametrised by $\oR\Gamma(Y,B^{n+2}\mu_2)\by^h_{\oR\Gamma(X,B^{n+2}\mu_2) }\{1\}$. Adapting \cite[Theorem 5.3.2.5]{lurieHigherAlgebra} from filtered $E_{n+2}$-algebras to $BD_{n+2}$-algebras would establish the required actions of $(n+2)$-groupoids $\ho\Lim_{i \in \Delta} B^{n+2}D^i(A)^{\by}$  generalising $\mathrm{TLB}$ from \S \ref{lbsn}.

However, since these spaces will come from unipotent group actions on quantised polyvectors, the actions of the torsion groups $B^{n+1}\mu_2(A),  B^{n+1}\mu_2(B)$ must be trivial, so  the spaces of twisted self-dual quantisations will be canonically equivalent as $(\sG,\sL)$ varies.
\end{remark}

\section{A ``Fukaya category'' for algebraic Lagrangians}\label{fukayasn}

In \cite[\S 5.3]{BehrendFantechiIntersections}, Behrend and Fantechi discussed the construction of a dg category whose objects are local systems on Lagrangian submanifolds of a complex symplectic variety. An extensive survey of related results is given in  \cite[Remark 6.15]{BBDJS}, where Joyce et al. discuss possible approaches to constructing such a ``Fukaya category'' with complexes of vanishing cycles as morphisms. 

One related construction not mentioned there is the $2$-category $\ddot{L}(T^{\vee}U)$ of \cite[\S 4]{KapustinRozansky3DTFT2}, which depends only on the holomorphic symplectic structure of the cotangent bundle $T^{\vee}U$. The $1$-categories of morphisms are given by matrix factorisation categories, so applying periodic cyclic homology as in \cite[Theorem 1.3]{efimovHCMF} would turn the $2$-category into a $\Z/2$-graded dg category over $\Cx((\hbar))$ with complexes of vanishing cycles as morphisms. It also follows from \cite[Theorem 1.4]{efimovHCMF} that applying negative cyclic homology would give a $\Cx\brh$ form of this category resembling a $\Z/2$-graded version of a part of our dg category $\cF$ in Definition \ref{fukayadef} below.

As described in \cite{joyceFukaya}, upgrading complexes of vanishing cycles to form a dg category presents a serious challenge,  
but on a complex symplectic manifold, \cite[Remark 6.15]{BBDJS} explains that   a likely candidate for the  subcategory of smooth Lagrangians is given  by the derived category of simple holonomic DQ modules for a DQ algebroid quantisation of the sheaf of analytic functions, by combining the results of  Kashiwara and Schapira \cite{kashiwaraschapira} (cf. \cite[\S 3.3]{schapiraMicrolocalSurvey}) with \cite{DAgnoloSchapira}. It is this approach which generalises naturally in our setting.

\subsection{Quantised intersections and internal $\Hom$s}\label{QIntHom}

Given a $BD_1$-algebra $\tilde{A}$ with right and left actions on $BD_0$-algebras $\tilde{B}$ and $\tilde{C}$, respectively, \cite[Proposition 5.8 and Theorem 5.10]{safronovPoissonRednCoisotropic} give a natural $BD_0$-algebra structure on $\tilde{B}\ten^{\oL}_{\tilde{A}}\tilde{C}$. Since $BD_1$-algebras acting on $BD_0$-algebras are a special case of our definition of quantised co-isotropic structures in Definition \ref{QPdef}, this can be interpreted for these cases as saying that the intersection of quantised $0$-shifted co-isotropic structures is a $(-1)$-shifted quantisation of the   intersection.

The purpose of this section is first to generalise this (Proposition \ref{tenQprop})  by working with both stacky CDGAs and non-trivial line bundles. We will then give an analogous result  (Proposition \ref{HomQprop}) for $\Hom$s instead of tensors.

\begin{lemma}\label{coderBVlemma}
Given a stacky CDGA $A$, regarded as an almost commutative DGAA with trivial filtration, any element $\phi \in \gamma_r\C\C_R(A,A)$ is a differential operator of order $\le r$ with respect to the shuffle multiplication of Definition \ref{bardef}, when  regarded as a coderivation on the bar construction $\b A$.  
\end{lemma}
\begin{proof}

To say that $\phi$ is has order $\le r$ is equivalent to vanishing of the  map $[\phi]_r \co (\b A)^{\ten r+2} \to \b A$ given by
 \begin{align*}
 a_0\ten \ldots \ten a_r\ten b \mapsto &[\ldots[[\phi,a_0],a_1]\ldots ,a_r](b),
                                &= \sum_{I \subset \{a_0, \ldots,a_r\}} (-1)^{|I|} (\prod_{j \notin I}a_j) \phi((\prod_{i \in I} a_i)b)
 \end{align*}
 where $a_i \in \b A$ is regarded as an element of  $\cHom_R(\b A,\b A)$ via the shuffle multiplication,  $[-,-]$ denotes the commutator in $\End_R(\b A)$, and $\prod$ is defined using the shuffle product with appropriate Koszul signs.

By construction, $[\phi]_r(a_0\ten \ldots \ten a_r\ten b )=0$ whenever any $a_i\in R$. Since $\b A = R \oplus \beta^1\b A$, it follows that  $(\b A)^{\ten (r+2)}$ is the sum of  $\beta^{r+1}((\b A)^{\ten (r+1)})\ten \b A$ and a subspace on which $[\phi]_r$ automatically vanishes. 

By definition of the filtration $\gamma_r$ in Definition \ref{HHdef0}, the coderivation $\phi$ sends  $\beta^j \b A$ to $\beta^{j+1-r}\b A$, for the filtration $\beta$ of Definition \ref{betadef}. Since shuffle multiplication preserves the filtration, it follows that $[\phi]_r$ sends $\beta^j( (\b A)^{\ten r+2})$ to $ \beta^{j+1-r}\b A$, and in particular the composite of $[\phi]_r \co \beta^{r+1} (\b A)^{\ten r+2}) \to \b A $ with the cogenerator map $\b A \to A_{[-1]}$ vanishes.

Combining the last two paragraphs, it follows that $[\phi]_r$ vanishes on cogenerators. If we denote the iterated shuffle multiplication by $  \nabla_{i} \co \b A^{\ten i} \to \b A$, then the maps $ \nabla_{i+1} \circ (\id^{\ten i} \ten \phi \circ \nabla_{j})$ are all $\nabla_{i+j}$-coderivations. In particular, this implies that $[\phi]_r$ is a $\nabla_{r+2}$-coderivation, being an alternating sum of such. Since  it vanishes on cogenerators, it must therefore be zero.
\end{proof}

\begin{lemma}\label{coconnBVlemma}
 Take  a stacky CDGA $A$, an $A$-module $M$, a coderivation $\phi \in \gamma_r\C\C_R(A,A)$ as in Lemma \ref{coderBVlemma}, and an element $\theta \in \gamma_r\C\C_R(A,\End_R(M))$ with respect to the trivial filtration on $M$. 
 Regarding $\theta$ as a map $(\b A)\ten_R M \to M$, the associated $\phi$-coconnection $\theta_{\phi} \co (\b A)\ten_R M \to (\b A)\ten_R M$, given by
 \[
  (\phi \ten \id_M) +  (\id_{\b A}\ten \theta) \circ (\mu_{\b A} \ten \id_M)
 \]
for the comultiplication $\mu_{\b A}\co \b A \to (\b A) \ten (\b A)$,
  has order $\le r$ with respect to the shuffle multiplication by elements of $\b A$. 
\end{lemma}
\begin{proof}
 This proceeds in exactly the same way as the proof of Lemma \ref{coderBVlemma}. It suffices to establish vanishing of the map
 $[\theta_{\phi}]_r \co (\b A)^{\ten r+2}\ten M \to (\b A)\ten M$ given by
 \begin{align*}
 a_0\ten \ldots \ten a_r\ten m \mapsto &[\ldots[[\theta_{\phi},a_0],a_1]\ldots ,a_r](m).
 \end{align*}
 This automatically vanishes whenever any $a_i \in R$, so it suffices to show that it vanishes on
 $\beta^{r+1}((\b A)^{\ten (r+1)})\ten (\b A)\ten M$.
 
 The conditions that $\phi,\theta \in \gamma_r$ then imply that 
 the composite of $[\theta_{\phi}]_r \co \beta^{r+1} (\b A)^{\ten r+2})\ten M \to (\b A) \ten M$ with the cogenerator map $(\b A)\ten M \to M$ vanishes, so $[\theta_{\phi}]_r$ vanishes on cogenerators. Since $[\theta_{\phi}]_r$ is a $(\nabla_{r+2},[\phi]_r)$-coconnection, it thus vanishes everywhere.
 \end{proof}

\begin{lemma}\label{coconnBVlemmaHom}
 Under the conditions of Lemma \ref{coconnBVlemma}, the $\phi^*$-connection $\theta_{\phi^*} \co \cHom_R(\b A, M) \to \cHom_R(\b A, M)$, given by
 \[
 f \mapsto  f \circ \phi + \theta \circ (\id_{\b A} \ten f)\circ \mu_{\b A},
 \]
% for the comultiplication $\mu_{\b A}\co \b A \to (\b A) \ten (\b A)$,
  has order $\le r$ with respect to the shuffle multiplication by elements of $\b A$.
\end{lemma}
\begin{proof}
 This works in exactly the same way as Lemma \ref{coconnBVlemma}. The question reduces to showing that the similarly defined commutator map $[\theta_{\phi^*}]_r \co \cHom_R((\b A)^{\ten r+2}, M) \to \cHom_R(\b A, M)$ vanishes, but this follows from the vanishing of $[\theta_{\phi}]_r$ in the proof of Lemma \ref{coconnBVlemma}.
 
\end{proof}

\begin{definition}\label{QPpairdef}
 Given  morphisms $C \la A\to B$ of stacky CDGAs  and strict line bundles $M$ and $N$ over $B$ and $C$ respectively, define the spaces $Q\cP(A,M,N;0)$ and $Q\cP(A,M^{\op},N;0)$ to be the homotopy fibre products
 \begin{align*}
  Q\cP(A,M,N;0)&:=Q\cP(A,M;0)\by^h_{Q\cP(A,0)}Q\cP(A,N;0)\\
  Q\cP(A,M^{\op},N;0)&:=Q\cP(A,M;0)\by^h_{i,Q\cP(A,0)}Q\cP(A,N;0),
 \end{align*}
where $i$ is the involution of Lemma \ref{involutiveHH}, which sends a quantisation $\tilde{A}$ of $A$ to the opposite $BD_1$-algebra $\tilde{A}^{\op}$.
 \end{definition}
In other words, elements of $Q\cP(A,M,N;0)$ consist of quantised co-isotropic structures on the pairs $(A,M)$ and $(A,N)$, with the same underlying    $0$-shifted quantisation   $\tilde{A}$ of $A$. On the other hand, elements of $Q\cP(A,M^{\op},N;0)$ consist of quantised co-isotropic structures on the pairs $(A,M)$ and $(A,N)$, but with opposite underlying    $0$-shifted quantisations  $\tilde{A}^{\op}$ and $\tilde{A}$ of $A$. Thus in $Q\cP(A,M,N;0)$, both $M$ and $N$ are being deformed as certain left $\tilde{A}$-modules, while in $Q\cP(A,M^{\op},N;0)$ we are deforming $M$ as a right  $\tilde{A}$-module and $N$ as a left  $\tilde{A}$-module.

%%would like to work out self-dual for the following, but that becomes tricky. We need right $\sD$-mod str $T_{\b A} \to \b A$. Does the choice $\Co\Lie(A[\pm 1])$ suffice? On $\Symm(V)$, things should be OK, with $V^*\ten \Symm(V)\to \Sym(V)$ given  by contraction. SElf-duality for $\nabla_A$, on the other hand, just says that it commutes with an involution $*$ of $\b A$ (which does preserve the commutative product, and doesn't seem to be identity on $\Co\Lie(A)$, since it's reversing sign of the cobracket --- no, commutative alg has trivial anti-involution, so the involution on $\b A$ will have to give $-1$ on $\Co\Lie$, if you look at {involutiveHH}. MAYBE we shd think of this as $B \ten V^*\ten \Symm(V)\ten C$, so worry about $\det$ pulled back there (no, no section of $\Co\Ass \to \Co\Lie$.
%
%%Or could just register a protest in response, saying I assumed additivity conjecture would subsume much of this, and that hopefully this will provide motivation for a systematic study of self-duality and of genuine $\sD$-moduels in the derived context (the last  Beraldo seems just to have started doing).

\begin{proposition}\label{tenQprop}
 Given  morphisms $C \la A\to B$ of stacky CDGAs  and strict line bundles $M$ and $N$ over $B$ and $C$ respectively, there is a natural derived tensor product construction
 \[
  Q\cP(A,M^{\op},N;0)\to Q\cP(M\ten_A^{\oL}N,-1)
 \]
to the space $Q\cP(M\ten_A^{\oL}N,-1)= Q\cP(R, M\ten_A^{\oL}N;0)$ of $(-1)$-shifted quantised Poisson structures on the line bundle $M\ten_A^{\oL}N$ over $B\ten_A^{\oL}C$, as in \cite[Definition \ref{DQvanish-Qpoissdef}]{DQvanish}.
\end{proposition}
\begin{proof}
 We adapt the approach of \cite[Proposition 5.8 and Theorem 5.10]{safronovPoissonRednCoisotropic}. We will deform the Hochschild homology complex 
 \[
  \C\C^R(A,N\ten_RM)_{\#}:=\bigoplus_n (M\ten_R A^{\ten_R n} \ten_RN)_{[-n]}, 
 \]
which has chain differential $\delta \pm b$, for the 
Hochschild differential
\begin{align*}
b(m,a_1, \ldots , a_r,n) = &(ma_1,a_2, \ldots, a_r,n)\\
 &+ \sum_{i=1}^{r-1}(-1)^i (m,a_1, \ldots, a_{i-1}, a_ia_{i+1}, a_{i+2}, \ldots, a_r,n)\\
&+ (-1)^r (m,a_1, \ldots, a_{r-1},a_rn).
\end{align*} 
Note that we may write $\C\C^R(A,N\ten_RM)_{\#}= (M\ten_R\b A \ten_RN)_{\#}$ for the bar construction $\b$ of Definition \ref{bardef}, and regard this as the cofree left $\b A$-comodule cogenerated by $N\ten_RM$. Then  $b$ is the  $\b A$-coderivation given on cogenerators by the difference of the multiplication maps $M\ten_R A \ten_RN \to M\ten_RN$.  

There is a graded-commutative multiplication on $\C\C^R(A,C\ten_RB)$ given by combining those on $B$ and $C$ with the shuffle multiplication $\nabla$ on $\b A$ from Definition \ref{betadef}. This makes $\C\C^R(A,C\ten_RB)$ a model for the stacky CDGA $B\ten_A^{\oL}C$, and similarly the $\C\C^R(A,C\ten_RB)$-module
  $\C\C^R(A,N\ten_RM)$ is a strict line bundle, and a model for the   $M\ten_A^{\oL}N$. %% Shuffle product on $\b A$ given on cogenerators by the identity maps $(A\ten R) \oplus (R\ten A) \to A$ 

Since the map $Q\widehat{\Pol}_R(A,N;0) \to Q\widehat{\Pol}_R(A,0;0)$ of filtered DGLAs is surjective, the map $Q\cP(A,N;0) \to Q\cP(A,0)$ is a fibration, so a model for   $Q\cP(A,M^{\op},N;0)$ is given by the fibre product $Q\cP(A,M;0)\by_{i,Q\cP(A,0)}Q\cP(A,N;0)$.
Given an element $(\Delta_A,\Delta_M,\Delta_N) \in Q\cP(A,M^{\op},N;0) $, we construct an operator  $ (\Delta_{M})_{\Delta_{A}}\mp(\Delta_{N})_{\Delta_{A}}$ on $\C\C^R(A,N\ten_RM)\brh $
 by first extending $\Delta_A$ to a coderivation on $\b A$ as in Lemma \ref{coderBVlemma}, then constructing  coconnections associated to $ \Delta_{M}, \Delta_{N}$ as in Lemma \ref{coconnBVlemma}.

 It suffices to show that $ (\Delta_{M})_{\Delta_{A}}\mp(\Delta_{N})_{\Delta_{A}}$ is a differential operator in $\prod_{i \ge 1} \hbar^iF_{i+1}\sD_{\C\C^R(A,C\ten_RB)}(\C\C^R(A,N\ten_RM))$, since it then defines an element of  $Q\cP(M\ten_A^{\oL}N,-1)$. Equivalently, for arbitrary elements $x_i \in \C\C^R(A,C\ten_RB)$, this says that we want the commutator $[\ldots[[(\Delta_{M})_{\Delta_{A}}\mp(\Delta_{N})_{\Delta_{A}} ,x_1],x_2]\ldots ,x_r]$ to be  divisible by $\hbar^{r-1}$.
 
 Since $\Delta_A \in \prod_{i\ge 1} \hbar^i\gamma_{i+1}$ % Lemma \ref{coderBVlemma} implies that the associated coderivation lies in $\prod_{i\ge 1} \hbar^iF_{i+1}\sD_{\b A}$. since 
 and $\Delta_M, \Delta_N \in \prod_{i\ge 1} \hbar^i(\gamma F)_{i+1} \subset \prod_{i\ge 1} \hbar^i\gamma_{i+1}$,  working modulo $\hbar^{r-1}$ these all lie in $\gamma_{r-1}$,  so
 Lemma \ref{coconnBVlemma}  implies that     the operator $(\Delta_{M})_{\Delta_{A}}\mp(\Delta_{N})_{\Delta_{A}} $ has order $\le r-1$ with respect to the shuffle multiplication by elements of $\b A$, giving  the  commutator above  the required property whenever the elements $x_i$ all lie in $\b A$.
 
 Finally, for $y \in B\ten C$ we have $[\Delta_A,y]=0$, so  the commutator $[(\Delta_{M})_{\Delta_{A}}\mp(\Delta_{N})_{\Delta_{A}} ,y]$ is $\b A$-linear. 
 Moreover, since the filtration $F$ is almost commutative,  for $y_i \in B\ten C$ we have $[\ldots[[(\Delta_{M})_{\Delta_{A}}\mp(\Delta_{N})_{\Delta_{A}} ,y_1],y_2]\ldots ,y_r] \in \prod_{i\ge 1} \hbar^i(\gamma F)_{i+1-r} \subset \prod_{i\ge r-1} \hbar^i\gamma_{i+1-r}$, 
 so applying Lemma \ref{coconnBVlemma} modulo $\hbar^{r+s-1}$ with the trivial coderivation shows that for $x_j \in \b A$, the commutator 
 \[
  [\ldots[[[\ldots[[(\Delta_{M})_{\Delta_{A}}\mp(\Delta_{N})_{\Delta_{A}} ,y_1],y_2]\ldots ,y_r] ,x_1],x_2]\ldots ,x_s]
 \]
 is divisible by $\hbar^{r+s-1}$, as required.
 \end{proof}
 
\begin{remarks}\label{compattenQrmk}
It is natural to ask how this intersection construction for quantisations relates to the natural constructions of generalised symplectic structures on Lagrangian intersections. Even for unquantised shifted Poisson intersections in \cite[\S\S 3 and 4]{melanisafronovII} this is not spelt out, but we expect that it should be possible to formulate compatibility using a map with a target DGAA related to   the brace tensor product $\CCC(A, \sD_B)\ten^{\oL}_{\CCC(A)}\CCC(A,\sD_C) $. 
 
However, the relation of intersections  with self-duality will be much more subtle, because the notion of self-duality depends on choices of line bundles with  right $\sD$-module structures. In virtually LCI settings, where the dualising complex is a line bundle, we expect the intersection of self-dual quantised $0$-shifted co-isotropic structures will give a self-dual $E_0$-quantisation via Proposition \ref{tenQprop}.
\end{remarks}
%%Real worry of how these additivity results relate to Lagrangian/co-isotropic correspondence. Above only seems to be giving us a DGLA map $\C\C(A)\to \cD(\b A)$. Thinking about comptibility, we should get $\mu\co \DR(B)\ten^{\oL}_{\DR(A)}\DR(C) \to \CCC(\CCC(A, \sD_B))\ten^{\oL}_{\CCC(A)}\CCC(\CCC(A,\sD_C))$, and would somehow have to relate that to $\sD_{B\ten^{\oL}_AC}$. What's $\CCC(A, \sD_B)\ten^{\oL}_{\CCC(A)}\CCC(A,\sD_C)$? Looks about right for $\sD_{B\ten^{\oL}_AC}$, because generators sth like $\T_{B/A} \oplus T_{C/A} \oplus T_A[1][-1]$.

Proposition \ref{tenQprop} has the following generalisation, with much the same proof, just with an additional $A_{\infty}$-action by $(A_2\brh, \Delta_{A_2})$ to incorporate:
 \begin{proposition}\label{tenQcorrprop}
 Given  morphisms $A_1\to B$ and $A_1\ten A_2 \to C$ of stacky CDGAs  and strict line bundles $M$ and $N$ over $B$ and $C$ respectively, there is a natural derived tensor product construction
 \[
  (Q\cP(A_1,M;0) \by Q\cP(A_2,0) )\by_{(i\ten \id),Q\cP(A_1\ten A_2,0)} Q\cP(A_1\ten A_2,N;0)\to  Q\cP(A_2,M\ten_{A_1}^{\oL}N;-1)
 \]
to the space of $0$-shifted quantised co-isotropic structures on the line bundle $M\ten_{A_1}^{\oL}N$ on $B\ten_{A_1}^{\oL}C$ over $A_2$.
\end{proposition}

\begin{proposition}\label{HomQprop}
 Given  morphisms $C \la A\to B$ of stacky CDGAs  and strict line bundles $M$ and $N$ over $B$ and $C$ respectively, there is a natural derived $\Hom$ construction 
 from
$  
Q\cP(A,M,N;0) %\to Q\cP(M\ten_A^{\oL}N,-1)
 $
to the space of $R$-linear deformations of $\oR\cHom_A(M,N)$ given by differential operators
\[
 \Delta \in \prod_{i \ge 1} \hbar^iF_{i+1} \hat{\Tot}\sD_{B\ten^{\oL}_AC}( \oR\cHom_A(M,N)).
\]

In particular, if $\oR\cHom_A(M,N)$ is an invertible $B\ten^{\oL}_AC$-module, this gives  a map 
\[
  Q\cP(A,M,N;0)\to Q\cP(\oR\cHom_A(M,N),-1)
 \]
to the space of $(-1)$-shifted quantised Poisson structures on the line bundle $\oR\cHom_A(M,N)$ over $B\ten_A^{\oL}C$.
%%might want to say about $\lambda_1-\lambda_2$ for assoc symplectic.
\end{proposition}
\begin{proof}
The construction arises by sending $(\Delta_A,\Delta_M,\Delta_N) \in Q\cP(A,M,N;0) $, to the differential operator  $ ((\Delta_{M})_{\Delta_{A}})^*\mp(\Delta_{N})_{\Delta_{A}^*}$ on $\cHom_R(M\ten \b A,N)\brh$ defined analogously to  Proposition \ref{tenQprop}. Explicitly, $ (\Delta_{M})_{\Delta_{A}}$ on $M \ten \b A\brh$ is  given by Lemma \ref{coconnBVlemma}, and then pre-composition yields an operator $((\Delta_{M})_{\Delta_{A}})^*$ on $\cHom_{R\brh}(M\ten \b A\brh,N\brh) \cong\cHom_R(M\ten \b A,N) $. The operator $(\Delta_{N})_{\Delta_{A}^*}$ is given by Lemma \ref{coconnBVlemmaHom}.
\end{proof}

\subsection{DQ modules associated to quantised Lagrangians}

Since we are working algebraically rather than analytically, our analogue of a DQ module is 
simply an $\hbar$-adically complete module over a fixed quantisation $\tilde{\sO}_{Y}$  of ${\sO}_{Y}$. When $Y$ is a derived DM  stack, we can interpret $\tilde{\sO}_{Y}$ as an $A_{\infty}$-algebroid deformation of $\sO_Y$ on the \'etale site of $Y$ as in Example \ref{algdex}.(\ref{algdexDM}), and DQ modules are then objects of its derived dg category $\oR\Lim_i\sD_{dg}(\tilde{\sO}_{Y}/\hbar^i)$. 
When $Y$ is a derived Artin stack, the deformation $\tilde{\sO}_{Y}$ is defined on a site of stacky CDGAs and may incorporate curvature, so we have to be a little more careful. Essentially, we take a DQ-module to be a module for the curved $A_{\infty}$-algebra   $\tilde{\sO}_{Y}$,  but there are boundedness conditions coming from $\hat{\Tot}$ as in Remark \ref{curvedrmk}. 

For simplicity, we now just describe the full dg subcategory of DQ modules coming from quantised Lagrangians. The idea is that for a line bundle $\sL$ on a %relatively affine 
derived Lagrangian $\phi\co X \to Y$,  each deformation quantisation $(\tilde{\sO}_{Y}, \Delta_{\sL})$ of $(\sO_{Y}, \sL)$ gives rise to such a DQ %bimodule as follows.
 module as $\oR \phi_*\tilde{\sL}$, where as in Remark \ref{algdrmk},
\[
 \tilde{\sL} := (\sL\brh, \delta + \Delta_{\sL} \cdot-) %%NB not a right $\O_X$-mod because $\sD_X$ doesn't act $\O_X$-linearly. 
\]
is the left $\phi^{-1}\tilde{\sO}_{Y}$-module associated to the quantisation $\Delta_{\sL}$. 

    \begin{definition}\label{QCIdgdefglobal}
Given an $E_1$-quantisation $\tilde{\sO}_Y \in Q\cP(Y,0)$ of a derived Artin $n$-stack $Y$ over $R$, we define the $R\brh$-linear dg category $Q\C\cI_{\dg}(\tilde{\sO}_Y)$  of quantised co-isotropic structures over $\tilde{\sO}_Y$ as follows. 

There is an object for each quantised co-isotropic structure $(\tilde{\sO}_Y, \tilde{\sL}) \in Q\cP(Y,\sL;0)\by^h_{Q\cP(Y,0)}\{\tilde{\sO}_Y\}$  over $\tilde{\sO}_Y$, for each  line bundle $\sL$ on each derived Artin stack $\phi \co X \to Y$ representable over $Y$ by quasi-compact quasi-separated derived algebraic spaces. 

  Given objects $\tilde{\sL}_1, \tilde{\sL}_2$, we set the  associated $\Hom$-complex  
$
 \oR\hat{\HHom}_{\tilde{\sO}_Y}(\oR \phi_{1*}\tilde{\sL}_1,\oR \phi_{2*}\tilde{\sL}_2)
 $
 to be the homotopy limit, taken over all %objects  $\Spec B_1 \to \Spec A \la \Spec B_2$ in $(dg_+DG\Aff_{\et}^{[1]}\da X_1/Y)\by^h_{dg_+DG\Aff_{\et}\da Y}(dg_+DG\Aff_{\et}^{[1]}\da X_2/Y) $, %%no: we just want sth over the base.
 homotopy formally \'etale morphisms $f \co \Spec DA \to Y$ (i.e. $f \in \ho \Lim_i X(D^iA)$) from  stacky CDGAs $A$,
  of the complexes
 \[
\hat{\Tot}(\cHom_R( f^*(\oR\phi_{1*}\sL_1) \ten \b A,f^*(\oR\phi_{2*}\sL_2))\brh, \delta + ((\Delta_{\sL_1})_{\Delta_{A}})^*\mp(\Delta_{\sL_2})_{\Delta_{A}^*})
\]
with notation for differentials as in the proofs of Propositions \ref{tenQprop}, and \ref{HomQprop}.
 
The associative composition law
 \begin{align*}
 &\oR\hat{\HHom}_{\tilde{\sO}_Y}(\oR \phi_{1*}\tilde{\sL}_1,\oR \phi_{2*}\tilde{\sL}_2)\ten_R \oR\hat{\HHom}_{\tilde{\sO}_Y}(\oR \phi_{2*}\tilde{\sL}_2,\oR \phi_{3*}\tilde{\sL}_3)\\
 &\to  \oR\hat{\HHom}_{\tilde{\sO}_Y}(\oR \phi_{1*}\tilde{\sL}_1,\oR \phi_{3*}\tilde{\sL}_3)
 \end{align*}
 arises naturally on 
rewriting each $\cHom_R(M\ten \b A,N)\brh$ symmetrically as the double complex $\cHom_{\b A}(M\ten \b A,N\ten \b A)\brh$ of $\b A$-colinear maps.
\end{definition}
  
\begin{remarks}\label{QCIredrmk}
 Similar reasoning to \cite[Proposition \ref{DQnonneg-Perprop}]{DQnonneg} shows that if we reduce modulo $\hbar$, the resulting $R$-linear dg category $Q\C\cI_{\dg}(\tilde{\sO}_Y)/\hbar$ is quasi-equivalent to a full dg subcategory of the derived category $\sD_{dg}(\sO_Y)$ of quasi-coherent complexes on $Y$. Its  objects are complexes of the form $\oR \phi_*\sL$, for line bundles $\sL$ on derived stacks over $Y$ for which the data $(\phi \co X \to Y, \sL)$ admits a quantised co-isotropic structure lifting $\tilde{\sO}_Y$.   

The hypothesis in Definition \ref{QCIdgdefglobal} that $\phi$ be representable is stronger than strictly necessary. All we really need is that $\oR \phi_*$ preserves quasi-coherence and commutes with derived base change.  With that modification, the representability hypothesis on the correspondence in  Proposition \ref{Lagcorrprop} below can be relaxed accordingly.
\end{remarks}

\begin{definition}\label{fukayadef}
Fix a  non-degenerate involutive quantisation $\tilde{\O}_{Y} \in Q\cP(Y,0)^{\nondeg,sd}$ quantising a symplectic structure $\omega \in \H^2F^2\DR(Y/R)$, and assume that $\tilde{\O}_{Y}$ is $w$-compatible with $\omega \cdot a$ for some $w \in  \Levi_{\GT}^t(\Q)$ and $a \in \H^0\DR(Y/R)\brhh$.

Now define the dg category $\cF(\tilde{\O}_{Y})$ to be the full subcategory of  $Q\C\cI_{\dg}(\tilde{\sO}_Y)$ consisting of  self-dual 
quantised Lagrangian structures $\tilde{\sL}$ on line bundles  $\sL$  with a given right $\sD$-module structure on $\sL^{\ten 2}$. 
\end{definition}

\begin{remarks}
The self-duality hypotheses and  the condition that $\tilde{\O}_{Y}$ is $w$-compatible with $\omega \cdot a$ for some $w \in  \Levi_{\GT}^t(\Q)$ and $a \in \H^0\DR(Y/R)\brhh$ ensure via Corollary \ref{quantcorsd} that $\cF(\tilde{\O}_{Y})$ has objects for  every self-dual line bundle $\sL$ on a Lagrangian $(X,\lambda)$ over $(Y, \omega)$. One such quantisation will correspond to the generalised Lagrangian $(\omega\cdot a, \lambda\cdot a)$.
% % That condition,  which is probably independent of $w$, 

Thus Corollary \ref{quantcorsd} plays an analogous role in our setting to that played classically by \cite{DAgnoloSchapira}, which shows that there exists a simple holonomic DQ-module supported on any smooth closed complex Lagrangian equipped with a square root of the dualising bundle. As explained in \cite[Remark 6.15]{BBDJS}, the  DQ-modules of  \cite{DAgnoloSchapira} are expected to provide the objects of the complex Fukaya category in the smooth underived setting.
\end{remarks}

Since we are permitting all derived Lagrangians to give rise to elements of $\cF(\tilde{\O}_{Y})$, we cannot expect all morphisms in our  dg category $\cF(\tilde{\O}_{Y})$ to be related to vanishing cycles as in the dg category conjectured in \cite{BBDJS,joyceFukaya}. However, when Grothendieck--Verdier duality applies (such as for finite virtually LCI morphisms $X_i \to Y$),   we will now relate the $\Hom$-complexes to $(-1)$-shifted quantisations, which in turn relate to vanishing cycles as in \cite{DQvanish}.

\begin{proposition}\label{cfvanish1}
In the setting of Definition \ref{QCIdgdefglobal},  if  $\phi_1$ is proper and virtually LCI of relative dimension $d$, 
then the complex 
$\oR\hat{\HHom}_{\tilde{\sO}_Y}(\oR \phi_{1*}\tilde{\sL}_1,\oR \phi_{2*}\tilde{\sL}_2)[d]$
 is given by derived global sections of an $E_0$-deformation quantisation (see \cite{DQvanish}) of the %%note that $K_{X_1} \simeq K_{X_1/Y}$
 line bundle
\[
 \oL\phi_2^*(\sL_1^{\phantom{1}-1}\ten^{\oL} \det \oL\Omega^1_{X_1/Y})\ten^{\oL} \oL\phi_1^*\sL_2
\]
 on the derived intersection $X_1\by^h_{Y}X_2$. 
 \end{proposition}
 \begin{proof}%[Proof (sketch).]
By definition, $\oR\hat{\HHom}_{\tilde{\sO}_Y}(\oR \phi_{1*}\tilde{\sL}_1,\oR \phi_{2*}\tilde{\sL}_2)[d]$ is given by taking the homotopy limit over all 
homotopy formally \'etale stacky CDGAs $A$  over $Y$ (i.e. $f \co \Spec DA \to Y$) of  $\hat{\Tot}$-complexes of the double complexes
\[
C_A:= (\C\C_R(A, \cHom_R(f^*(\oR\phi_{1*}\sL_1),f^*(\oR\phi_{2*}\sL_2)))\brh,  \delta +((\Delta_{\sL_1})_{\Delta_{A}})^*\mp(\Delta_{\sL_2})_{\Delta_{A}^*}  ).
\]
We may rewrite 
\[
\C\C_R(A, \cHom_R(f^*(\oR\phi_{1*}\sL_1),f^*(\oR\phi_{2*}\sL_2)))\simeq \oR\phi_{2*}\C\C_R(\phi_2^{-1}A, \cHom_R(\phi_2^{-1}f^*\oR\phi_{1*}\sL_1,f^*\sL_2)),
 \]
 where we  use the same notation for a map  and its pullbacks. When $\phi_1$ is affine (or equivalently finite), the result now follows as a direct consequence of Proposition \ref{HomQprop} and Verdier duality, but we must work much harder in general.

If we consider the natural morphisms  
\begin{align*}
\oR\cHom_{\phi_2^{-1}A}(\phi_2^{-1}f^*\oR\phi_{1*}\sL_1,f^*\sL_2) &\to \C\C_R(\phi_2^{-1}A, \oR\cDiff_{\phi_2^{-1}A}(\phi_2^{-1}f^*\oR\phi_{1*}\sL_1,f^*\sL_2))\\ 
&\to  \C\C_R(\phi_2^{-1}A, \cHom_R(\phi_2^{-1}f^*\oR\phi_{1*}\sL_1,f^*\sL_2)),
\end{align*}
then the first is a levelwise quasi-isomorphism as in Remark \ref{extremecasesrmk} (replacing both $A$ and $B$ with $\phi_2^{-1}A$), and the composite is a levelwise quasi-isomorphism essentially by definition of the Hochschild complex. Thus the second map, coming from the inclusion of differential operators in $R$-linear maps, is also a levelwise quasi-isomorphism.
Since the operators $\Delta_{\sL_i}$ are differential operators, their action restricts to the second double complex, giving us a levelwise quasi-isomorphism
\[
 (\oR\phi_{2*}\C\C_R(\phi_2^{-1}A, \oR\cDiff_{\phi_2^{-1}A}(\phi_2^{-1}f^*\oR\phi_{1*}\sL_1,f^*\sL_2))\brh,  \delta +((\Delta_{\sL_1})_{\Delta_{A}})^*\mp(\Delta_{\sL_2})_{\Delta_{A}^*}  ) \to C_A.
\]

Now, 
\[
 \oR\cDiff_{\phi_2^{-1}A}(\phi_2^{-1}f^*\oR\phi_{1*}\sL_1,f^*\sL_2)\simeq \oR\cHom_{\phi_2^{-1}A}(\phi_2^{-1}f^*\oR\phi_{1*}\sL_1,\cDiff_{\phi_2^{-1}A}(\phi_2^{-1}A,f^*\sL_2)),
\]
 so
\begin{align*}
 &\oR\phi_{2*}\C\C_R(\phi_2^{-1}A, \oR\cDiff_{\phi_2^{-1}A}(\phi_2^{-1}f^*\oR\phi_{1*}\sL_1,f^*\sL_2))\simeq \\
 &\C\C_R(A, \oR\cHom_A(f^*\oR\phi_{1*}\sL_1,  \oR\phi_{2*}\cDiff_{\phi_2^{-1}A}(\phi_2^{-1}A,f^*\sL_2))).
\end{align*}
We can now use Grothendieck--Verdier duality to rewrite this as 
\begin{align*}
 &\C\C_R(A, \oR \phi_{1*}\oR\cHom_{f^*\sO_{X_1}}(f^*\sL_1,  \phi_1^!\oR\phi_{2*}\cDiff_{\phi_2^{-1}A}(\phi_2^{-1}A,f^*\sL_2))) \simeq\\ &\C\C_R(A, \oR \phi_{1*}(\phi_1^!\oR\phi_{2*}\cDiff_{\phi_2^{-1}A}(\phi_2^{-1}A,f^*\sL_2)\ten_{\sO_{X_1}}^{\oL} \sL_1^{-1}));
\end{align*}
importantly for us, $\phi^!$ preserves right $\sD$-module structures via the constructions of \cite{GaitsgoryRozenblyumCrystal} and \cite[Example \ref{DQvanish-stratDAex}]{DQvanish}, so our coefficients of $\C\C_R(A,-)$ above are a right $\sD_{X_1}(\sL_1)$-module, as well as inheriting a left $\sD_{X_2}(\sL_2)$-module structure from $ \cDiff_{\phi_2^{-1}\sO_Y}(\phi_2^{-1}\sO_Y,\sL_2)$. Thus the operator $\Delta_{\sL_1}$ acts on this on the right, while $\Delta_{\sL_2}$ acts on the left; this step is our reason for having to introduce $\cDiff$ in the proof. 

Writing $\psi \co X_1\by^h_YX_2 \to Y$ for the canonical map, we may rearrange this double complex to write $C_A$ as  
\[
(\oR\psi_*\C\C_R(\psi^{-1}A, (\phi_1^!\cDiff_{\phi_2^{-1}A}(\phi_2^{-1}A,f^*\sL_2)\ten_{\phi_2^{-1}\sO_{X_1}} \phi_2^{-1}\sL_1^{-1}))\brh,  \delta +((\Delta_{\sL_1})_{\Delta_{A}})^*\mp(\Delta_{\sL_2})_{\Delta_{A}^*}  ).
\]

The argument of Propositions \ref{tenQprop} and \ref{HomQprop} adapts to show that this deformation is given by differential operators of the correct orders, and it only remains to show that the complex
\[
 \CCC_R(\psi^{-1}\sO_Y, (\phi_1^!\cDiff_{\phi_2^{-1}\sO_Y}(\phi_2^{-1}\sO_Y,\sL_2)\ten_{\phi_2^{-1}\sO_{X_1}}^{\oL} \phi_2^{-1}\sL_1^{-1}))
\]
it deforms is quasi-isomorphic to a shift of a  line bundle on $X_1\by^h_YX_2 \to Y$. 
% Since
% \[
%  \phi_1^!\cDiff_{\phi_2^{-1}A}(\phi_2^{-1}A,f^*\sL_2)\simeq 
%  \phi_1^{-1}\cDiff_{\phi_2^{-1}A}(\phi_2^{-1}A,f^*\sL_2)\ten_{\psi^{-1}A}f^*\phi_2^{-1}\det \oL\Omega^1_{X_1/Y}[-d], 
% \]
In order to do this, we reverse some of the equivalences above, using Remark \ref{extremecasesrmk} to replace differential operators with $R$-linear maps, giving quasi-isomorphisms
\begin{align*}
& \CCC_R(\psi^{-1}\sO_Y, (\phi_1^!\cDiff_{\phi_2^{-1}\sO_Y}(\phi_2^{-1}\sO_Y,\sL_2)\ten_{\phi_2^{-1}\sO_{X_1}} \phi_2^{-1}\sL_1^{-1}))\\
&\simeq %%\CCC_R(\psi^{-1}\sO_Y, \sHom_R(\psi^{-1}\sO_Y,\sL_2)\ten_{\psi^{-1}\sO_Y} \phi_2^{-1}\sL_1^{-1}))
\oR\hom_{\sO_{X_1\by^h_YX_2}}(\oL\phi_2^*\sL_1, \phi_1^!\sL_2)\\
&\simeq ( \oL\phi_2^*\sL_1^{-1} \ten^{\oL} \oL\phi_2^*\det \oL\Omega^1_{X_1/Y}\ten^{\oL}  \oL\phi_1^*\sL_2)[-d],
\end{align*}
so shifting by $d$ indeed gives us a deformation quantisation in
\[
 Q\cP(\oL\phi_2^*(\sL_1^{-1} \ten^{\oL} \det \oL\Omega^1_{X_1/Y})\ten^{\oL}  \oL\phi_1^*\sL_2,-1)
\]
with the desired global sections.
\end{proof}

%%remark here saying that it might be tempting to think of this caty as linearisation of some quantn of $(\infty,1)$-caty of Lagrangian correspondences. However, the shift ruins that.

We also have functoriality of the dg categories $Q\C\cI_{\dg}(\tilde{\sO}_Y)$  of quantised co-isotropic structures with respect to co-isotropic correspondences:

\begin{proposition}\label{Lagcorrprop}
 Assume we are given  quantisations $\tilde{\O}_{Y} \in Q\cP(Y,0)$ and  $\tilde{\O}_{Z}\in Q\cP(Y,0) $ of derived Artin $N$-stacks $Y,Z$, a  morphism $\psi \co T \to Y \by Z$ for  which $\psi_2 \co T \to Z$ is representable by quasi-compact quasi-separated derived algebraic spaces,  and a line bundle $\sM$ on $T$ with a quantised co-isotropic structure
$( \tilde{\O}_{Y}^{\op}\ten\tilde{\O}_{Z},\tilde{\sM}) \in  Q\cP(Y\by Z,\sM;0)$ lifting the quantisation $\tilde{\O}_{Y}^{\op}\ten\tilde{\O}_{Z}$ of $Y \by Z$.

Then there is a natural dg functor $Q\C\cI_{\dg}(\tilde{\sO}_Y) \to Q\C\cI_{\dg}(\tilde{\sO}_Z)$ between  the respective dg categories of quantised co-isotropic structures, which modulo $\hbar$ is quasi-equivalent  to the  dg  functor 
\[
 (X \xra{\phi} Y, \sL) \mapsto (X\by^h_YT, \oL\pr_2^*\sM\ten^{\oL} \oL\pr_1^*\sL ).  
\]
\end{proposition}
\begin{proof}
 On objects, the functor is given by applying the derived tensor product construction of Proposition \ref{tenQcorrprop}. In the Deligne--Mumford setting, that means we send a quantised co-isotropic structure $\tilde{\sL}$ on $X$ to the quantised line bundle $  (\pr_2^{-1}\tilde{\sM})\ten_{\chi^{-1}\tilde{\sO}_Y}^{\oL}(\pr_1^{-1}\tilde{\sL})$ on $X\by^h_YT$, for the natural map $\chi \co X\by^h_YT \to Y$. 
 
In order to consider morphisms,  observe that we can rewrite $\oR \chi_*( \oL\pr_2^*\sM\ten^{\oL} \oL\pr_1^*\sL)$ as $\oR \psi_{1*}(\sM \ten^{\oL} \oR \pr_{2*} \pr_1^*\sL) \simeq \oR \psi_{1*}(\sM \ten^{\oL}  \oL \psi_1^*\oR \phi_{*}\sL)$, 
%%first step $\oR \pr_{2*}( \oL\pr_2^*\sM\ten^{\oL} \oL\pr_1^*\sL) \simeq \sM \ten^{\oL} \oR \pr_{2*} \pr_1^*\sL$. 
where $\psi_1 \co T \to Y$ composes $\psi$ with the projection $Y \by Z \to Y$. That description allows us to substitute  into Definition \ref{QCIdgdefglobal} to pass from morphisms in  $Q\C\cI_{\dg}(\tilde{\sO}_Y)$ (defined in terms of $\oR \phi_*\sL$)  to morphisms in  $Q\C\cI_{\dg}(\tilde{\sO}_Z)$ (defined in terms of $\oR \chi_*( \oL\pr_2^*\sM\ten^{\oL} \oL\pr_1^*\sL)$).
\end{proof}

\begin{remark}
It is natural to ask whether there are conditions under which Proposition \ref{Lagcorrprop} restricts to give a dg functor $\cF(\tilde{\O}_{Y}) \to \cF(\tilde{\O}_{Z})$. A necessary condition is that the co-isotropic structure on $\psi \co T \to Y \by Z$ must be Lagrangian, since the correspondence must send quantised Lagrangians to quantised Lagrangians. Additional conditions will be required to ensure that the correspondence preserves self-duality. It seems plausible that self-duality of the given quantisation of $T$ suffices, but it is not clear that  self-duality interacts well with additivity statements such as Propositions \ref{tenQprop} and \ref{tenQcorrprop} (and indeed Proposition \ref{HomQprop}), although it seems likely. 
\end{remark}

\subsection{Uniqueness of quantisations for Lagrangians} 

The Fukaya category envisaged in \cite[\S 5.3]{BehrendFantechiIntersections} had an object for each  local system on a Lagrangian submanifold $L$. By contrast, the dg category conjectured in \cite[Remark 6.15]{BBDJS} only had one object for each square root of $K_L$. Our approach in Definition \ref{fukayadef} has an object for each self-dual quantisation of a square root of the dualising complex, making it  closest in flavour to the dg category of simple DQ modules supported on smooth Lagrangians constructed using \cite{kashiwaraschapira, DAgnoloSchapira,kashiwaraschapiraConstructibility}   and also described explicitly in \cite[Remark 6.15]{BBDJS}. 

While Corollary \ref{quantcorsd} ensured that  self-dual quantisations of square roots of the dualising complex always exist,  we now investigate how unique they are.

 Once we have fixed our quantisation $\tilde{\O}_{Y}$ in $Q\cP(Y,0)^{\nondeg,sd}$ and a compatible Lagrangian $(\omega,\lambda) \in \Lag(Y,X;0)$, the  homotopy fibre of 
\[
Q\cP(Y,\sL;0)^{ \nondeg,sd}\to \Lag(Y,X;0)\by^h_{\Sp(Y,0)}Q\cP(Y,0)^{\nondeg,sd}
\]
over $(\tilde{\O}_{Y}, \lambda)$ parametrises self-dual $\tilde{\O}_{Y}$-module quantisations of the line bundle $\sL$ on the Lagrangian $(X,\lambda)$, where $\Sp(Y,0)=\Lag(Y,\emptyset;0)$, the space of $0$-shifted symplectic structures on $Y$. We now explain how this homotopy fibre can be regarded as a torsor for the group of  self-dual rank $1$ local systems, so comes close to the intention of \cite{BehrendFantechiIntersections}.

By Theorem \ref{quantpropsd}, components  of the homotopy fibre are a torsor for the even de Rham power series
% \[
%  \H^1(F^2\DR(X/R))^{\nondeg} \by \prod_{i>0}  \H^1\DR(X/R)\hbar^{2i}, %htpy fibre kills first term
% \]
$\hbar^2\H^1\DR(X/R)\brhh$, 
although the parametrisation depends on an even associator $w \in \Levi_{\GT}^t$.

As in \cite[Remark \ref{DQvanish-rightDmodrmk}]{DQvanish}, quantisations $(\sL\brh, \delta + \Delta)$   of $\sL$ give rise to deformations  $\sE_{\hbar}:=(\sL\ten_{\sO_{X}}^{\oL}\sD_{X}\llbracket \hbar \rrbracket, \delta +\Delta\cdot\{-\})$ of $\sL\ten_{\sO_{X}}^{\oL}\sD_{X}$ as a right $\sD_{X}$-module. Other deformations of this form can be obtained by tensoring with deformations $\O'_{\hbar}$ of $\sO_{X}\brh $ as a left $\sD_{X}$-module, and indeed $\mmc(\hbar F^1\DR(X)\brh)$ is the space of such deformations. When $\sL^{\ten 2} = K_{X}$, the self-duality condition for $\sE_{\hbar}$ is
\[
 \sE_{-\hbar} \simeq \oR\hom_{\sD_{X}^{\op}\llbracket\hbar\rrbracket}(\sE_{\hbar},\sD_{X}\llbracket\hbar\rrbracket)\ten_{\sO_{X}} K_{X}
\]
as right $\sD_{X}\brh$-modules. The condition for $\O'_{\hbar}\ten\sE_{-\hbar}$ to also be self-dual is then 
\[
 \O'_{-\hbar} \simeq \oR\hom_{\sO_{X}\brh}(\O'_{\hbar}, \sO_{X}\brh)
\]
as left $\sD_{X}$-modules.

Over $\Cx$, the associated analytic $\sD_{X}^{\an}\brh$-modules  on the analytic space $X(\Cx)$  correspond to rank $1$ $\Cx\brh$-linear local systems $\vv$ equipped with an  inner product $ \vv\by \vv \to \Cx\brh$ which is sesquilinear in the sense that 
\[
\<a(\hbar)u,b(\hbar)v\> =a(\hbar)b(-\hbar)\<u,v\>, 
\]
for $a(\hbar), b(\hbar) \in \Cx\brh$;  equivalently, these correspond  to  $\exp(\hbar \Cx\brhh)$-torsors. 

In general, we have a similar statement in terms of classical algebraic $\sD$-modules on any formally smooth formal thickening $W$ of the underived truncation $\pi^0X \subset X$. Specifically, we may restrict attention to $\sO_W\brh$-modules $\sO'_{W,\hbar}$ with flat connection $\nabla$   deforming $\sO_W$, equipped with an $\sO_W\brh$-sesquilinear  inner product $\sO'_{W,\hbar} \by \sO'_{W,\hbar} \to \sO_W\brh $ satisfying $d\<u,v\>=\<\nabla u,v\>+\<u,\nabla v\>$.  Explicitly, such $\sD_W\brh$-modules take the form $(\sO_W\brh, d+ \omega)$, for $\omega \in \hbar\z^1\DR(W)\brhh$, so correspond to $\exp(\hbar \DR(\sO_W)\brhh)$-torsors. 

 We now show that the parametrisation in terms of de Rham cohomology corresponds to  the homotopy fibre above  being a torsor for this group of $\sD\brh$-modules. 
   
 In the following proposition, we denote by $\hat{B}_0$ and $\widehat{\DR(B_0)}$ the completions of $B_0$ and $\DR(B_0)$ over $\H_0B$.
 
\begin{proposition}\label{quanttorsorprop}
 Assume that the base ring $R$ is discrete (i.e. $R \cong \H_0R$) and Noetherian, with $A$ and $B$  cofibrant stacky $R$-CDGAs such that $B$ is of finite type. Then over points at which  the obstruction of Corollary \ref{quantcorsd} vanishes,
 the homotopy fibres of
\[
Q\cP(A,M;0)^{ \nondeg,sd}\to \Lag(A,B;0)\by^h_{\Sp(A,0)}Q\cP(A,0)^{\nondeg,sd}
\]
are torsors for the  $2$-group consisting of those  semilinearly self-dual strict line bundles $E$ on $\hat{B}_0\brh$ with flat connection  which are trivial modulo $\hbar$, i.e.  $E\ten_{\hat{B}_0\brh}\hat{B}_0 \cong \hat{B}_0$. %%note that connection then automatically  trivial there, since self-dual
The action on a quantisation $(\Delta_A, \Delta_M)$ is given  by tensoring the corresponding $(A,\Delta_A)- \sD_B$-bimodule $(M\ten _B\sD_B, \delta + \Delta_M\cdot -)$ on the left with the left $\sD$-module $E$. %% $(A, \Delta_A)$ an algebroid 
\end{proposition}
\begin{proof}
The tensor product of a left $\sD$-module and a right $\sD$-module is again a right $\sD$-module. The action is easiest to describe when the self-dual left $\sD_B$-module takes form $(B\brh, d+ \omega)$ for genuinely closed $1$-forms $\omega \in \hbar\ker(d \co \Omega^1_{B_0} \to \Omega^2_{B_0})\brhh$, in which case the right $\sD_B$-action is transformed as follows. First, form the derivation $\omega \lrcorner$ of $\sD_B$ determined by the property that it is given on tangent vectors by contraction --- the condition $d\omega=0$ ensures that this is well defined. For example, if $\omega =db$ we then have $db\lrcorner (\theta) =[b,\theta]$. The right action of a tangent vector $v$ on the tensor product is then given by $v - \omega \lrcorner v$, from which we conclude that an arbitrary element $\theta \in \sD$ acts as $\exp(-\omega \lrcorner)(\theta)$ on the tensor product, for the ring automorphism $\exp(-\omega \lrcorner)$ of $\sD$ given by exponentiating the   derivation $-\omega \lrcorner$, which is locally nilpotent because it reduces the order. 

Thus the $(A,\Delta_A)- \sD_B$-bimodule $ (B\brh, d+ \omega)\ten_B(M\ten _B\sD_B,  \delta +\Delta_M\cdot -) $ is just $ (M\ten _B\sD_B,  \delta + \Delta_M\cdot -)$ again, but with right $\sD$-module structure twisted by the automorphism $\exp(-\omega \lrcorner) $. By applying the inverse of that automorphism to elements of $\sD$, we see that this is isomorphic to $ (M\ten _B\sD_B, \delta+ \exp(\omega \lrcorner)(\Delta_M)\cdot -)$ with its natural right $\sD$-module structure, for $\exp(\omega \lrcorner) $ the automorphism of $\sD_B(M)$ defined by the same procedure as on $\sD_B$.

Now observe that as in \cite[Proposition \ref{stacks2-fthm}]{stacks2}, completion along $B_0\to \H_0B$ gives quasi-isomorphisms on everything in sight, and in particular the map $\sD(M) \to \cDiff_B(M,M\ten_{B_0}\hat{B}_0)$ of DGAAs is a filtered  quasi-isomorphism, so we can use the latter to define quantisations.

On our stacky CDGA $B$, semilinearly self-dual strict line bundles $E$ as above take the form $\hat{B}_0\brh_c$ as in Definition \ref{bistrictlb}, for $c \in \hbar \ker(\pd \co \hat{B}^1_0 \to \hat{B}^2_0)\brhh$, with connections of the  form $d+ \omega$ for $ \omega \in \hbar\ker(d \co \Omega^1_{\hat{B}_0^0} \to \Omega^2_{\hat{B}_0^0})\brhh$ and $\pd \omega =dc \in (\Omega^1_{\hat{B}_0})^1$. On tensoring, the effect of this on a quantisation $\Delta $ is thus to send $\Delta$ to $c+  \exp(\omega \lrcorner)(\Delta)$. %%note that $\delta \omega=0$, and I think we need to keep $\pd$ separate anyway, so think of $\Delta$ as jsut being in chain direction.
Also note that replacing $(\omega,c)$ with $(\omega +db, \pd c)$ for $b\in\hbar \hat{B}^0_0\brhh$ has the same effect on this quantisation as conjugation by $\exp(b)\in \hat{B}^0_0\brh$.

For such line bundles, we can thus characterise this tensor action  as the inverse limit of a system of morphisms
\[
 \mmc(\hat{\Tot} \hbar \widehat{\DR(B_0)}[\hbar^2]/\hbar^{2i} )%\lefttorightarrow 
 %\lcirclearrowright
 \curvearrowright
 %\by \mmc(\tilde{F}^2Q\widehat{\Pol}(M;-1)^{sd}/G^{2i+1})\to
 \mmc(\tilde{F}^2Q\widehat{\Pol}(A,M;0)^{sd}/G^{2i+1})
\]
of Maurer--Cartan spaces. 
  At the $i$th level, we can consider the action of terms of the form $\hbar^{2i-1}(\omega,c)$, and we see that this sends a quantisation $\Delta$ of $\pi$ to $\Delta + \hbar^{2i-1}(c+ \omega \lrcorner \pi)$ modulo $G^{2i+1}$. In other words, the fibres of $\mmc(\hat{\Tot} \hbar \widehat{\DR(B_0)}[\hbar^2]/\hbar^{2i} )\to \mmc(\hat{\Tot} \hbar \widehat{\DR(B_0)}[\hbar^2]/\hbar^{2i-2} ) $ act on the fibres of $\mmc(\tilde{F}^2Q\widehat{\Pol}(A,M;0)^{sd}/G^{2i+1}) \to \mmc(\tilde{F}^2Q\widehat{\Pol}(A,M;0)^{sd}/G^{2i-1}) $ by addition via the unquantised compatibility map
\[
 \hbar^{-1}\mu(-,\pi)\co  \hbar^{2i-1}\widehat{\DR(B_0)}\to \hbar^{2i-1}T_{(\varpi,\pi)} \widehat{\Pol}(A,M;0).
\]

Since de Rham cohomology only detects reduced structure, we know that $\widehat{\DR(B_0)}$ is quasi-isomorphic to $\DR(B)$ (cf. \cite{FeiginTsygan}), and non-degeneracy of $\pi$ thus implies that 
the compatibility map
\[
 \mu(-,\pi)\co \widehat{\DR(B_0)} \to  \ker(T_{(\varpi,\pi)} \widehat{\Pol}(A,M;0) \to T_{\varpi} \widehat{\Pol}(A,0))
\]
is a quasi-isomorphism, as in the proof of Proposition \ref{QcompatP1}. It thus follows inductively that taking the action on  any chosen point of $Q\cP(A,M;0)^{ \nondeg,sd}$   gives a weak equivalence  from $\mmc(\hat{\Tot} \hbar \widehat{\DR(B_0)}\brhh )$ to the whole  homotopy fibre over $ \Lag(A,B;0)\by^h_{\Sp(A,0)}Q\cP(A,0)^{\nondeg,sd}$.
\end{proof}

\begin{remarks}\label{torsorsdrmks}
By applying descent arguments as in \S \ref{lbsn}, Proposition \ref{quanttorsorprop} gives a characterisation of
the homotopy fibres of
\[
Q\cP(Y,\sL;0)^{ \nondeg,sd}\to \Lag(Y,X;0)\by^h_{\Sp(X,0)}Q\cP(X,0)^{\nondeg,sd}
\]
in terms of line bundles with connection on any formally smooth formal thickening of the underived truncation $\pi^0X$ of the derived Artin stack $X$.

Over $\Cx$, Proposition \ref{quanttorsorprop} amounts to saying that the non-empty homotopy fibres are  torsors for the $2$-group of $\exp(\hbar \Cx\brhh)$-torsors on the  analytic site of $X(\Cx)$. This follows from the correspondences between torsors and local systems, and between local systems and  $\sD$-modules, with direct correspondence between algebraic and analytic $\sD$-modules because the torsors are unipotent. The action of $\exp(\hbar \Cx\brhh)$-torsors on quantisations admits a  much simpler explicit  description than the action of $\sD$-modules, with  $\hbar \Cx\brhh$ simply acting on $\tilde{F}^2Q\widehat{\Pol}(\sL^{\an},-1)$ by addition.
\end{remarks}

We now look at the analogue of Proposition \ref{quanttorsorprop} in the case where the quantisations are not self-dual. In order to obtain a similar statement, we need to allow the underlying line bundle to vary slightly, so the result is phrased in terms of the 
total space $\oR (Q\cP(X,Y;0)^{ \nondeg}/^h\bG_m)$ of quantised Lagrangians from Definition \ref{quotientbyGmdef}, whose homotopy fibres over line  bundles $\sL \in \map(X, B\bG_m)$ are the spaces $Q\cP(Y,\sL;0)^{ \nondeg}$ of quantisations of line bundles, as in Definition \ref{QPdefglobal}. Observe that since co-isotropic structures are independent of the choice of line bundle, the $\bG_m$-action on the  first cotruncation is trivial, so 
\[
 \oR((Q\cP(X,Y;0)/G^1)^{ \nondeg}/^h\bG_m) \simeq \cP(X,Y;0) \by \map(X, B\bG_m).
\]

For economy of notation, we write $D_*F(A):= \ho \Lim_{i\in \Delta} F(D^iA)$ for stacky CDGAs $A$.

\begin{proposition}\label{quanttorsorpropnonsd}
 Assume that the base ring $R$ is discrete (i.e. $R \simeq \H_0R$) and Noetherian, with $A$ and $B$ are cofibrant stacky $R$-CDGAs such that $B$ is of finite type. Then the non-empty 
  homotopy fibres of the map 
\begin{align*}
&D_*(\oR Q\cP(A,B;0)^{ \nondeg}/^h\bG_m)\\
&\to \map(\Spec D\H_0B, B\bG_m)\by \Lag(A,B;0)\by^h_{\Sp(A,0)}Q\cP(A,0)^{\nondeg}
\end{align*}
are torsors for the  $2$-group consisting of those   strict line bundles $E$ on $\hat{B}_0\brh$ with flat connection  which are trivial over $\H_0B$ i.e.  $E\ten_{\hat{B}_0\brh}\H_0B \cong \H_0B$.
The action on a quantisation $(\Delta_A, \Delta_M)$ is given  by tensoring the corresponding $(A,\Delta_A)- \sD_B$-bimodule $(M\ten _B\sD_B, \delta + \Delta_M\cdot -)$ on the left with the left $\sD$-module $E$. %% $(A, \Delta_A)$ an algebroid 
\end{proposition}
\begin{proof}
 This proceeds in a similar fashion to Proposition \ref{quanttorsorprop}, as follows. The key observation to make is that for a strict line bundle $M$ on $B$, the homotopy fibre of
 \[
  D_*(\oR Q\cP(A,B;0)/^h\bG_m) \to \map(\Spec D\H_0B, B\bG_m)
 \]
over $\{M\ten_B\H_0B\}$ is given by applying the Maurer--Cartan functor $\mmc$ to the pro-nilpotent DGLA
\[
 \tilde{F}^2 Q\widehat{\Pol}(A,M;0) \oplus \hat{\Tot}\ker(B\to \H_0B) \subset \tilde{F}^1Q\widehat{\Pol}(A,M;0).
\]
 %% to see this is closed: Braces automatically $0$ (drop a degree), and within $\C\C(A,\cD(M))$, we're OK whenever $\hbar$ appears (land in $\tilde{F}^2$), so only $(\gamma F)_0=B$ can cause trouble, and commutator there is $0$.
This follows because, via completion and exponentiation, $\mmc(\hat{\Tot}\ker(B\to \H_0B) )$ is canonically equivalent to the homotopy kernel of $ \map(\Spec DB, B\bG_m)\to  \map(\Spec D\H_0B, B\bG_m) $, and the tangent of the conjugation action of $\bG_m(B)$ on $\sD_B$  corresponds to taking the Lie bracket as differential operators. For a closely related construction, see the twisted quantisations of \cite[Definition \ref{DQvanish-Qpoissdef}]{DQvanish}.

The question thus reduces to understanding the morphism
\[
 \tilde{F}^2 Q\widehat{\Pol}(A,M;0) \oplus \hat{\Tot}\ker(B\to \H_0B) \to F^2\widehat{\Pol}(A,B;0)\by_{F^2\widehat{\Pol}(A,0)}\tilde{F}^2Q\widehat{\Pol}(A,0)
\]
of DGLAs, or rather of the associated Maurer--Cartan spaces.
 
 On our stacky CDGA $B$, since $\hat{B}_0 \to \H_0B$ is a pro-nilpotent extension,  strict line bundles $E$  as above take the form $\hat{B}_0\brh_c$, with $c \in \hat{B}^1_0\brh$ in the kernel of $\hat{B}^1_0\brh \to \H_0B$, such that $\pd c =0\in \hat{B}^2_0\brh$. Flat connections for such bundles  then take the form $d+ \omega$ for $ \omega \in \ker(d 
 \co \Omega^1_{\hat{B}_0^0} \to \Omega^2_{\hat{B}_0^0})\brhh$ with  $\pd \omega =dc\in (\Omega^1_{\hat{B}_0})^1$. 

 On tensoring, the effect of this on a twisted quantisation $\Delta $ is to send $\Delta$ to $c+  \exp(\omega \lrcorner)(\Delta)$, defined as in Proposition \ref{quanttorsorprop}. This action by the space of line bundles is thus giving us a system of group actions
 \[
 \mmc(\hat{\Tot} \ker(\widehat{\DR(B_0)}[\hbar]/\hbar^{i} \to \H_0B) )%\lefttorightarrow 
 %\lcirclearrowright
 \curvearrowright
 %\by \mmc(\tilde{F}^2Q\widehat{\Pol}(M;-1)^{sd}/G^{2i+1})\to
 \mmc(  (\hat{\Tot}\ker(B\to \H_0B)\oplus\tilde{F}^2Q\widehat{\Pol}(A,M;0))/G^{i}).
\]
On the fibres at the $i$th level, this map is just truncation of 
the unquantised compatibility map
\[
 \hbar^{-1}\mu(-,\pi)\co  \hbar^{i-1}\widehat{\DR(B_0)}\to \hbar^{i-1}T_{(\varpi,\pi)} \widehat{\Pol}(A,M;0)
\]
for $i>1$ and
\[
 \hbar^{-1}\mu(-,\pi)\co  \ker(\widehat{\DR(B_0)}\to \H_0B) \to \ker(T_{(\varpi,\pi)} \widehat{\Pol}(A,M;0)\to \H_0B)
\]
for $i=1$. The conclusion now follows exactly as for Proposition \ref{quanttorsorprop}.
\end{proof}

\begin{remark}\label{torsornonsdrmks}
By applying descent arguments as in \S \ref{lbsn}, Proposition \ref{quanttorsorpropnonsd} gives a characterisation of
the homotopy fibres of
\[
\oR(Q\cP(Y,\sL;0)^{ \nondeg}/^h\bG_m)\to \map(\pi^0X,B\bG_m) \by \Lag(Y,X;0)\by^h_{\Sp(X,0)}Q\cP(X,0)^{\nondeg}
\]
in terms of line bundles with connection on any formally smooth formal thickening of the underived truncation $\pi^0X$ of the derived Artin stack $X$.

%%Over $\Cx$, the analytic description from Remarks \ref{torsorsdrmks} has an analogue in this non-self-dual setting, replacing $\hbar\Cx\brhh$ with the tow-term complex $\Cx\brh \to \sO_{\pi^0X}^{\an}$. %%NOT TRUE: map only goes one way.
\end{remark}

\subsection{Morphisms in terms of vanishing cycles} 
The complex Fukaya category envisaged in \cite{joyceFukaya} had complexes of morphisms coming from shifts of the perverse sheaf of vanishing cycles, but the required composition law was purely conjectural. Our construction in Definition \ref{fukayadef} is manifestly a dg category, and we will now show that on inverting $\hbar$, its resulting $\Hom$-complexes indeed come from sheaves of vanishing cycles, so the $R((\hbar))$-linear dg category $\cF(\tilde{\O}_{Y})[\hbar^{-1}]$ has all the expected properties.  

\begin{corollary}\label{cfvanish2}
If  $\phi_1$ is proper and virtually LCI of relative dimension $d$, 
 then the complex  $\oR\hat{\HHom}_{\tilde{\sO}_Y}(\oR \phi_{1*}\tilde{\sL}_1,\oR \phi_{2*}\tilde{\sL}_2)$ from Definition \ref{QCIdgdefglobal}
 is given by derived global sections of a complex $\bH$ of  sheaves on the homotopy fibre product $X_1\by_{Y}^hX_2$. 
 
If $R$ is discrete and Noetherian, then on any \'etale neighbourhood $U$  of $X_1\by^h_{Y}X_2$ which is equivalent as a $(-1)$-shifted  symplectic stack to the derived critical locus of a function $f \co Z \to \bA^1$ on a smooth DM $\infty$-stack $Z$ over $R$, we have
 \[
\oR \Gamma(U, \bH) \simeq \oR\Gamma(Z, (\Omega^*_{\hat{Z}}\ten_{\sO_{\hat{Z}}}\sM, \hbar \nabla + df \wedge -))[\dim Z -d] 
 \]
for $\hat{Z}:=\hat{Z}_{\pi^0U}$ the formal completion of $Z$ along the critical locus $\pi^0U$ (the underived truncation of $U$) and $\sM$ some    rank $1$ $\sO_{\hat{Z}}\brh$-module whose restriction to $\pi^0U$ is trivial, equipped with a flat connection $\nabla$.
%%note that $[a, \pd_x\pd_{\xi}]= \frac{\pd a}{\pd x}\pd_{\xi}$, so $d+\omega \lrcorner d =\nabla$.

If $R=\Cx$, then for the  rank $1$ local system $\vv$ of $\Cx\brh$-modules on the analytic site of $U(\Cx)$  given by  horizontal sections of $\nabla$ in $\sM^{\an}$, we have
\[
 \oR \Gamma(U,\bH[\hbar^{-1}]) \simeq \oR\Gamma(U(\Cx)^{\an}, \bigoplus_{c \in \Cx} \phi_{f-c} [\dim Z -d-1]\ten_{\Cx} \vv[\hbar^{-1}]),
\]
for 
 $\phi$ the vanishing cycles  
complex. 
  \end{corollary}
\begin{proof}
By Proposition \ref{cfvanish1}, the complex $\bH[d]$ is given by  quantisation of a line bundle on the $(-1)$-shifted symplectic derived stack  $X_1\by_{Y}^hX_2$. 

On the neighbourhood $U$, we know from \cite[Lemma \ref{DQvanish-PVlemma}]{DQvanish} that the  twisted de Rham complex 
\[
(\Omega^*_{Z} , \hbar d + df \wedge -)[\dim Z]
\]
 is an element of $Q\cP(\Omega^d_Z\ten_{\sO_Z}\sO_U ,-1)^{sd}_{\lambda_f}$ for the canonical $(-1)$-shifted symplectic structure $\lambda_f$.

Proposition \ref{quanttorsorpropnonsd} applied to the $0$-shifted Lagrangian $ U \to \Spec R$   thus implies that on $U$, our quantisation is given by the tensor product of the twisted de Rham complex and a  rank $1$ $\sO_{\hat{Z}}\brh$-module $\sM$ with flat connection $\nabla$ as above, giving rise to the first desired expression by applying  the formulae in the proof of Proposition \ref{quanttorsorprop} to the second order differential operator $\hbar d$. 

On inverting $\hbar$, the twisted de Rham complex becomes a vanishing cycles complex, as in \cite[Theorem 1.1]{SabbahTwistedII} (see also \cite[Proposition \ref{DQvanish-PVprop}]{DQvanish}), yielding the second expression.
 \end{proof}

\begin{remarks}
When  $X_1\by^h_{Y}X_2$ is a derived DM $\infty$-stack, note that since it is $(-1)$-shifted symplectic, the 
results of \cite{BBBJdarboux,BouazizGrojnowski} imply that it is covered by \'etale neighbourhoods of the form $U$ in Corollary \ref{cfvanish2}. 

As in Remark \ref{compattenQrmk}, it seems reasonable to expect that the quantisations of Proposition \ref{cfvanish1} are self-dual, in which case a stronger statement than Corollary \ref{cfvanish2} would hold, using Proposition \ref{quanttorsorprop} in place of Proposition \ref{quanttorsorpropnonsd} to conclude that  the $\sD\brh$-module $(\sM,\nabla)$ must be semi-linearly self-dual with respect to the involution $\hbar \mapsto -\hbar$. 
%thus $\sD$-mod exists algebraically.

Over $\Cx$, one way to interpret this is that on inverting $\hbar$, non-degenerate self-dual $(-1)$-shifted quantisations give  a form of perverse sheaf over the $*$-algebra (i.e. ring with involution) $\Cx((\hbar))$, whereas \cite{BBDJS} constructed perverse sheaves of vanishing cycles over rings such as $\R$ or $\Cx$. The choice of orientation in \cite{BBDJS} could be regarded as a torsor for the group $\{\pm 1\} = \{a \in \Cx^{\by}~:~ a=a^{-1}\}$, while the self-dual quantisations of Proposition \ref{quanttorsorprop} depend on a choice of torsor for the group $\pm \exp(\hbar\Cx\brh)= \{a \in \Cx((\hbar))^{\by}~:~ a(-\hbar)=a(\hbar)^{-1}\}$.

The vanishing cycles sheaf  from \cite{BBDJS} was constructed by discarding much of the derived structure, while our constructions here and in \cite{DQvanish} depend on the full derived structure. We expect that the $R\brh$-linear dg category $\cF(\tilde{\O}_{Y})$ depends on the derived structure in an essential way, and that the same is probably true of any variant such as that envisaged in \cite{joyceFukaya}. 
\end{remarks}

\bibliographystyle{alphanum}
\bibliography{references.bib}
\end{document}